\documentclass[times,final]{elsarticle}

\usepackage{jcomp}

\usepackage{framed,multirow}

\usepackage{amssymb,amsmath,amsthm,bm}
\usepackage{latexsym}
\usepackage{mathtools}

\usepackage{cleveref}
\usepackage{subcaption}
\usepackage{graphicx}

\usepackage{booktabs}
\usepackage{makecell}
\usepackage{enumitem} 
\usepackage{bm}
\usepackage{url}
\usepackage{xcolor}

\usepackage{tikz}
\usetikzlibrary{arrows.meta, positioning, calc}
\usepackage{caption}
\usepackage{ulem}
\definecolor{newcolor}{rgb}{.8,.349,.1}
\usepackage{booktabs} 
\usepackage{siunitx} 
\usepackage{subcaption}
\makeatletter

\newcommand{\Rmnum}[1]{\uppercase\expandafter{\romannumeral #1}}
\makeatother

\newtheorem{theorem}{Theorem}[section]

\theoremstyle{definition}

\newtheorem{expl}{Example}[section]

\theoremstyle{remark}
\newtheorem{remark}{Remark}[]

\newcommand{\dd}{\mathrm{d}}

\newcommand{\halfone}{\frac{1}{2}}

\def\average#1{\lbrace\!\lbrace #1  \rbrace\!\rbrace }

\allowdisplaybreaks

\usepackage{tikz}
\usetikzlibrary{shapes.geometric, arrows.meta, positioning}

\tikzstyle{startstop} = [ellipse, rounded corners, minimum width=0.5cm, minimum height=0.5cm, text centered, draw=black, fill=orange!20]
\tikzstyle{process} = [rectangle, minimum width=2cm, minimum height=0.7cm, text centered, draw=blue!50, fill=blue!20]
\tikzstyle{decision} = [diamond, minimum width=1cm, minimum height=1cm, text centered, draw=black, fill=gray!30]
\tikzstyle{arrow} = [thick,->,>=stealth]

\usepackage{lmodern} 
\usetikzlibrary{shapes.misc} 
\usepackage{anyfontsize}

\begin{document}

	\begin{frontmatter}
		
		\title{{\bf An active-flux-type scheme for ideal MHD with provable positivity and discrete divergence-free property}}
		
		\tnotetext[tnote1]{The work of K.~Wu,  M.~Liu, and D.~Pang  was partially supported by Shenzhen Science and Technology Program (Grant No.~RCJC20221008092757098), National Natural Science Foundation of China (Grant No.~12171227), and the fund of the Guangdong Provincial Key Laboratory of Computational Science and Material Design (No.~2019B030301001). 
			The work of M.~Liu was also partially supported by China Postdoctoral Science Foundation under Grant Number 2024M761270. 
        The work of R.~Abgrall was partially supported by
        SNSF grant 200020$\_$204917. 
        }
		
		\author[1,2]{Mengqing {Liu}}
		\ead{liumq@sustech.edu.cn}
		\author[1]{Dongwen {Pang}}
		\ead{12331010@mail.sustech.edu.cn}
		\author[3]{R\'{e}mi {Abgrall}}
		\ead{remi.abgrall@math.uzh.ch}
		\author[1,2,4]{Kailiang {Wu}\corref{cor1}}
		\cortext[cor1]{Corresponding author.}
		\ead{wukl@sustech.edu.cn}
		
		\address[1]{Department of Mathematics, Southern University of Science and Technology, Shenzhen 518055, China}
		\address[2]{Shenzhen International Center for Mathematics, Southern University of Science and Technology, Shenzhen 518055, China}
		\address[3]{Institute of Mathematics, University of Z\"urich 
			8057 Z\"urich, Switzerland }
		\address[4]{Guangdong Provincial Key Laboratory of Computational Science and Material Design, Southern University of Science and Technology, Shenzhen 518055, China}
		
\begin{abstract}
	The PAMPA (Point-Average-Moment PolynomiAl-interpreted) method, proposed in [R.~Abgrall, {\it Commun.~Appl.~Math.~Comput.}, 5:370--402, 2023], is a compact numerical framework that combines conservative and nonconservative formulations of hyperbolic conservation laws.
	In this paper, we present a novel {\it positivity-preserving} (PP) PAMPA scheme that enforces a {\it discrete divergence-free} (DDF) magnetic field for the ideal magnetohydrodynamics (MHD) equations on Cartesian grids. Building on our recent one-dimensional invariant-domain-preserving (IDP) PAMPA framework [R.~Abgrall, M.~Jiao, Y.~Liu, and K.~Wu, {\it SIAM J.~Sci.~Comput.}, to appear], we extend the methodology to the multidimensional, multiwave MHD system. The proposed scheme features a {\it limiter-free} PP update of interface point values via a new nonconservative reformulation, together with a local DDF projection that maintains a DDF magnetic field. The cell-average update is {\it provably} PP under a mild apriori positivity condition on a single cell-centered value and combines four ingredients: (i) a DDF constraint at interface point values, (ii) a PP limiter applied only to the cell-centered value, (iii) a PP numerical flux with properly estimated wave speeds, and (iv) a suitable discretization of the Godunov--Powell source term. The PP proof for cell averages is conducted within the geometric quasi-linearization (GQL) framework [K.~Wu  and C.-W.~Shu, {\it SIAM Review}, 65:1031--1073, 2023], which transforms the nonlinear pressure-positivity constraint into an equivalent linear form. The resulting scheme avoids explicit polynomial reconstructions, is compatible with arbitrarily high-order strong-stability-preserving time discretizations, and is straightforward to implement. To enhance robustness and resolution, we introduce a problem-independent troubled-cell indicator based on a Lax-type entropy criterion that uses only two characteristic speeds computed from cell averages and preserves symmetry, as well as a {\it convex oscillation elimination} (COE) mechanism with a new norm measuring solution differences between target and neighboring cells. Extensive tests---including a blast wave with plasma beta as low as $2.51\times10^{-6}$ and a jet with Mach numbers up to $10^{4}$---demonstrate high-order accuracy, sharp resolution of complex MHD structures, strong-shock robustness, and reliable troubled-cell detection. To the best of our knowledge, this is the first active-flux-type method for ideal MHD that is rigorously PP for both cell averages and interface point values while maintaining a DDF constraint throughout the evolution.
\end{abstract}

		\begin{keyword}
			\vspace{-11mm}
		\end{keyword}

	\end{frontmatter}

	\section{Introduction}
	The active flux (AF) method is a compact, high-order finite volume scheme, originally proposed in~\cite{Eymann2011active, Eymann2012active}. Unlike Godunov-type methods, the AF method simultaneously evolves both cell averages and point values located at cell interfaces. The cell averages are updated conservatively using numerical fluxes computed directly from point values, thereby eliminating the need for Riemann solvers.
	Two main strategies have been developed for updating the point values. The first employs exact or approximate evolution operators~\cite{Eymann2013multi, Wasilij2019active, Helzel2019ADER}, which solve local initial value problems over a time step. The second approach, known as the semi-discrete AF method~\cite{abgrall2023active11271, Remi2024newapp, AbgrallAF2025}, first discretizes the spatial derivatives to obtain a semi-discrete system, which is then integrated in time using standard strong-stability-preserving Runge--Kutta (SSP–RK) schemes.

 The Point-Average-Moment PolynomiAl-interpreted (PAMPA) method~\cite{Abgrall2023Combination, Abgrall2023Extensions, Abgrall2023hybrid} is a novel variant of the AF method. By incorporating higher-order moments and avoiding the need for exact or approximate evolution operators—which are often difficult to construct for nonlinear or multidimensional systems—PAMPA attains higher than third-order accuracy within a purely semi-discrete framework. It has been successfully applied to a range of problems, including the Euler equations on Cartesian grids~\cite{AbgrallAF2025}, compressible flows on triangular meshes~\cite{abgrall2023active11271}, hyperbolic balance laws~\cite{Remi2024newapp, liu2025wellbalance, liu2025arbitrarily}, and generalized high-order formulations~\cite{barsukow2025generalized}.
 Another notable strength of the PAMPA method lies in its versatility. Unlike the traditional AF method, the evolution of point values is not restricted to conservative formulations; instead, it permits customized updates based on arbitrary  non-conservative formulations. This added flexibility makes PAMPA particularly suitable for complex hyperbolic systems.

The ideal magnetohydrodynamic (MHD) equations serve as a fundamental model for describing the dynamics of inviscid, perfectly conducting plasmas. These equations play a central role in modeling a broad range of phenomena in astrophysics, space physics, and plasma physics, such as solar flares, coronal mass ejections, and magnetospheric dynamics.
Simulating the ideal MHD system presents substantial challenges. Due to its hyperbolic and nonlinear nature, even smooth initial data may evolve into discontinuities such as shock waves. In addition to accurately capturing these complex wave phenomena, it is also crucial to preserve the system’s intrinsic physical structures during numerical simulation. Failure to preserve these properties can lead to spurious oscillations, violation of physical bounds, or even simulation breakdown.

One fundamental property is the divergence-free (DF) constraint on the magnetic field, reflecting the physical law of the nonexistence of magnetic monopoles. Even slight violations of this condition can lead to nonphysical artifacts or severe numerical instabilities~\cite{Balsara1999A, Brackbill1980The, Li2005, Toth2000}. Over the past decades, various strategies have been proposed to control divergence errors in simulations of MHD, including projection methods based on Hodge decomposition~\cite{Brackbill1980The}, the constrained transport (CT) method~\cite{helzel2011unstaggered, Evans1988Simulation} and its variants~\cite{Londrillo2004divergence, GARDINER2005509, Christlieb2015PP}, and divergence-cleaning techniques involving auxiliary hyperbolic–parabolic equations~\cite{Dedner2002Hyperbolic}. Additional approaches include locally and globally divergence-free schemes within the discontinuous Galerkin and finite volume frameworks~\cite{Li2005, BALSARA20095040, Li2011, Fu2018, Balsara2021Globally}. Another important method is the eight-wave formulation~\cite{Powell1994, Powell1995AnUS}, which augments the MHD system with non-conservative source terms proportional to the magnetic divergence. These terms restore symmetric hyperbolicity, ensuring well-posedness even in the presence of small divergence errors, and allow such errors to be advected with the flow rather than accumulating locally.

In addition to the DF constraint, physical solutions of the MHD system must also satisfy essential algebraic constraints—specifically, the positivity of density and internal energy—which are critical for maintaining both the system’s hyperbolicity and its physical admissibility.
However, preserving these positivity conditions in numerical simulations remained a longstanding challenge, particularly in the presence of strong shocks, high Mach numbers, low internal energy, low density, or strong magnetic fields. Violation of these algebraic constraints may cause the system to lose hyperbolicity, leading to severe numerical instabilities and even simulation breakdown. Initial efforts to address this challenge employed robust one-dimensional Riemann solvers~\cite{Bouchut2007, Bouchut2010A, JANHUNEN2000A}. These were later followed by second-order MUSCL–Hancock schemes~\cite{Waagan2009, WAAGAN20113331}, which achieve positivity through the use of a relaxation Riemann solver \cite{Bouchut2007, Bouchut2010A}. More recently, a variety of high-order positivity-preserving (PP) limiters have been proposed by adapting techniques developed for the compressible Euler equations~\cite{Xiong2016parametrized, ZHANG20108918} to ideal  MHD~\cite{BALSARA2012Self, CHENG2013255,  Christlieb2015PP, CHRISTLIEB2016218}. However, the majority of these methods have been validated primarily through numerical experiments, with rigorous multidimensional proofs of their positivity preservation largely absent.

Recent theoretical studies have revealed a deep connection between the PP property and the DF condition. Although the PP property is an algebraic, pointwise constraint and the DF condition is a differential one, they are in fact strongly linked~\cite{Wu2018SINUpositivity,Wu2023provab}. In particular, it has been rigorously proven that, for conservative MHD schemes, enforcing a discrete DF (DDF) condition is essential to fully maintain positivity~\cite{Wu2018SINUpositivity}. Even small violations of the DF constraint can lead to negative pressures, not only in numerical solutions~\cite{Wu2018SINUpositivity} but even in exact smooth solutions of the standard MHD equations~\cite{WuShu2018}.
This strong coupling poses a significant challenge for constructing high-order PP schemes, especially in multiple dimensions where simultaneously enforcing both properties is highly nontrivial. An effective remedy is the modified MHD system with Godunov--Powell source terms~\cite{Powell1994, Powell1995AnUS}. These terms keep the system well-posed even in the presence of small divergence errors and weakly decouple the PP property from the globally DF constraint, allowing it to be enforced via a locally DF condition compatible with standard local scaling PP limiters \cite{WuShu2018,WuShu2019}. Building on this symmetric MHD formulation, provably PP and locally DF discontinuous Galerkin methods were proposed in \cite{WuShu2018,Wu2023provab} on Cartesian meshes, and a general framework of constructing provably PP finite volume and discontinuous Galerkin schemes on general meshes was established in \cite{WuShu2019}.
More recently, building on these works,
\cite{DingWu2024SISCMHD} developed a novel discrete locally DF projection to ensure provably PP finite volume WENO methods, and \cite{liu2025structure} proposed a locally DF oscillation-eliminating structure-preserving discontinuous Galerkin scheme for ideal MHD based on \cite{peng2023oedg}. Similar findings and provably structure-preserving schemes were also provided for the relativistic version of compressible MHD equations \cite{WuTangM3AS,Wu2021RMHD}.

This paper aims to develop PP and DDF PAMPA schemes for the ideal MHD system.
Compared with standard finite-volume or discontinuous Galerkin methods, the development of structure-preserving PAMPA schemes is still in its infancy, and several strategies have been explored in recent years.
The work in \cite{Chudzik2021125501} modified the original active-flux framework by adding a limiter to interface point values, thereby ensuring bound preservation for linear advection.
Another study \cite{Abgrall2023Combination} embedded the MOOD procedure into PAMPA schemes to maintain positive density and pressure for the Euler equations.
In \cite{Duan2025active}, the authors proposed bound-preserving schemes for both cell averages and point values, using convex limiting and scaling limiting together with multiple flux–vector–splitting techniques designed to handle transonic nonlinear flows.
A different approach was taken in \cite{abgrall2024boundpreserving}, where the residuals of a high-order scheme and a low-order local Lax–Friedrichs-type scheme were blended, yielding bound preservation for both cell averages and point values; this strategy has also been extended to other hyperbolic systems such as the Euler equations.
In the context of one-dimensional conservation laws, \cite{abgrall2024novel} developed an invariant-domain-preserving (IDP) PAMPA formulation that couples a provably IDP finite-volume update for cell averages with an automatic reformulation for point values, eliminating the need for additional limiters.
Very recently, \cite{duan2025AFMHD} presented a third-order positivity-preserving active-flux scheme for ideal MHD, which blends high-order AF discretization with a standard finite-volume method, together with a parametrized flux limiter for cell averages and a scaling limiter for interface point values.
However, the DF constraint of magnetic field was not enforced in the scheme of \cite{duan2025AFMHD}.

The contribution of this paper is to present a new PAMPA scheme that is provably PP and enforces a DDF magnetic field for the ideal MHD equations on Cartesian grids, extending our recent one-dimensional (1D) IDP PAMPA framework~\cite{abgrall2024novel} to a highly complex multidimensional hyperbolic system.
The main contributions are:
\begin{itemize}
	\item \textbf{Rigorous positivity for both cell averages and point values.}
	The scheme {\it automatically} ensures positivity at interface point values without any limiters; moreover, the magnetic field at these points satisfies a DDF constraint.
	For cell averages, the update is also provably PP under a mild a~priori positivity condition on a single cell-centered value, which is the only quantity requiring a PP limiter.
	To the best of our knowledge, this is the first active-flux-type approach for MHD that guarantees positivity for both cell averages and interface point values while maintaining a DDF constraint throughout the evolution.
	
	\item \textbf{Limiter-free, automatically PP update of point values.}
	Leveraging the flexibility of the PAMPA framework to use nonconservative formulations for point-value updates, we design specialized MHD reformulations that render the point-value update {\it automatically PP}, without any PP limiters.
	A highly efficient local projection enforces a DDF condition on the magnetic field at point values, which in turn is critical for ensuring positivity in the subsequent cell-average update.
	The discretization avoids explicit polynomial reconstructions, yielding an efficient and implementation-friendly method.
	
	\item \textbf{Provably PP update of cell averages.}
	The updated cell averages are guaranteed to remain positive by combining four ingredients:
	(i) a DDF constraint at interface point values;
	(ii) a PP limiter applied only to the single cell-centered value to enforce the apriori positivity condition;
	(iii) a PP numerical flux with properly estimated wave speeds; and
	(iv) a suitable discretization of the Godunov–Powell source term.
	Under (i)–(iv), we provide a rigorous PP analysis for the cell-average update within the geometric quasi-linearization (GQL)  framework \cite{Wu2018Positivity, Wu2023Geometric}.
	
	\item \textbf{A novel, problem-independent troubled-cell indicator.}
	Inspired by \cite{LiuFeng2021}, we introduce a Lax-type entropy criterion for detecting discontinuities.
	The indicator uses only two characteristic wave speeds computed from cell averages, making it both efficient and robust, and it preserves symmetry properties—an attribute not guaranteed by several alternative indicators~\cite{LiuFeng2021, DEEPRAY2019}.
	
	\item \textbf{Convex Oscillation Elimination (COE) technique.}
	To suppress spurious oscillations while retaining high-order accuracy in smooth regions, we propose a COE procedure that forms convex combinations of limited point values and cell averages.
	A key ingredient is a new norm that measures discrepancies between neighboring polynomials and effectively identifies oscillatory behavior.
	While conceptually distinct from the modal-filtering OE method of~\cite{peng2023oedg}, the proposed COE retains two essential structural properties—scale invariance and evolution invariance—ensuring consistency and robustness across a broad range of problem scales and wave speeds.
	The procedure is non-intrusive and can be readily extended to other systems.
	
	\item \textbf{Extensive numerical validation.}
	We validate the proposed PP PAMPA schemes on a comprehensive suite of benchmark and challenging ideal-MHD problems, including a blast wave with plasma beta as low as $2.51\times10^{-6}$ and a jet with Mach numbers up to $10^{4}$.
	The results confirm high-order accuracy, sharp resolution of complex flow and magnetic structures, robustness under extremely strong shocks, and reliable troubled-cell detection, highlighting both the effectiveness and broad applicability of the scheme.
\end{itemize}

	The remainder of this paper is organized as follows. In \Cref{sec:govequ}, we outline the ideal MHD system, its physical constraints, and the Godunov–Powell formulation. In \Cref{S.2}, we present the construction of the proposed PP PAMPA scheme for the MHD system. We begin by presenting the key components of the method, followed by a detailed description of schemes for point values and cell averages, along with a rigorous PP analysis. The Lax-type entropy criterion and the COE procedure are discussed in \Cref{trouble-cell} and \Cref{OE}, respectively. In \Cref{S.5}, numerical examples are presented to demonstrate the accuracy, robustness, and resolution of the scheme. Finally, \Cref{conclusion} summarizes the main results of this work.

	\section{Governing Equations}\label{sec:govequ}
	Without loss of generality, we focus on the two-dimensional (2D) ideal magnetohydrodynamics (MHD) system. However, the numerical scheme and analysis presented in this paper naturally extend to the three-dimensional case. 
	In 2D, the governing equations of ideal MHD can be expressed as a hyperbolic system of conservation laws:
	\begin{equation}\label{eq:MHDadd2D}
		\partial_{t}	\bm{U}+\nabla \cdot \bm{F} (\bm{U})= \bm{0},
	\end{equation}
	where $\bm{U}$ denotes the vector of conservative variables, and $\bm{F}(\bm{U}) = (\bm{F}_1(\bm{U}), \bm{F}_2(\bm{U}))$ represents the fluxes in the $x$- and $y$-directions, respectively:
	\begin{equation*}
		\bm{U}=\left(\begin{array}{c}
			\rho \\
			m_1 \\
			m_2 \\
			m_3 \\
			B_1 \\
			B_2 \\
			B_3 \\
			E \\
		\end{array}\right), \quad
		\bm{F}_1(\bm{U})=
		\left(\begin{array}{c}
			m_1 \\
			m_1 v_1 - B_1^2 + p_{tot} \\
			m_1 v_2 - B_1 B_2 \\
			m_1 v_3 - B_1 B_3 \\
			0 \\
			v_1 B_2 - v_2 B_1 \\
			v_1 B_3 - v_3 B_1 \\
			v_1(E + p_{tot}) - B_1(\bm{v} \cdot \bm{B}) \\
		\end{array}\right), \quad
		\bm{F}_2(\bm{U})=
		\left(\begin{array}{c}
			m_2 \\
			m_2 v_1 - B_2 B_1 \\
			m_2 v_2 - B_2^2 + p_{tot} \\
			m_2 v_3 - B_2 B_3 \\
			v_2 B_1 - v_1 B_2 \\
			0 \\
			v_2 B_3 - v_3 B_2 \\
			v_2(E + p_{tot}) - B_2(\bm{v} \cdot \bm{B}) \\
		\end{array}\right).
	\end{equation*}
	Here, $\rho$ is the density, $p$ is the thermal pressure, $\bm{m}=(m_1, m_2, m_3)$ is the momentum vector, and $\bm{v}=(v_1, v_2, v_3)=\bm{m}/\rho$ is the velocity field. The magnetic field is given by $\bm{B}=(B_1, B_2, B_3)$, and the total pressure is defined as $p_{tot} = p + \frac{1}{2}|\bm{B}|^2$. The total energy $E$ satisfies
	\[
	E = \frac{p}{\gamma - 1} + \frac{1}{2} \rho |\bm{v}|^2 + \frac{1}{2} |\bm{B}|^2,
	\]
	where $\gamma$ is the adiabatic index (ratio of specific heats).
	
	The physical admissibility of a solution to \eqref{eq:MHDadd2D} requires the density and internal energy (or, equivalently, the pressure) to be positive. That is,
	\begin{equation}\label{eq:ppcondi}
		\bm{U} \in \mathcal{G} := \left\{ \bm{U} = (\rho, \bm{m}, \bm{B}, E)^{\top} : \rho > 0, ~ \mathcal{E}(\bm{U}) := E - \frac{|\bm{m}|^2}{2\rho} - \frac{|\bm{B}|^2}{2} > 0 \right\},
	\end{equation}
	where $\mathcal{E}(\bm{U})$ denotes the internal energy, and $\mathcal{G}$ is referred to as the admissible state set. Beyond its physical significance, the condition $\bm{U} \in \mathcal{G}$ is also essential for the hyperbolicity of the MHD system \eqref{eq:MHDadd2D}; failure to satisfy it may lead to imaginary sound speeds and eigenvalues, rendering the system ill-posed. Therefore, it is crucial for any numerical MHD scheme to maintain solutions within $\mathcal{G}$.
	
	In addition, the magnetic field must satisfy the divergence-free constraint:
	\begin{equation}\label{eq:divB}
		\nabla \cdot \bm{B} = 0.
	\end{equation}
	While the continuous form of \eqref{eq:MHDadd2D} ensures that this condition remains satisfied for all time if it holds initially, standard numerical discretizations do not automatically preserve \eqref{eq:divB}, leading to non-zero divergence errors. Furthermore, the enforcement of the positivity condition \eqref{eq:ppcondi} is closely tied to the numerical preservation of a discrete divergence-free condition, as demonstrated in \cite{Wu2018Positivity, WuShu2019}.
	
	One effective approach for controlling divergence errors is the eight-wave formulation \cite{Powell1994}, which augments the conservative system \eqref{eq:MHDadd2D} with the Godunov–Powell source terms \cite{Godunov1972SymmetricFO, Powell1994}, resulting in the modified MHD system:
	\begin{equation}\label{eq:MHDadd2DPowell}
		\partial_{t}	\bm{U} + \nabla \cdot \bm{F}(\bm{U}) = -(\nabla \cdot \bm{B}) \bm{S}(\bm{U}),
	\end{equation}
	where $\bm{S}(\bm{U}) = (0, \bm{B}, \bm{v}, \bm{v} \cdot \bm{B})^{\top}$. These source terms restore symmetric hyperbolicity, thereby making the system well-posed even in the presence of small divergence errors. Moreover, they enable divergence errors to be transported with the flow, rather than accumulating locally. 
	Most notably, it has been shown in \cite{WuShu2018, WuShu2019} that the Godunov–Powell source terms decouple the PP property from the divergence-free constraint at the continuous level. At the numerical level, for suitably discretized versions of the modified system \eqref{eq:MHDadd2DPowell}, the PP property depends only on satisfying a {\it locally} divergence-free condition \cite{WuShu2018, WuShu2019}. This condition is compatible with standard PP limiters and can be more readily enforced by numerical schemes.

	\section{Structure-Preserving PAMPA Method for Ideal MHD}\label{S.2}
	This section presents the structure-preserving PAMPA method for the ideal MHD equations.  
	We consider a two-dimensional domain discretized using an $N_x \times N_y$ Cartesian mesh, where each cell is defined as $I_{i,j} = [x_{i-\halfone}, x_{i+\halfone}] \times [y_{j-\halfone}, y_{j+\halfone}]$ with centroid at $(x_i, y_j)$.  
	Without loss of generality, we focus on a uniform mesh, with grid spacings $\Delta x$ and $\Delta y$ in the $x$- and $y$-directions, respectively.

	\subsection{Framework and key components of structure-preserving PAMPA method}
	
	As a variant of the active flux method \cite{Eymann2011active,Eymann2012active},  
	the PAMPA scheme \cite{Abgrall2023Combination} combines both conservative and non-conservative formulations of hyperbolic systems.  
	Unlike traditional finite volume or finite difference methods, PAMPA offers a more compact stencil, lower memory requirements, and greater flexibility.  
	In the classical third-order PAMPA scheme \cite{Abgrall2023Combination}, the degrees of freedom (DoFs), denoted by $\bm{q}$, include both cell averages and point values located at cell interfaces.  
	The cell averages are evolved using the conservative form of the equations, while the point values are evolved via a non-conservative formulation (equivalent to the conservative form in smooth regions):
	\begin{equation}\label{non-conservative}
		\frac{\partial \bm{W}}{\partial t} + \bm{J}^x \frac{\partial \bm{W}}{\partial x} + \bm{J}^y \frac{\partial \bm{W}}{\partial y} = \bm{0}.
	\end{equation}
	Here, $\bm{W} = {\bm \Psi}(\bm{U})$ denotes the non-conservative variables (typically the primitive variables), and ${\bm \Psi}$ is the transformation from conservative to non-conservative variables.  
	The Jacobian matrices $\bm{J}^x$ and $\bm{J}^y$ are given by 
	\begin{equation}
		\bm{J}^x = \left(\frac{\partial {\bm \Psi} (\bm{U})}{\partial \bm{U}}\right) \frac{\partial \bm{F}_1(\bm{U})}{\partial \bm{U}} \left(\frac{\partial {\bm \Psi} (\bm{U})}{\partial \bm{U}}\right)^{-1}, \quad
		\bm{J}^y = \left(\frac{\partial {\bm \Psi} (\bm{U})}{\partial \bm{U}}\right) \frac{\partial \bm{F}_2(\bm{U})}{\partial \bm{U}} \left(\frac{\partial {\bm \Psi} (\bm{U})}{\partial \bm{U}}\right)^{-1}.
	\end{equation}

	Although the PAMAP scheme is constructed based on polynomial representation of the numerical solution, its implementation is actually polynomial-free.   
	In a third-order PAMPA scheme, each cell $I_{i,j}$ involves nine DoFs (with eight point values shared among neighboring cells):
	\begin{itemize}
		\item \textbf{Cell average:} ${\bm q}_{0,i,j} = \overline{{\bm U}}_{i,j} \approx \frac{1}{\Delta x \Delta y}\int_{I_{i,j}}\bm{U}(x, y, t)\, \dd x \dd y$;
		\item \textbf{Four point values at vertices:} 
		${\bm q}_{1,i,j} = {\bm W}_{i+\frac{1}{2},j+\frac{1}{2}}$, 
		${\bm q}_{2,i,j} = {\bm W}_{i+\frac{1}{2},j-\frac{1}{2}}$,  
		${\bm q}_{3,i,j} = {\bm W}_{i-\frac{1}{2},j+\frac{1}{2}}$, 
		${\bm q}_{4,i,j} = {\bm W}_{i-\frac{1}{2},j-\frac{1}{2}}$;
		\item \textbf{Four point values at edge centers:} 
		${\bm q}_{5,i,j} = {\bm W}_{i,j+\frac{1}{2}}$, 
		${\bm q}_{6,i,j} = {\bm W}_{i,j-\frac{1}{2}}$, 
		${\bm q}_{7,i,j} = {\bm W}_{i+\frac{1}{2},j}$, 
		${\bm q}_{8,i,j} = {\bm W}_{i-\frac{1}{2},j}$.
	\end{itemize}
	Here, ${\bm W}_{i+\mu, j+\nu}$ approximates the values of $\bm W$ at the corresponding spatial point $(x_{i+\mu}, y_{j+\nu})$.  
	For clarity, the distribution of DoFs is illustrated in \Cref{fig:DoFs}.

	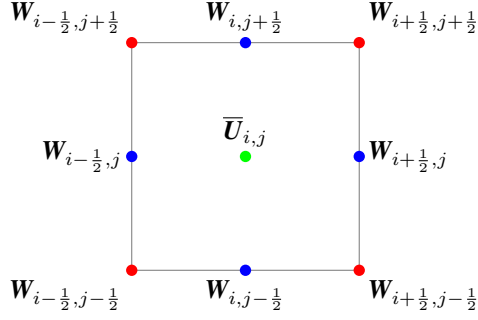
\begin{figure}[htbp]
		\centering
		\begin{tikzpicture}[scale=1.5]
			\draw[gray, thin] (0,0) rectangle (2,2);
			
			\tikzstyle{node}=[circle, fill=blue, inner sep=1.5pt]
			\tikzstyle{corner}=[circle, fill=red, inner sep=1.5pt]
			\tikzstyle{special}=[circle, fill=blue, inner sep=1.5pt, cross out, draw=purple, thick]
			\tikzstyle{center}=[circle, fill=green, inner sep=1.5pt]
			
			\node[corner] at (0,0) {};
			\node[corner] at (0,2) {};
			\node[corner] at (2,0) {};
			\node[corner] at (2,2) {};
			
			\node[node] at (1,0) {};
			\node[node] at (1,2) {};
			\node[node] at (0,1) {};
			\node[node] at (2,1) {};
			
			\node[center] at (1,1) {};
			
			\node[above] at (1,1) {$\overline{{\bm U}}_{i, j}$};
			
			\node[below left] at (0,0) {${\bm W}_{i-\frac{1}{2},j-\frac{1}{2}}$};
			\node[below right] at (2,0) {${\bm W}_{i+\frac{1}{2},j-\frac{1}{2}}$};
			\node[above left] at (0,2) {${\bm W}_{i-\frac{1}{2},j+\frac{1}{2}}$};
			\node[above right] at (2,2) {${\bm W}_{i+\frac{1}{2},j+\frac{1}{2}}$};
			
			\node[left] at (0,1) {${\bm W}_{i-\frac{1}{2},j}$};
			\node[below] at (1,0) {${\bm W}_{i,j-\frac{1}{2}}$};
			\node[right] at (2,1) {${\bm W}_{i+\frac{1}{2},j}$};
			\node[above] at (1,2) {${\bm W}_{i,j+\frac{1}{2}}$};
		\end{tikzpicture}
		\caption{The nine degrees of freedom (DoFs) in the third-order PAMPA method include the cell average, point values at edge centers, and point values at cell vertices.}
		\label{fig:DoFs}
	\end{figure}

	Our goal is to construct a structure-preserving PAMPA scheme that ensures a discrete divergence-free (DDF) condition for the magnetic field  
	and provably preserves the positivity of the updated point values and cell averages.

	To achieve this, we design an unconditionally positivity-preserving (PP) finite difference formulation for updating point values, together with a DDF projection to enforce \eqref{key111}, and combine these with a provably PP finite volume formulation for updating cell averages.  
	Thanks to the flexibility of the PAMPA framework, these two parts can be seamlessly coupled, resulting in a compact and efficient third-order structure-preserving scheme for the ideal MHD system.  
	Specifically, the proposed structure-preserving PAMPA scheme is built upon the following four key components (KC):
	
	\begin{description}
		\item[KC-1: Automatically PP reformulations and unconditionally PP updates of point values.]  
		We propose a non-conservative reformulation of the ideal MHD system such that, for any non-conservative variable \( \bm{W} \in \mathbb{R}^8 \), the corresponding conservative variable \( \bm{U} = {\bf \Psi}^{-1}(\bm{W}) \) always lies in the admissible state set \( \mathcal{G} \).  
		Based on this reformulation, we design an unconditionally PP finite difference scheme for updating point values. All updated point values are guaranteed to remain in \( \mathcal{G} \) without the need for any additional PP limiter.  
		See \Cref{App} for further details.
		
		\item[KC-2: Discrete divergence-free (DDF) projection.]  
		To control divergence errors and ensure the PP property of updated cell averages, we introduce an efficient DDF projection that can be performed locally in each cell.  
		For cell $I_{i,j}$, the eight magnetic field point values are projected into a local DDF space such that the modified values satisfy the DDF constraint:
		\begin{equation}\label{key111}
			\nabla_{i,j} \cdot \tilde{{\bm B}} := \frac{ \overline {\mathcal B}^{x,R}_{i,j} - \overline {\mathcal B}^{x,L}_{i,j}  }{\Delta x}  +  \frac{ \overline {\mathcal B}^{y,U}_{i,j} - \overline {\mathcal B}^{y,D}_{i,j}  }{\Delta y} = 0, 
		\end{equation}
		with 
		\begin{align*}
			\overline {\mathcal B}^{x,R/L}_{i,j} &= \frac{1}{6} (\tilde B_1)^{L/R,U}_{i \pm \frac{1}{2}, j - \frac{1}{2}} + \frac{4}{6} (\tilde B_1)^{L/R}_{i \pm \frac{1}{2}, j} + \frac{1}{6} (\tilde B_1)^{L/R,D}_{i \pm \frac{1}{2}, j + \frac{1}{2}}, \\
			\overline {\mathcal B}^{y,U/D}_{i,j} &= \frac{1}{6} (\tilde B_2)^{R,D/U}_{i - \frac{1}{2}, j \pm \frac{1}{2}} + \frac{4}{6} (\tilde B_2)^{D/U}_{i, j \pm \frac{1}{2}} + \frac{1}{6} (\tilde B_2)^{L,D/U}_{i + \frac{1}{2}, j \pm \frac{1}{2}},
		\end{align*}
		where $\tilde{\bm{B}}^{L/R, D/U}$ denotes the result of applying the DDF projection to $\bm{B}$.  
		Thanks to its local nature, the DDF projection is both efficient and easy to implement.  
		However, it breaks continuity of the point values, resulting in two modified values at edge centers and four at cell vertices.  
		Note that these modified values will be used only for updating cell averages---not for updating point values themselves.  
		Details of the DDF projection are provided in \Cref{IDPCell}.
		
		\item[KC-3: PP limiter for the point value at the cell center.]  
		In addition to interface point values, the scheme also requires the value ${\bm U}_{i,j}^c$ at the cell center $(x_i, y_j)$, computed as
		\begin{equation}\label{cell-CenterU}
			\begin{aligned}
				\bm{U}_{i,j}^{c} &= \frac{36}{16} \overline{\bm{U}}_{i,j} 
				- \frac{1}{16} (\bm{U}_{i-\halfone, j-\halfone}^{R,U} + \bm{U}_{i+\halfone, j-\halfone}^{L,U} + \bm{U}_{i-\halfone, j+\halfone}^{R,D} + \bm{U}_{i+\halfone, j+\halfone}^{L,D}) \\
				&\quad - \frac{4}{16} (\bm{U}_{i-\halfone, j}^{R} + \bm{U}_{i+\halfone, j}^{L} + \bm{U}_{i, j-\halfone}^{U} + \bm{U}_{i, j+\halfone}^{D}).
			\end{aligned}
		\end{equation}
		This value does not, in general, lie in $\mathcal{G}$. Therefore, a local scaling PP limiter is applied to enforce $\check{\bm{U}}_{i,j}^c \in \mathcal{G}$.  
		This limited value ensures that the corresponding non-conservative variable $\check{\bm{W}}_{i,j}^{c}$ is well-defined.  
		This step is essential not only for the PP update of cell averages, but also for enabling the update of point values using $\check{\bm{W}}_{i,j}^{c}$.  
		Further details are given in \Cref{IDPCell}.
		
		\item[KC-4: Convex-oscillation-eliminating (COE) procedure and a shock indicator.]  
		The conventional PAMPA scheme may produce non-physical oscillations near discontinuities.
		To address this issue, we propose a novel COE procedure that suppresses such oscillations while preserving the compactness of the original scheme, along with the DDF and PP properties.
		A Lax-type entropy criterion is employed to detect potential shocks.
		The COE procedure is then selectively applied only in troubled cells identified by the shock detector.
		This targeted application avoids unnecessary dissipation and reduces computational cost, while maintaining stability near discontinuities.
		Importantly, the COE procedure does not alter the cell averages, thereby preserving conservativeness.
		The point values of the conservative variables after applying the COE procedure are denoted by $\hat{\bm{U}}$.
		Details of the COE procedure and the shock indicator are provided in \Cref{OE} and \Cref{trouble-cell}, respectively.
	\end{description}

	Based on the above key components, we construct a PAMPA scheme that is discrete divergence-free (DDF), essentially non-oscillatory, and PP. 
	For clarity and notational convienience, we consider the forward Euler method for time discretization, while the extension to high-order strong-stability-preserving (SSP) time stepping methods will be discussed in \Cref{sec:time}.   
	The proposed scheme ensures that 
	\begin{itemize}
		\item the updated point values
		\[
		{\bm q}^{n+1}_{r,i,j} = {\bm q}^{n}_{r,i,j}  + \Delta t_n {\bm L}_{r,i,j}({\bm W}^n, \overline{{\bm U}}^n, \hat{{\bm U}}^n), \quad r = 1, \ldots, 8
		\]
		are automatically and unconditionally PP;
		
		\item the updated cell averages
		\[
		{\bm q}^{n+1}_{0,i,j} = {\bm q}^{n}_{0,i,j} + \Delta t_n {\bm L}_{0,i,j}(\hat{{\bm U}}^n)
		\]
		are provably PP under a CFL-type condition on the time step size $\Delta t_n$, as established in \Cref{them:PPcfl};
	\end{itemize}
	Here, ${\bm L}_{r,i,j}({\bm W}^n, \overline{{\bm U}}^n, \hat{{\bm U}}^n)$ for $1 \le r \le 8$, and ${\bm L}_{0,i,j}(\hat{{\bm U}}^n)$ denote the spatial discretization operators corresponding to the non-conservative and conservative MHD formulations, respectively; these will be detailed later. For clarity, a flowchart of the proposed method is shown in \Cref{fig:flowchat}.
	
	\begin{figure}[!ht]
		\centering
		\includegraphics{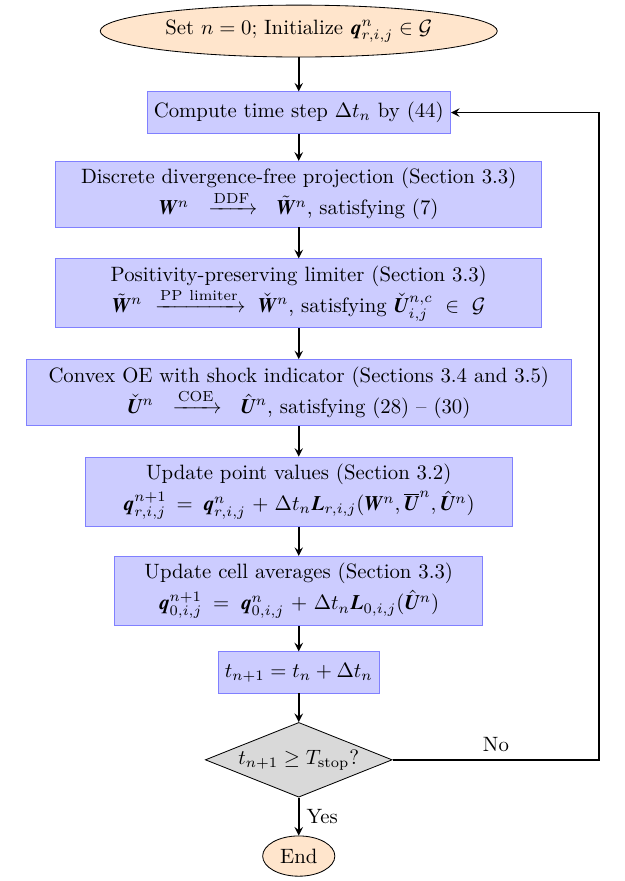}
		\caption{The flowchart of our structure-preserving PAMPA scheme for ideal MHD.}
		\label{fig:flowchat}
	\end{figure}

	\subsection{Automatically PP approach for point-value evolution}\label{App}
	
	Thanks to the flexibility of the PAMPA framework, the point values can be updated using any non-conservative form of the ideal MHD system.  
	The core idea of our automatically PP approach is to select a special set of non-conservative variables ${\bm W}$ such that the mapping $\bm \Psi^{-1}$ from ${\bm W}$ to $\bm U$ satisfies the following property:
	\begin{equation}\label{eq:WUG}
		{\bm U} = \bm \Psi^{-1} ( {\bm W} ) \in \mathcal{G}, \qquad \forall \, {\bm W} \in \mathbb{R}^8.
	\end{equation}
	That is, regardless of how ${\bm W}$ is computed, the corresponding ${\bm U}$ always lies in the admissible state set $\mathcal{G}$.  
	Therefore, by discretizing the reformulated MHD system in terms of these variables ${\bm W}$, the resulting scheme is unconditionally PP without requiring any additional PP limiters.
	
	In this subsection, we introduce the automatically PP reformulation with such ${\bm W}$, followed by a description of the spatial discretization operator ${\bm L}_{r,i,j}({\bm W}^n, \overline{{\bm U}}^n, \hat{{\bm U}}^n)$ used for the point-value evolution \eqref{non-conservative}.
	
	We define
	\[
	\bm{W} = (q, \bm{v}, \bm{B}, s)^\top, \qquad \text{with} \quad s = \ln p - \gamma \ln \rho,
	\]
	where $s$ is the specific entropy, and $q$ is defined using the inverse of the Softplus activation function commonly used in machine learning:
	\begin{equation}\label{changeform1}
		q = \ln \left( \mathrm{e}^{\rho / \rho_{\mathrm{ref}}} - 1 \right),
	\end{equation}
	with $\rho_{\rm ref} > 0$ being a user-defined reference density. For any $q \in \mathbb{R}$, this yields
	\[
	\rho = \rho_{\rm ref} \ln \left( \mathrm{e}^q + 1 \right) > 0,
	\]
	where $\ln \left( \mathrm{e}^x + 1 \right)$ is the Softplus function---a smooth approximation of ReLU---ensuring positivity in neural network outputs.
	
	\begin{remark}\label{changeform}
		The choice of non-conservative variables ${\bm W}$ is not unique.  
		For instance, another valid form of $q$ that ensures the positivity of density is
		\begin{equation}\label{changeform2}
			q = \frac{\rho}{\rho_{\rm ref}} - \frac{\rho_{\rm ref}}{4 \rho},
		\end{equation}
		which implies
		\[
		\rho = \frac{\rho_{\rm ref}}{2} \left( q + \sqrt{q^2 + 1} \right) > 0, \qquad \forall q \in \mathbb{R}.
		\]
	\end{remark}
	
	For any $\bm{W} \in \mathbb{R}^8$, the density $\rho$ computed using either \eqref{changeform1} or \eqref{changeform2} is strictly positive.  
	The thermal pressure is then reconstructed via $s = \ln p - \gamma \ln \rho$, yielding
	\[
	p = \mathrm{e}^s \rho^\gamma > 0,
	\]
	which is also strictly positive.  
	Hence, the conservative variables ${\bm U} = \bm \Psi({\bm W})$ derived from any ${\bm W} \in \mathbb{R}^8$ are guaranteed to belong to the admissible state set $\mathcal{G}$.
	
	Our numerical experiments show that both choices \eqref{changeform1} and \eqref{changeform2} for $q$ perform well in practice.  
	For clarity, we focus on \eqref{changeform1} in the following discussions.  
	We now derive the non-conservative formulation \eqref{non-conservative} in terms of the variables $\bm{W} = (q, \bm{v}, \bm{B}, s)^\top$ from the MHD system \eqref{eq:MHDadd2DPowell}.  
	The Jacobian matrices $\bm{J}^x$ and $\bm{J}^y$ are given by
	\begin{equation}\label{eq:Jx}
		\bm{J}^x = \begin{pmatrix}
			\frac{m_1}{\rho} & \frac{\rho \mathrm{e}^{\rho / \rho_{\mathrm{ref}}}}{\rho_{\mathrm{ref}} \sigma_2} & 0 & 0 & 0 & 0 & 0 & 0 \\
			-\frac{\gamma \rho_{\mathrm{ref}} \mathrm{e}^{-\rho / \rho_{\mathrm{ref}}} \sigma_2 (\gamma - 1) \sigma_1}{2 \rho^3} & \frac{m_1}{\rho} & 0 & 0 & 0 & \frac{B_2}{\rho} & \frac{B_3}{\rho} & -\frac{(\gamma - 1) \sigma_1}{2 \rho^2} \\
			0 & 0 & \frac{m_1}{\rho} & 0 & 0 & -\frac{B_1}{\rho} & 0 & 0 \\
			0 & 0 & 0 & \frac{m_1}{\rho} & 0 & 0 & -\frac{B_1}{\rho} & 0 \\
			0 & 0 & 0 & 0 & \frac{m_1}{\rho} & 0 & 0 & 0 \\
			0 & B_2 & -B_1 & 0 & 0 & \frac{m_1}{\rho} & 0 & 0 \\
			0 & B_3 & 0 & -B_1 & 0 & 0 & \frac{m_1}{\rho} & 0 \\
			0 & 0 & 0 & 0 & 0 & 0 & 0 & \frac{m_1}{\rho}
		\end{pmatrix},
	\end{equation}
	and 
	\begin{equation}\label{eq:Jy}
		\bm{J}^y = \begin{pmatrix}
			\frac{m_2}{\rho} & 0 & \frac{\rho \mathrm{e}^{\rho / \rho_{\mathrm{ref}}}}{\rho_{\mathrm{ref}} \sigma_2} & 0 & 0 & 0 & 0 & 0 \\
			0 & \frac{m_2}{\rho} & 0 & 0 & -\frac{B_2}{\rho} & 0 & 0 & 0 \\
			-\frac{\gamma \rho_{\mathrm{ref}} \mathrm{e}^{-\rho / \rho_{\mathrm{ref}}} \sigma_2 (\gamma - 1) \sigma_1}{2 \rho^3} & 0 & \frac{m_2}{\rho} & 0 & \frac{B_1}{\rho} & 0 & \frac{B_3}{\rho} & -\frac{(\gamma - 1) \sigma_1}{2 \rho^2} \\
			0 & 0 & 0 & \frac{m_2}{\rho} & 0 & 0 & -\frac{B_2}{\rho} & 0 \\
			0 & -B_2 & B_1 & 0 & \frac{m_2}{\rho} & 0 & 0 & 0 \\
			0 & 0 & 0 & 0 & 0 & \frac{m_2}{\rho} & 0 & 0 \\
			0 & 0 & B_3 & -B_2 & 0 & 0 & \frac{m_2}{\rho} & 0 \\
			0 & 0 & 0 & 0 & 0 & 0 & 0 & \frac{m_2}{\rho}
		\end{pmatrix},
	\end{equation}
	where
	\begin{equation}
		\sigma_1 = \rho |\bm{B}|^2 + |\bm{m}|^2 - 2 \rho E, \qquad 
		\sigma_2 = \mathrm{e}^{\rho / \rho_{\mathrm{ref}}} - 1.
	\end{equation}

	We now introduce a finite difference scheme for updating $\bm{W} = (q, \bm{v}, \bm{B}, s)^\top$  by solving the non-conservative form \eqref{non-conservative} with \eqref{eq:Jx}, \eqref{eq:Jy}.

	For each cell $I_{i,j}$, given $\bm{q}^{n}_{r,i,j} \in \mathcal{G}$ for $r = 0, \dots, 8$,  
	a biparabolic reconstruction  
	\[
	\bm{P}_{i,j}^n(x,y) : \left[ x_{i-\frac{1}{2}}, x_{i+\frac{1}{2}} \right] \times \left[ y_{j-\frac{1}{2}}, y_{j+\frac{1}{2}} \right] \to \mathbb{R}^8
	\]
	can be constructed to achieve third-order accuracy by matching
	\begin{equation}\label{recon}
		\frac{1}{\Delta x \Delta y} \int_{I_{i,j}} \bm{P}_{i,j}^n(x,y)\, \dd x \dd y = \bm{q}^n_{0, i,j}, \qquad 
		\bm{P}_{i,j}^n(x_r,y_r) = {\bm \Psi}^{-1}(\bm{q}^n_{r, i,j}), \quad r = 1, \dots, 8,
	\end{equation}
	where $(x_r, y_r)$ denotes the location of the point corresponding to the degree of freedom $\bm{q}^n_{r, i,j}$ on the cell interface, as illustrated in \Cref{fig:DoFs}.   
	Using the exactness of Simpson’s rule for the integration of $\bm{P}_{i,j}^n(x,y)$ over $I_{i,j}$ gives
	\begin{equation}\label{eq:CAD}
		\begin{aligned}
			\overline{\bm{U}}_{i,j}^n &= \frac{16}{36} {\bm{U}}_{i,j}^{n, c} 
			+ \frac{1}{36} \left( \bm{U}_{i-\frac{1}{2}, j-\frac{1}{2}}^n + \bm{U}_{i+\frac{1}{2}, j-\frac{1}{2}}^n + \bm{U}_{i-\frac{1}{2}, j+\frac{1}{2}}^n + \bm{U}_{i+\frac{1}{2}, j+\frac{1}{2}}^n \right) \\
			&\quad + \frac{4}{36} \left( \bm{U}_{i-\frac{1}{2}, j}^n + \bm{U}_{i+\frac{1}{2}, j}^n + \bm{U}_{i, j-\frac{1}{2}}^n + \bm{U}_{i, j+\frac{1}{2}}^n \right),
		\end{aligned}
	\end{equation}
	where ${\bm{U}}_{i,j}^{n, c} = \bm{P}_{i,j}^n(x_i, y_j)$ denotes the value of the reconstructed polynomial at the cell center.  
	Thus, ${\bm{U}}_{i,j}^{n, c}$ can be computed directly, without explicitly reconstructing the polynomial, via
	\begin{equation}\label{cell-CenterU2}
		\begin{aligned}
			\bm{U}_{i,j}^{n,c} &= \frac{36}{16} \overline{\bm{U}}_{i,j}^n 
			- \frac{1}{16} \left( \bm{U}_{i-\frac{1}{2}, j-\frac{1}{2}}^n + \bm{U}_{i+\frac{1}{2}, j-\frac{1}{2}}^n + \bm{U}_{i-\frac{1}{2}, j+\frac{1}{2}}^n + \bm{U}_{i+\frac{1}{2}, j+\frac{1}{2}}^n \right) \\
			&\quad - \frac{4}{16} \left( \bm{U}_{i-\frac{1}{2}, j}^n + \bm{U}_{i+\frac{1}{2}, j}^n + \bm{U}_{i, j-\frac{1}{2}}^n + \bm{U}_{i, j+\frac{1}{2}}^n \right).
		\end{aligned}
	\end{equation}

	The spatial derivatives $\frac{\partial \bm{W}}{\partial x}$ and $\frac{\partial \bm{W}}{\partial y}$ in the non-conservative formulation \eqref{non-conservative}  
	are approximated by differentiating the reconstructed polynomial \cite{AbgrallAF2025} within each cell, and upwinding is incorporated through Lax--Friedrichs splitting of the Jacobian matrices.  
	We now present the semi-discrete formulation of the non-conservative system \eqref{non-conservative}.  
	Time discretization can be carried out using the SSP Runge--Kutta method to achieve high-order temporal accuracy, as described in \Cref{sec:time}.
	
	The point value $\bm{W}_{i+\halfone, j}$ is updated via
	\begin{equation}\label{eq:pointevolve1}
		\frac{\dd}{\dd t} \bm{W}_{i+\halfone, j} = -\left(\bm{J}^x \bm{D}^x\right)_{i+\halfone, j}^{\mathrm{upw}} - \bm{J}_{i+\halfone, j}^y \bm{D}_{i+\halfone, j}^y,
	\end{equation}
	where
	\begin{equation*}
		\left(\bm{J}^x \bm{D}^x\right)_{i+\halfone, j}^{\mathrm{upw}} := \left(\bm{J}^x\right)_{i+\halfone, j}^{+} \left(\bm{D}^x\right)_{i+\halfone, j}^{+} + \left(\bm{J}^x\right)_{i+\halfone, j}^{-} \left(\bm{D}^x\right)_{i+\halfone, j}^{-}, \quad
		\left(\bm{J}^x\right)_{i+\halfone, j}^{\pm} = \frac{\bm{J}_{i+\halfone, j}^x \pm \alpha_1^\mathrm{LF} \bm{I}}{2},
	\end{equation*}
	\begin{equation*}
		\begin{aligned}
			\left(\bm{D}^x\right)_{i+\halfone, j}^{+} &= \frac{1}{4 \Delta x} \bigg( 4\left(-9 \overline{\bm{W}}_{i,j} + 2(\bm{W}_{i-\halfone, j} + 2\bm{W}_{i+\halfone, j})\right) + 4(\bm{W}_{i, j-\halfone} + \bm{W}_{i, j+\halfone}) \\
			& \quad + \bm{W}_{i-\halfone, j-\halfone} + \bm{W}_{i+\halfone, j-\halfone} + \bm{W}_{i-\halfone, j+\halfone} + \bm{W}_{i+\halfone, j+\halfone} \bigg), \\
			\left(\bm{D}^x\right)_{i+\halfone, j}^{-} &= -\frac{1}{4 \Delta x} \bigg( -36 \overline{\bm{W}}_{i+1,j} + 8(2\bm{W}_{i+\halfone, j} + \bm{W}_{i+\frac{3}{2}, j}) + \bm{W}_{i+\halfone, j-\halfone} \\
			& \quad + 4(\bm{W}_{i+1, j-\halfone} + \bm{W}_{i+1, j+\halfone}) + \bm{W}_{i+\frac{3}{2}, j-\halfone} + \bm{W}_{i+\halfone, j+\halfone} + \bm{W}_{i+\frac{3}{2}, j+\halfone} \bigg), \\
			\left(\bm{D}^y\right)_{i+\halfone, j} &= \frac{\bm{W}_{i+\halfone, j+\halfone} - \bm{W}_{i+\halfone, j-\halfone}}{\Delta y}.
		\end{aligned}
	\end{equation*}
	For the point value $\bm{W}_{i, j+\halfone}$:
	\begin{equation}\label{eq:pointevolve2}
		\frac{\dd}{\dd t} \bm{W}_{i, j+\halfone} = -\bm{J}_{i,j+\halfone}^x \bm{D}_{i,j+\halfone}^x - \left(\bm{J}^y \bm{D}^y\right)_{i,j+\halfone}^{\mathrm{upw}},
	\end{equation}
	where
	\begin{equation*}
		\left(\bm{J}^y \bm{D}^y\right)_{i,j+\halfone}^{\mathrm{upw}} := \left(\bm{J}^y\right)_{i,j+\halfone}^{+} \left(\bm{D}^y\right)_{i,j+\halfone}^{+} + \left(\bm{J}^y\right)_{i,j+\halfone}^{-} \left(\bm{D}^y\right)_{i,j+\halfone}^{-}, \quad
		\left(\bm{J}^y\right)_{i,j+\halfone}^{\pm} = \frac{\bm{J}_{i,j+\halfone}^y \pm \alpha_2^\mathrm{LF} \bm{I}}{2},
	\end{equation*}
	\begin{equation*}
		\begin{aligned}
			\left(\bm{D}^y\right)_{i, j+\halfone}^{+} &= \frac{1}{4 \Delta y} \bigg( 4(\bm{W}_{i-\halfone, j} - 9\overline{\bm{W}}_{i,j} + \bm{W}_{i+\halfone, j}) + \bm{W}_{i-\halfone, j\pm\halfone} + \bm{W}_{i+\halfone, j\pm\halfone} \\
			& \quad + 8(\bm{W}_{i, j-\halfone} + 2\bm{W}_{i, j+\halfone}) \bigg), \\
			\left(\bm{D}^y\right)_{i, j+\halfone}^{-} &= -\frac{1}{4 \Delta y} \bigg( 4(\bm{W}_{i-\halfone, j+1} - 9\overline{\bm{W}}_{i,j+1} + \bm{W}_{i+\halfone, j+1}) + \bm{W}_{i\pm\halfone, j+\halfone} \\
			& \quad + \bm{W}_{i\pm\halfone, j+\frac{3}{2}} + 8(2\bm{W}_{i,j+\halfone} + \bm{W}_{i,j+\frac{3}{2}}) \bigg), \\
			\left(\bm{D}^x\right)_{i, j+\halfone} &= \frac{\bm{W}_{i+\halfone, j+\halfone} - \bm{W}_{i-\halfone, j+\halfone}}{\Delta x}.
		\end{aligned}
	\end{equation*}
	For the vertex point value $\bm{W}_{i+\halfone, j+\halfone}$:
	\begin{equation}\label{eq:pointevolve3}
		\frac{\dd}{\dd t} \bm{W}_{i+\halfone, j+\halfone} = -\left(\bm{J}^x \bm{D}^x\right)_{i+\halfone, j+\halfone}^{\mathrm{upw}} - \left(\bm{J}^y \bm{D}^y\right)_{i+\halfone,j+\halfone}^{\mathrm{upw}},
	\end{equation}
	with
	\begin{equation*}
		\left(\bm{J}^x\right)_{i+\halfone, j+\halfone}^{\pm} = \frac{\bm{J}_{i+\halfone, j+\halfone}^x \pm \alpha_1^\mathrm{LF} \bm{I}}{2}, \quad
		\left(\bm{J}^y\right)_{i+\halfone, j+\halfone}^{\pm} = \frac{\bm{J}_{i+\halfone, j+\halfone}^y \pm \alpha_2^\mathrm{LF} \bm{I}}{2},
	\end{equation*}
	\begin{equation*}
		\begin{aligned}
			\left(\bm{D}^x\right)_{i+\halfone, j+\halfone}^{+} &= \frac{\bm{W}_{i-\halfone, j+\halfone} - 4\bm{W}_{i, j+\halfone} + 3\bm{W}_{i+\halfone, j+\halfone}}{\Delta x}, \\
			\left(\bm{D}^x\right)_{i+\halfone, j+\halfone}^{-} &= \frac{4\bm{W}_{i+1, j+\halfone} - 3\bm{W}_{i+\halfone, j+\halfone} - \bm{W}_{i+\frac{3}{2}, j+\halfone}}{\Delta x}, \\
			\left(\bm{D}^y\right)_{i+\halfone, j+\halfone}^{+} &= \frac{\bm{W}_{i+\halfone, j-\halfone} - 4\bm{W}_{i+\halfone, j} + 3\bm{W}_{i+\halfone, j+\halfone}}{\Delta y}, \\
			\left(\bm{D}^y\right)_{i+\halfone, j+\halfone}^{-} &= \frac{4\bm{W}_{i+\halfone, j+1} - 3\bm{W}_{i+\halfone, j+\halfone} - \bm{W}_{i+\halfone, j+\frac{3}{2}}}{\Delta y}.
		\end{aligned}
	\end{equation*}
	Here, $\bm{I}$ denotes the identity matrix and $\alpha_i^\mathrm{LF}$ is a numerical viscosity parameter, which will be specified in the next subsection.
	
	Note that the quantity $\overline{\bm{W}}_{i,j}^n$ is not directly available, and the cell-centered value from \eqref{cell-CenterU} may not lie within the admissible state set $\mathcal{G}$, causing the transformation $\bm{W}_{i,j}^{n,c} = {\bf \Psi}(\bm{U}_{i,j}^{n,c})$ to fail.  
	To address both issues, we propose computing $\overline{\bm{W}}_{i,j}^n$ via
	\begin{equation}\label{avg:W}
		\begin{aligned}
			\hat{\bm{U}}_{i,j}^{n, c} &= \frac{36}{16} \overline{\bm{U}}_{i,j}^n - \frac{1}{16} (\hat{\bm{U}}_{i-\halfone, j-\halfone}^{n, RU} + \hat{\bm{U}}_{i+\halfone, j-\halfone}^{n, LU} + \hat{\bm{U}}_{i-\halfone, j+\halfone}^{n, RD} + \hat{\bm{U}}_{i+\halfone, j+\halfone}^{n, LD}) \\
			&\quad - \frac{4}{16} (\hat{\bm{U}}_{i-\halfone, j}^{n, R} + \hat{\bm{U}}_{i+\halfone, j}^{n, L} + \hat{\bm{U}}_{i, j-\halfone}^{n, U} + \hat{\bm{U}}_{i, j+\halfone}^{n, D}), \\
			\overline{\bm{W}}_{i,j}^n &= \frac{16}{36} \hat{\bm{W}}_{i,j}^{n, c} + \frac{1}{36} (\hat{\bm{W}}_{i-\halfone, j-\halfone}^{n, RU} + \hat{\bm{W}}_{i+\halfone, j-\halfone}^{n, LU} + \hat{\bm{W}}_{i-\halfone, j+\halfone}^{n, RD} + \hat{\bm{W}}_{i+\halfone, j+\halfone}^{n, LD}) \\
			&\quad + \frac{4}{36} (\hat{\bm{W}}_{i-\halfone, j}^{n,+} + \hat{\bm{W}}_{i+\halfone, j}^{n,-} + \hat{\bm{W}}_{i, j-\halfone}^{n,+} + \hat{\bm{W}}_{i, j+\halfone}^{n,-}).
		\end{aligned}
	\end{equation}
	The limiting values such as $\hat{\bm{U}}_{i+\halfone, j-\halfone}^{n, LU}$ and $\hat{\bm{U}}_{i-\halfone, j}^{n, R}$ are obtained after applying  
	the PP limiter (see KC-3) and the COE procedure (see KC-4), as detailed in \Cref{IDPCell,OE}, respectively.  
	The PP limiter ensures that $\hat{\bm{U}}_{i,j}^{n,c} \in \mathcal{G}$, so that $\hat{\bm{W}}_{i,j}^{n,c} = \Psi(\hat{\bm{U}}_{i,j}^{n,c})$ is well-defined.

	\begin{remark}
		Although the reconstructed polynomial ${\bm P}_{i,j}^n(x,y)$ is used conceptually in designing the scheme, the final scheme and its actual implementation does not involve it explicitly.  
		Only the point values and the cell-center value (computed via \eqref{avg:W}) are required.  
		Thus, the proposed PAMPA scheme for point-value evolution is indeed {\it polynomial-avoiding}.
	\end{remark}

	\subsection{Provably PP approach for cell-average evolution}\label{IDPCell}
	
	This section presents a provably PP scheme for updating the cell averages of $\bm{U}$ by properly discretizing the modified MHD equations \eqref{eq:MHDadd2DPowell} with the Godunov--Powell source term. 
	As shown in \cite{WuShu2018,WuShu2019}, when this source term is appropriately discretized, the PP property of a finite volume or discontinuous Galerkin scheme for \eqref{eq:MHDadd2DPowell} depends only on a {\it locally}---rather than globally---divergence-free condition, which can be more readily enforced in the PAMPA framework.
	Due to the continuity of point values across cell interfaces, 
	classical PAMPA schemes for cell averages often employ single-state fluxes. However, \cite{abgrall2024novel} demonstrates that such fluxes are generally infeasible to ensure positivity of the cell averages. 
	This motivates the use of {\it two-state} numerical fluxes as a necessary ingredient for achieving a provably PP scheme. 
	
	In our structure-preserving PAMPA scheme for ideal MHD, the cell-average update takes the form:
	\begin{equation}\label{eq:avgGL}
		\begin{aligned}
			\frac{\dd}{\dd t} \overline{\bm{U}}_{i,j}(t) = {\bm L}_{0,i,j}(\hat{{\bm U}}) = & -\frac{1}{\Delta x} \left( 	 \overline{{\bm F}}_{1,i+\frac12,j} - \overline{{\bm F}}_{1,i-\frac12,j} \right) 
			-\frac{1}{\Delta y}  \left( \overline{{\bm F}}_{2,i,j+\frac12} - 	 \overline{{\bm F}}_{2,i,j-\frac12} \right)
			- \bm{S}_{i,j}(\hat{\bm{U}}),
		\end{aligned}
	\end{equation}
	with 
	\begin{equation}\label{eq:numflux}
		\begin{aligned}
			\overline{{\bm F}}_{1,i+\frac12,j}  &:= \frac{1}{6} \hat{\bm{F}}_1( \hat{\bm{U}}_{i+\halfone,j+\halfone}^{L,U}, \hat{\bm{U}}_{i+\halfone,j+\halfone}^{R,U}) + \frac{4}{6} \hat{\bm{F}}_1( \hat{\bm{U}}_{i+\halfone,j}^{L}, \hat{\bm{U}}_{i+\halfone,j}^{R}) + \frac{1}{6} \hat{\bm{F}}_1( \hat{\bm{U}}_{i+\halfone,j-\halfone}^{L,D}, \hat{\bm{U}}_{i+\halfone,j-\halfone}^{R,D}), \\
			\overline{{\bm F}}_{2,i,j+\frac12}  &:= \frac{1}{6} \hat{\bm{F}}_2( \hat{\bm{U}}_{i-\halfone,j+\halfone}^{R,D}, \hat{\bm{U}}_{i-\halfone,j+\halfone}^{R,U}) + \frac{4}{6} \hat{\bm{F}}_2( \hat{\bm{U}}_{i,j+\halfone}^{D}, \hat{\bm{U}}_{i,j+\halfone}^{U}) + \frac{1}{6} \hat{\bm{F}}_2( \hat{\bm{U}}_{i+\halfone,j+\halfone}^{L,D}, \hat{\bm{U}}_{i+\halfone,j+\halfone}^{L,U}).			
		\end{aligned}
	\end{equation}
	Here, $\hat{\bm{F}}_\ell(\cdot,\cdot)$ denotes a suitable two-state numerical flux, and $\bm{S}_{i,j}$ is the discretized Godunov--Powell source term, to be specified later.
	
	The limiting values (e.g., $\hat{\bm{U}}_{i+\halfone,j+\halfone}^{L,U}$ or $\hat{\bm{U}}_{i+\halfone,j}^{L}$) are obtained after applying the DDF projection, the PP limiter, and the COE procedure. Due to their local scaling nature, these operations may break continuity of the point values across cell interfaces. As a result, each edge-center point value splits into two one-sided limits (indicated by superscripts $L/R$ or $U/D$), and each vertex yields four limiting values, labeled $RU$, $LU$, $RD$, and $LD$ to denote their relative positions (Right-Up, Left-Up, Right-Down, Left-Down). 
	See \Cref{fig:grid_dofs} for an illustration.
	\begin{figure}[ht]
		\centering
		\includegraphics{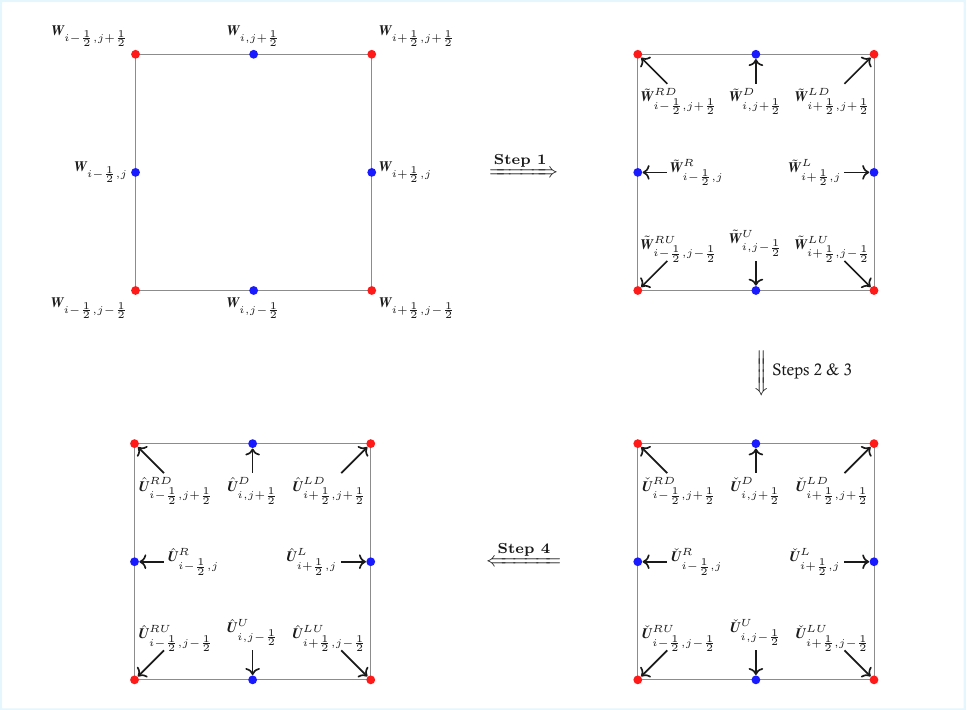}
		\caption{Illustration of the limiting values produced after the DDF projection ({\bf Step 1}), the PP limiter ({\bf Steps 2 and 3}), and the COE procedure ({\bf Step 4}).}
		\label{fig:grid_dofs}
	\end{figure}
	Given the nine degrees of freedom (DoFs) at time $t_n$:
	\[
	\overline{\bm{U}}_{i,j} \in \mathcal{G}, \quad 
	{\bm W}_{i\pm \frac{1}{2},j\pm \frac{1}{2}}, \quad {\bm W}_{i,j\pm \frac{1}{2}}, \quad {\bm W}_{i\pm\frac{1}{2},j},
	\]
	we now describe the four steps as follow to obtain the limiting values used in \eqref{eq:numflux} to ensure the provably PP property of the updated cell averages.  
	For notational convenience, we omit the time-level superscript $n$ from here on.
	
	\subsubsection{Step1: DDF projection} 
	For each cell $I_{i,j}$, we modify the normal magnetic field components at cell interfaces to satisfy the discrete divergence-free (DDF) condition \eqref{key111}, as follows:
	\begin{align*}
		(\tilde B_1)^{L/R,U}_{i \pm \frac{1}{2}, j - \frac{1}{2}} &= (B_1)_{i \pm \frac{1}{2}, j - \frac{1}{2}} \mp A_1, \quad
		(\tilde B_1)^{L/R}_{i \pm \frac{1}{2}, j} = (B_1)_{i \pm \frac{1}{2}, j} \mp A_1, \\
		(\tilde B_1)^{L/R,D}_{i \pm \frac{1}{2}, j + \frac{1}{2}} &= (B_1)_{i \pm \frac{1}{2}, j + \frac{1}{2}} \mp A_1, \\
		(\tilde B_2)^{L,D/U}_{i - \frac{1}{2}, j \pm \frac{1}{2}} &= (B_2)_{i - \frac{1}{2}, j \pm \frac{1}{2}} \mp A_2, \quad
		(\tilde B_2)^{D/U}_{i , j\pm \frac{1}{2}} = (B_2)_{i , j\pm \frac{1}{2}} \mp A_2, \\
		(\tilde B_2)^{R, D/U}_{i + \frac{1}{2}, j \pm \frac{1}{2}} &= (B_2)_{i + \frac{1}{2}, j \pm \frac{1}{2}} \mp A_2,
	\end{align*}
	where
	\[
	A_1 = \frac{\Delta x \nabla_{i,j} \cdot \bm{B}}{2(1 + (\Delta x / \Delta y)^2)}, \qquad
	A_2 = \frac{\Delta y \nabla_{i,j} \cdot \bm{B}}{2(1 + (\Delta y / \Delta x)^2)},
	\]
	and the discrete divergence of $\bm{B}$ is given by
	\begin{equation}\label{key11122}
		\nabla_{i,j} \cdot \bm{B} = \frac{\overline{\mathcal{B}}^{x}_{i+\frac{1}{2},j} - \overline{\mathcal{B}}^{x}_{i-\frac{1}{2},j}}{\Delta x} + \frac{\overline{\mathcal{B}}^{y}_{i,j+\frac{1}{2}} - \overline{\mathcal{B}}^{y}_{i,j-\frac{1}{2}}}{\Delta y},
	\end{equation}
	with
	\begin{align*}
		\overline{\mathcal{B}}^{x}_{i\pm \frac{1}{2},j} = \frac{1}{6}(B_1)_{i \pm \frac{1}{2}, j - \frac{1}{2}} + \frac{4}{6}(B_1)_{i \pm \frac{1}{2}, j} + \frac{1}{6}(B_1)_{i \pm \frac{1}{2}, j + \frac{1}{2}}, 
	\end{align*}
	\begin{align*}
		\overline{\mathcal{B}}^{y}_{i,j\pm \frac{1}{2}} = \frac{1}{6}(B_2)_{i - \frac{1}{2}, j \pm \frac{1}{2}} + \frac{4}{6}(B_2)_{i, j \pm \frac{1}{2}} + \frac{1}{6}(B_2)_{i + \frac{1}{2}, j \pm \frac{1}{2}}.
	\end{align*}
	It can be verified that the modified magnetic field values $\tilde{\bm{B}}$ satisfy the DDF condition $\nabla_{i,j} \cdot \tilde{\bm{B}} = 0$. 
	The point values after this projection are denoted as $\tilde{{\bm{W}}}$, which yield conservative variables $\tilde{\bm{U}} = \bm{\Psi}^{-1}(\tilde{\bm{W}}) \in \mathcal{G}$ by \eqref{eq:WUG}.

	\subsubsection{Step 2: PP limiter for cell-center point value}

	For each cell \( I_{i,j} \), we first compute the value of the conservative state at the cell center: 
	\begin{equation}\label{eq:wkl331}
		\begin{aligned}
			\tilde{\bm{U}}_{i,j}^{c} &= \frac{36}{16} \overline {\bm{U}}_{i,j} - \frac{1}{16} (\tilde{\bm{U}}_{i-\halfone, j-\halfone}^{R,U} + \tilde{\bm{U}}_{i+\halfone, j-\halfone}^{L,U} + \tilde{\bm{U}}_{i-\halfone, j+\halfone}^{R,D} + \tilde{\bm{U}}_{i+\halfone, j+\halfone}^{L,D}) \\
			&\quad - \frac{4}{16} (\tilde{\bm{U}}_{i-\halfone, j}^{R} + \tilde{\bm{U}}_{i+\halfone, j}^{L} + \tilde{\bm{U}}_{i, j-\halfone}^{U} + \tilde{\bm{U}}_{i, j+\halfone}^{D}). \\
		\end{aligned}
	\end{equation}
	Although all terms on the right-hand side of \eqref{eq:wkl331} (cell average and point values) belong to $\mathcal{G}$, the resulting cell-center value $\tilde{\bm{U}}_{i,j}^{c}$ may not stay in  $\mathcal G$. 
	To ensure admissibility, we adapt the PP limiter proposed by Zhang and Shu \cite{ZHANG20108918} to our case. Specifically, we modify $\tilde{\bm{U}}_{i,j}^{c}$ to 
	\begin{equation}\label{eq:PP_Uc}
		\check{\bm{U}}_{i,j}^{c} = \theta_2 \left( \tilde{\tilde{{\bm U}}}^c_{i,j} - \overline{\bm{U}}_{i,j} \right) + \overline{\bm{U}}_{i,j}, \qquad  \theta_2 = \min \left\{ 1, \frac{\mathcal{E}(\overline{\bm{U}}_{i,j}) - \epsilon_2}{\mathcal{E}(\overline{\bm{U}}_{i,j}) -  \mathcal{E}(\tilde{\tilde{\bm{U}}}^c_{i,j})} \right\},
	\end{equation}
	where $\epsilon_2 = \min \left\{ 10^{-13},\, \mathcal{E}(\overline{\bm{U}}_{i,j}) \right\}$ and the intermediate state $\tilde{\tilde{{\bm U}}}^c_{i,j}$ is defined as 
	\begin{equation*}
		\tilde{\tilde{{\bm U}}}^c_{i,j} = \theta_1 \left( \tilde{\bm{U}}_{i,j}^{c} - \overline{\bm{U}}_{i,j} \right) + \overline{\bm{U}}_{i,j}, \qquad
		\theta_1 = \min \left\{ 1, \frac{\overline{\rho}_{i,j} - \epsilon_1}{\overline{\rho}_{i,j} -  \tilde{\rho}^c_{i,j}} \right\}, \qquad
		\epsilon_1 = \min \left\{ 10^{-13},\, \overline{\rho}_{i,j} \right\},
	\end{equation*}
	It can be verified that the limited value satisfies $\check{\rho}_{i,j}^c > 0$ and $\mathcal{E}(\check{\bm{U}}_{i,j}^c) > 0$, ensuring $\check{\bm{U}}_{i,j}^c \in \mathcal{G}$.

	\subsubsection{Step 3: Convex decomposition of cell average} 
	
	To establish the provably PP property for the updated cell averages, we require a convex decomposition \cite{Cui2024On} that exactly expresses $\overline{\bm{U}}_{i,j}$ as a convex combination of point values in $I_{i,j}$. While this holds for the original cell-center value (see \eqref{eq:CAD}), the modified value $\check{\bm{U}}_{i,j}^{c}$ no longer matches $\overline{\bm{U}}_{i,j}$ and the surrounding point values: 
	\begin{equation*}
		\tilde{\mathbb{U}}_{i,j} = \left\{ \tilde{\bm{U}}^{L/R, U}_{i \pm \halfone, j-\halfone},\ \tilde{\bm{U}}^{ L/R }_{i \pm \halfone,j},\ \tilde{\bm{U}}^{L/R, D}_{i \pm \halfone, j + \halfone},\  \tilde{\bm{U}}^{ D/U }_{i, j \pm \halfone}  \right\}
	\end{equation*}
	To maintain the convex decomposition, we modify each point value $\tilde{\bm{U}}$ in $\tilde{\mathbb{U}}_{i,j}$ as
	\[
	\check{\bm{U}} = \theta_2 \theta_1 \left( \tilde{\bm{U}} - \overline{\bm{U}}_{i,j} \right) + \overline{\bm{U}}_{i,j},
	\]
	which ensures $\check{\bm{U}} \in \mathcal{G}$ by convexity. The resulting convex decomposition becomes
	\begin{equation}\label{eq:CAD2}
		\begin{aligned}
			\overline{\bm{U}}_{i,j} &= \frac{16}{36} \check {\bm{U}}_{i,j}^{c} 
			+ \frac{1}{36} \left( \check{\bm{U}}_{i-\frac{1}{2}, j-\frac{1}{2}}^{R,U} + \check{\bm{U}}_{i+\frac{1}{2}, j-\frac{1}{2}}^{L.U} + \check{\bm{U}}_{i-\frac{1}{2}, j+\frac{1}{2}}^{R,D} + \check{\bm{U}}_{i+\frac{1}{2}, j+\frac{1}{2}}^{L,D} \right) \\
			&\quad + \frac{4}{36} \left( \check{\bm{U}}_{i-\frac{1}{2}, j}^{R} + \check{\bm{U}}_{i+\frac{1}{2}, j}^{L} + \check{\bm{U}}_{i, j-\frac{1}{2}}^{U} + \check{\bm{U}}_{i, j+\frac{1}{2}}^{D} \right),
		\end{aligned}
	\end{equation}
	It is worth emphasizing that the modified values continue to satisfy the DDF condition: $\nabla_{i,j} \cdot \check{\bm{B}} = 0$.

	\subsubsection{Step 4: COE procedure}
	
	The above modified values are already sufficient to ensure a provably PP update of the cell averages, given appropriate numerical fluxes and a discretization of the Godunov--Powell source term. However, the PAMPA scheme may still produce spurious oscillations near strong shocks. To mitigate this, we introduce a novel COE procedure that suppresses such oscillations, coupled with a simple yet effective shock indicator to enhance both resolution and efficiency.

	In this step, a shock indicator (see \Cref{trouble-cell}) is first used to detect troubled cells that may contain physical discontinuities.
	For each troubled cell $I_{i,j}$, define
	$$
	\check{\mathbb{U}}_{i,j} = \left\{ \check{\bm{U}}^{L/R, U}_{i \pm \halfone, j-\halfone},\ \check{\bm{U}}^{ L/R }_{i \pm \halfone,j},\ \check{\bm{U}}^{L/R, D}_{i \pm \halfone, j + \halfone},\  \check{\bm{U}}^{ D/U }_{i, j \pm \halfone}  \right\}, 
	$$
	and apply the COE procedure to each $\check{\bm{U}} \in \check{\mathbb{U}}_{i,j} \cup {\check{\bm{U}}^c_{i,j}}$:
	$$
	\hat{\bm{U}} =  \theta^{OE}_{i,j} (\check{\bm{U}} - \overline{\bm{U}}_{i,j}) + \overline{\bm{U}}_{i,j},
	$$
	where the computation of $\theta^{\mathrm{OE}}_{i,j}$ is detailed in \Cref{OE}.
	
	Since $\theta^{\mathrm{OE}}_{i,j} \in [0, 1]$, the COE procedure preserves the DDF condition, the positivity of the point values, and the convex decomposition. Consequently, the modified values satisfy
	\begin{equation}\label{eq:PPcon}
		\hat{\bm{U}}^c_{i,j} \in \mathcal{G} \quad \text{and} \quad
		\hat{\bm{U}} \in \mathcal{G} \quad \forall \hat{\bm{U}} \in \hat{\mathbb{U}}_{i,j},
	\end{equation}
	\begin{equation}\label{eq:DDF3}
		\nabla_{i,j} \cdot \hat{\bm{B}} = 0,
	\end{equation}
	\begin{equation}\label{eq:CAD3}
		\begin{aligned}
			\overline{\bm{U}}_{i,j} &= \frac{16}{36} \hat {\bm{U}}_{i,j}^{c} 
			+ \frac{1}{36} \left( \hat {\bm{U}}_{i-\frac{1}{2}, j-\frac{1}{2}}^{R,U} + \hat {\bm{U}}_{i+\frac{1}{2}, j-\frac{1}{2}}^{L.U} + \hat{\bm{U}}_{i-\frac{1}{2}, j+\frac{1}{2}}^{R,D} + \hat{\bm{U}}_{i+\frac{1}{2}, j+\frac{1}{2}}^{L,D} \right) \\
			&\quad + \frac{4}{36} \left( \hat{\bm{U}}_{i-\frac{1}{2}, j}^{R} + \hat{\bm{U}}_{i+\frac{1}{2}, j}^{L} + \hat{\bm{U}}_{i, j-\frac{1}{2}}^{U} + \hat{\bm{U}}_{i, j+\frac{1}{2}}^{D} \right). 
		\end{aligned}
	\end{equation}

	With the above four steps, the limiting values $\hat{\bm{U}}$ in \eqref{eq:numflux} are determined.

	\subsubsection{PP numerical flux}
	
	We now describe the numerical fluxes in \eqref{eq:avgGL}. 
	To construct a provably positivity-preserving (PP) scheme for the ideal MHD equations, one may use either PP Lax–Friedrichs (LF) or PP HLL fluxes \cite{WuShu2018,WuShu2019}. For simplicity, we adopt the PP LF flux in \eqref{eq:avgGL}:
	\begin{equation}\label{LFflux}
		\hat{\bm{F}_{\ell}}(\bm{U}^{-}, \bm{U}^{+})=\halfone\left(\bm{F}_{\ell}(\bm{U}^{-})+\bm{F}_{\ell}(\bm{U}^{+})-\alpha_{\ell}^{\mathrm{LF}}(\bm{U}^{+}-\bm{U}^{-})\right), \quad \ell=1,2,
	\end{equation}
	where $\alpha_{\ell}^{\mathrm{LF}}$ are the numerical viscosity parameters. To guarantee the PP property, following theoretical estimates in \cite{WuShu2018,DingWu2024SISCMHD}, we take
	\begin{equation*}
		\alpha_1^{\mathrm{LF}} = \hat{\alpha}_1^{\mathrm{LF}}+\beta_1,  \quad \alpha_2^{\mathrm{LF}} = \hat{\alpha}_2^{\mathrm{LF}}+\beta_2
	\end{equation*}
	with
	\begin{equation*}
		\begin{aligned}
			&\hat{\alpha}_1^{\mathrm{LF}} =\max  \left\{\alpha^a_{1}, \alpha^b_{1}, \alpha^c_{1}\right\},\qquad \hat{\alpha}_2^{\mathrm{LF}} =\max  \left\{\alpha^a_{2}, \alpha^b_{2}, \alpha^c_{2}\right\}, \\
			&\alpha^a_{1} = \max_{i, j}  \left\{ \alpha_1(\bm{U}_{i+\halfone, j+\halfone}^{L,D}, \bm{U}_{i-\halfone, j+\halfone}^{R,D}), \alpha_1(\bm{U}_{i+\halfone, j+\halfone}^{R,D}, \bm{U}_{i-\halfone, j+\halfone}^{L,D}) \right\}, \\
			&\alpha^b_{1} = \max_{i, j}  \left\{ \alpha_1(\bm{U}_{i+\halfone, j}^{L}, \bm{U}_{i-\halfone, j}^{R}), \alpha_1(\bm{U}_{i+\halfone, j}^{R}, \bm{U}_{i-\halfone, j}^{L}) \right\}, \\
			&\alpha^c_{1} = \max_{i, j}  \left\{ \alpha_1(\bm{U}_{i+\halfone, j-\halfone}^{L,U}, \bm{U}_{i-\halfone, j-\halfone}^{R,U}), \alpha_1(\bm{U}_{i+\halfone, j-\halfone}^{R,U}, \bm{U}_{i-\halfone, j-\halfone}^{L,U}) \right\}, 
		\end{aligned}
	\end{equation*}
	\begin{equation*}
		\begin{aligned}
			&\alpha^a_{2} = \max_{i, j}  \left\{ \alpha_1(\bm{U}_{i+\halfone, j+\halfone}^{L,D}, \bm{U}_{i+\halfone, j-\halfone}^{L,U}), \alpha_1(\bm{U}_{i+\halfone, j+\halfone}^{L,U}, \bm{U}_{i+\halfone, j-\halfone}^{L,D}) \right\}, \\
			&\alpha^b_{2} = \max_{i, j}  \left\{ \alpha_1(\bm{U}_{i, j+\halfone}^{D}, \bm{U}_{i, j-\halfone}^{U}), \alpha_1(\bm{U}_{i, j+\halfone}^{U}, \bm{U}_{i, j-\halfone}^{D}) \right\}, \\
			&\alpha^c_{2} = \max_{i, j}  \left\{ \alpha_1(\bm{U}_{i-\halfone, j+\halfone}^{R,D}, \bm{U}_{i-\halfone, j-\halfone}^{R,U}), \alpha_1(\bm{U}_{i-\halfone, j+\halfone}^{R,U}, \bm{U}_{i-\halfone, j-\halfone}^{R,D}) \right\}, \\				
			& \beta_1=\max _{i, j}\left\{\frac{\left|[\![B_1]\!]_{i+\halfone, j+\halfone}^D\right|}{2 \sqrt{\average{\rho}_{i+\halfone, j+\halfone}^D}},\frac{\left|[\![B_1]\!]_{i+\halfone, j}\right|}{2 \sqrt{\average{\rho}_{i+\halfone, j}}},\frac{\left|[\![B_1]\!]_{i+\halfone, j-\halfone}^U\right|}{2 \sqrt{\average{\rho}^U_{i+\halfone, j-\halfone}}}\right\}, \\
			&\beta_2=\max _{i, j}\left\{\frac{\left|[\![B_2]\!]_{i+\halfone, j+\halfone}^L \right|}{2 \sqrt{\average{\rho}_{i+\halfone, j+\halfone}^L}}, \frac{\left|[\![B_2]\!]_{i, j+\halfone}\right|}{2 \sqrt{\average{\rho}_{i, j+\halfone}}}, \frac{\left|[\![B_2]\!]_{i-\halfone, j+\halfone}^R\right|}{2 \sqrt{\average{\rho}_{i-\halfone. j+\halfone}^R}}\right\}.
		\end{aligned}
	\end{equation*}
	The local wave speed estimate $\alpha_\ell(\bm{U}, \tilde{\bm{U}})$ is given by
	\begin{equation*}
		\alpha_\ell(\bm{U}, \tilde{\bm{U}})=\max \left\{\left|v_\ell\right|+\mathcal{C}_\ell,\left|\tilde{v}_\ell\right|+\tilde{\mathcal{C}}_\ell, \left|\frac{\sqrt{\rho} v_\ell+\sqrt{\tilde{\rho}} \tilde{v}_\ell}{\sqrt{\rho}+\sqrt{\tilde{\rho}}}\right|+\max \left\{\mathcal{C}_\ell, \tilde{\mathcal{C}}_\ell\right\}\right\}+\frac{|\bm{B}-\tilde{\bm{B}}|}{\sqrt{\rho}+\sqrt{\tilde{\rho}}},
	\end{equation*}
	\begin{equation*}
		\mathcal{C}_\ell=\frac{1}{\sqrt{2}}\left[\mathcal{C}_s^2+\frac{|\bm{B}|^2}{\rho}+\sqrt{\left(\frac{\mathcal{C}_s^2+|\bm{B}|^2}{\rho}\right)^2-
			\frac{4 \mathcal{C}_s^2 B_\ell^2}{\rho}}\right]^{\halfone}, \quad \mathcal{C}_s^2=\frac{p(\gamma-1)}{2\rho},
	\end{equation*}
	and $[\![\cdot]\!]$ denotes the jump of the limiting values across the cell interface
	\begin{equation*}
		\begin{aligned}
			&[\![B_1]\!]_{i+\halfone, j} :=(B_1)_{i+\halfone, j}^{R}-(B_1)_{i+\halfone, j}^{L}, \quad
			[\![B_2]\!]_{i, j+\halfone} :=(B_2)_{i, j+\halfone}^{U}-(B_2)_{i, j+\halfone}^{D}, \\
			&[\![B_1]\!]_{i+\halfone, j+\halfone}^D :=(B_1)_{i+\halfone, j+\halfone}^{R,D}-(B_1)_{i+\halfone, j+\halfone}^{L,D}, \quad
			[\![B_2]\!]_{i+\halfone, j+\halfone}^L :=(B_2)_{i+\halfone, j+\halfone}^{L,U}-(B_2)_{i+\halfone, j+\halfone}^{L,D}, \\
			&[\![B_1]\!]_{i+\halfone, j-\halfone}^U :=(B_1)_{i+\halfone, j-\halfone}^{R,U}-(B_1)_{i+\halfone, j-\halfone}^{L,U}, \quad
			[\![B_2]\!]_{i-\halfone, j+\halfone}^R :=(B_2)_{i-\halfone, j+\halfone}^{R,U}-(B_2)_{i-\halfone, j+\halfone}^{R,D}.		
		\end{aligned}
	\end{equation*}
	The average of limiting values is denoted by $\average{\cdot}$, e.g., 
	\begin{equation*}
		\average{\bm{U}}_{i+\halfone, j} := \frac{1}{2} \left(\bm{U}_{i+\halfone, j}^{R} + \bm{U}_{i+\halfone, j}^{L} \right), \quad
		\average{\bm{U}}_{i, j+\halfone} := \frac{1}{2} \left(\bm{U}_{i, j+\halfone}^{U} + \bm{U}_{i, j+\halfone}^{D} \right).
	\end{equation*}

	\subsubsection{Discretization of the Godunov--Powell source term}
	We now introduce the discretization of the Godunov--Powell source term in \eqref{eq:avgGL}. 
	Although the magnetic field satisfies a locally divergence-free (DF) condition after the DDF projection, divergence errors (i.e., jumps in the normal magnetic field component) may still exist across cell interfaces. The impact of these errors on the PP property can be controlled by incorporating the Godunov--Powell source term \cite{WuShu2018,WuShu2019}.
	
	Discretizing the Godunov--Powell source term requires careful treatment to exactly cancel the residual divergence effects on positivity preservation while preserving high-order accuracy. In this work, we adopt the discretization proposed in~\cite{DingWu2024SISCMHD}, which takes the form
	\begin{equation}\label{Dsource}
		\bm{S}_{i,j}(\hat{\bm{U}})=  \frac{1}{2\Delta x} \left(\mathcal{S}_{i+\halfone,j}+\mathcal{S}_{i-\halfone,j} \right) +\frac{1}{2\Delta y} \left(\mathcal{S}_{i,j+\halfone}+\mathcal{S}_{i,j-\halfone}\right), 
	\end{equation}
	where  
	\begin{equation*}
		\begin{aligned}
			\mathcal{S}_{i+\halfone,j} & :=  \frac{1}{6} [\![ \hat B_1]\!]_{i+\halfone, j+\halfone}^D \bm{S}\left(\average{\hat{\bm{U}}}_{i+\halfone, j+\halfone}^D\right) + \frac{4}{6} [\![\hat B_1]\!]_{i+\halfone, j} \bm{S}\left(\average{\hat{\bm{U}}}_{i+\halfone, j} \right) +\frac{1}{6} [\![ \hat B_1]\!]_{i+\halfone, j-\halfone}^U \bm{S}\left(\average{\hat{\bm{U}}}_{i+\halfone, j-\halfone}^U\right),\\
			\mathcal{S}_{i,j+\halfone} & :=  \frac{1}{6} [\![ \hat B_2]\!]_{i+\halfone, j+\halfone}^L \bm{S}\left(\average{\hat{\bm{U}}}_{i+\halfone, j+\halfone}^L\right) + \frac{4}{6} [\![\hat B_2]\!]_{i, j+\halfone} \bm{S}\left(\average{\hat{\bm{U}}}_{i, j+\halfone} \right) +\frac{1}{6} [\![\hat B_2]\!]_{i-\halfone, j+\halfone}^R \bm{S}\left(\average{\hat{\bm{U}}}_{i-\halfone, j+\halfone}^R \right).			
		\end{aligned}
	\end{equation*}

	\subsubsection{Proof of PP property for updated cell averages}

If the forward Euler method is used for the time discretization of the ODE system~\eqref{eq:avgGL}, then the cell averages are updated by
	\begin{equation}\label{EulerForwardCellAvg}
		\overline{{\bm U}}^{n+1}_{i,j} = \overline{{\bm U}}^{n}_{i,j} + \Delta t_n {\bm L}_{0,i,j}(\hat{{\bm U}}^n).
	\end{equation}
To achieve high-order temporal accuracy, one may adopt SSP Runge--Kutta methods; see \Cref{sec:time} for details. If the forward Euler scheme preserves positivity, then the SSP Runge--Kutta method, being a convex combination of forward Euler stages, also inherits the PP property.

The PP analysis of \eqref{EulerForwardCellAvg} is nontrivial because of the strong nonlinearity of the MHD system and the coupling of point values through the DDF condition. Fortunately, the geometric quasi-linearization (GQL) approach~\cite{Wu2018Positivity,Wu2023Geometric} reformulates the nonlinear admissible set \( \mathcal{G} \) into an equivalent linear form:
	\begin{equation}\label{GQL}
		\mathcal{G}_*=\left\{\bm{U}=(\rho, \boldsymbol{m}, \bm{B}, E)^{\top} \in \mathbb{R}^8: \bm{U} \cdot \bm{n}_1 >0, \bm{U} \cdot \bm{n}_*+\frac{\left|\bm{B}_*\right|^2}{2}>0, \quad \forall \boldsymbol{v}_*, \bm{B}_* \in \mathbb{R}^3\right\},
	\end{equation}
where $\bm{n}_1 := (1,0,0,0,0,0,0,0)^\top$ and $\bm{n}_* := \big(\tfrac{|\boldsymbol{v}_*|^2}{2}, -\boldsymbol{v}_*, -\bm{B}_*, 1\big)^{\top}$.
This converts the nonlinear constraint \( \mathcal{E}(\bm{U})>0 \) into a family of linear inequalities, thereby simplifying the analysis substantially.
With the properly designed numerical viscosity~\eqref{LFflux} and a suitable discretization of the Godunov--Powell source term~\eqref{Dsource}, the scheme~\eqref{EulerForwardCellAvg} can be rigorously proved PP by following the theoretical framework in \cite{Wu2018Positivity,WuShu2019,DingWu2024SISCMHD}.

Let $\overline{\mathcal U}$ and $\overline{\bm F}$ denote, respectively, the Simpson-rule edge averages of $\bm U$ and $\bm F$.  
For the edge $\{ x_{i+\frac{1}{2}} \} \times [y_{j-\frac{1}{2}}, y_{j+\frac{1}{2}}]$, set
	\begin{equation*}
		\begin{aligned}
			\mathcal{\overline{U}}_{i+\halfone,j}^{n,R} &:= \frac{1}{6} {\hat{\bm U}}^{n,RU}_{i+\halfone,j-\halfone} + \frac{4}{6} \hat{{\bm U}}^{n,R}_{i+\halfone,j} + \frac{1}{6} \hat{{\bm U}}^{n,RD}_{i+\halfone,j+\halfone}, \quad
			\mathcal{\overline{U}}_{i+\halfone,j}^{n,L} := \frac{1}{6} \hat{{\bm U}}^{n,LU}_{i+\halfone,j-\halfone} + \frac{4}{6} \hat{{\bm U}}^{n,L}_{i+\halfone,j} + \frac{1}{6} \hat{{\bm U}}^{n,LD}_{i+\halfone,j+\halfone}, \\
			\overline{\bm{F}}_{1,i+\frac{1}{2},j}^{n,R} & := \frac{1}{6} \bm{F}_1(\hat{\bm{U}}^{n,RU}_{i+\frac{1}{2},j-\frac{1}{2}}) + \frac{4}{6} \bm{F}_1(\hat{\bm{U}}^{n,R}_{i+\frac{1}{2},j}) + \frac{1}{6} \bm{F}_1(\hat{\bm{U}}^{n,RD}_{i+\frac{1}{2},j+\frac{1}{2}}), \\
			\overline{\bm{F}}_{1,i+\frac{1}{2},j}^{n,L} & := \frac{1}{6} \bm{F}_1(\hat{\bm{U}}^{n,LU}_{i+\frac{1}{2},j-\frac{1}{2}}) + \frac{4}{6} \bm{F}_1(\hat{\bm{U}}^{n,L}_{i+\frac{1}{2},j}) + \frac{1}{6} \bm{F}_1(\hat{\bm{U}}^{n,LD}_{i+\frac{1}{2},j+\frac{1}{2}}),		
		\end{aligned}	
	\end{equation*}
and similarly in the $y$-direction. Then \eqref{EulerForwardCellAvg} can be rewritten as
	\begin{equation}\label{cellDecom}
		\begin{aligned}
			\overline{{\bm U}}^{n+1}_{i,j} &= \overline{{\bm U}}^{n}_{i,j} 
			+ \Delta t_n \bm{S}_{i,j}(\hat{\bm{U}}^n) 
			+ \frac{1}{2} \kappa_1 \left( \mathcal{\overline{U}}_{i+\frac{1}{2},j}^{n,R} + \mathcal{\overline{U}}_{i-\frac{1}{2},j}^{n,L} - \mathcal{\overline{U}}_{i+\frac{1}{2},j}^{n,L} - \mathcal{\overline{U}}_{i-\frac{1}{2},j}^{n,R}  \right) \\ 
			&+ \frac{1}{2} \kappa_2 \left( \mathcal{\overline{U}}_{i,j+\frac{1}{2}}^{n,U} + \mathcal{\overline{U}}_{i,j-\frac{1}{2}}^{n,D} - \mathcal{\overline{U}}_{i,j+\frac{1}{2}}^{n,D} - \mathcal{\overline{U}}_{i,j-\frac{1}{2}}^{n,U}  \right) 
			+ \bm{N}_f,
		\end{aligned}
	\end{equation}
where $\kappa_1 := \frac{\alpha_1^{\mathrm{LF}} \Delta t_n}{\Delta x}$, $\kappa_2 := \frac{\alpha_2^{\mathrm{LF}} \Delta t_n}{\Delta y}$, and
	\begin{equation*}
		\bm{N}_f = -\frac{1}{2} \frac{\Delta t_n}{\Delta x} \left( \overline{\bm{F}}_{1,i+\frac{1}{2},j}^{n,L} - \overline{\bm{F}}_{1,i-\frac{1}{2},j}^{n,R} + \overline{\bm{F}}_{1,i+\frac{1}{2},j}^{n,R} - \overline{\bm{F}}_{1,i-\frac{1}{2},j}^{n,L} \right)  -\frac{1}{2} \frac{\Delta t_n}{\Delta y} \left( \overline{\bm{F}}_{2,i,j+\frac{1}{2}}^{n,D} - \overline{\bm{F}}_{2,i,j-\frac{1}{2}}^{n,U} + \overline{\bm{F}}_{2,i,j+\frac{1}{2}}^{n,U} - \overline{\bm{F}}_{2,i,j-\frac{1}{2}}^{n,D} \right).
	\end{equation*} 
	
Taking the dot product of \eqref{cellDecom} with $\bm n_1$, we aim to show $\overline{\bm U}^{\,n+1}_{i,j}\!\cdot\!\bm n_1>0$ by carefully estimating the right-hand side. First, note that the first component of $\bm{S}_{i,j}(\hat{\bm U}^{\,n})$ vanishes. Moreover, the standard estimate
	\begin{equation}\label{eq1}
		-\left( \bm{F}_{\ell} (\bm{U}) - \bm{F}_{\ell} (\widetilde{\bm{U}}) \right) \cdot \bm{n}_1 
		> -\alpha_\ell ({\bm U}, \widetilde{\bm{U}})\, ({\bm U} + \widetilde{\bm{U}}) \cdot \bm{n}_1, 
		\quad \ell = 1, 2,
	\end{equation}
holds for all $\bm{U}, \widetilde{\bm{U}} \in \mathcal{G}$.  
Combining \eqref{eq1} with the cell decomposition \eqref{eq:CAD3}, we obtain
	\begin{equation}\label{GQL1}
		\begin{aligned}
			\overline{\bm{U}}^{n+1}_{i,j} \cdot \bm{n}_1 >
			\left( \frac{1}{6} - \kappa \right) 
			\bigg[ \frac{\kappa_1}{\kappa} \left( \mathcal{\overline{U}}_{i-\halfone,j}^{n,R} + \mathcal{\overline{U}}_{i+\halfone,j}^{n,L} \right) \cdot \bm{n}_1 
			+ \frac{\kappa_2}{\kappa} \left( \mathcal{\overline{U}}_{i,j-\halfone}^{n,U} + \mathcal{\overline{U}}_{i,j+\halfone}^{n,D} \right) \cdot \bm{n}_1 \bigg] > 0,
		\end{aligned}
	\end{equation}
where $\kappa := \kappa_1 + \kappa_2$. The last inequality is ensured by the CFL condition~\eqref{eq:PPCFL}.

Next, we take the dot product of \eqref{cellDecom} with $\bm n_*$ and then add $\frac{|\bm B_*|^2}{2}$ to verify that
$\overline{\bm{U}}^{\,n+1}_{i,j} \cdot \bm{n}_* + \frac{|\bm{B}_*|^2}{2} > 0$.
For this estimate we use the two key inequalities \cite{Wu2018Positivity,WuShu2018}: 
		\begin{align} \label{eq:WU1}
			-\left( \bm{F}_{\ell} (\bm{U}) - \bm{F}_{\ell} (\widetilde{{\bm U}}) \right) \cdot \bm{n}_*  &\geq -\alpha_\ell({\bm U}, \widetilde{{\bm U}}) \left( ({\bm U}+ \widetilde{{\bm U}}) \cdot \bm{n}_* + \left| {\bm B}_* \right|^2 \right) - \left(B_\ell - \tilde{B}_\ell \right) \left( {\bm v}_* \cdot {\bm B}_* \right),  \\
			-\xi\left(\bm{S}(\bm{U}) \cdot \bm{n}_*\right) &\geq \xi\left(\bm{v}_* \cdot \bm{B}_*\right) - \frac{|\xi|}{\sqrt{\rho}} \left( \bm{U} \cdot \bm{n}_* + \frac{\left| \bm{B}_* \right|^2}{2} \right), \label{eq:WU2}
		\end{align}
which hold for all $\bm U, \widetilde{\bm U} \in \mathcal G$, $\xi \in \mathbb R$, and $\ell=1,2$.
From \eqref{eq:WU1}, the flux term $\bm N_f$ satisfies
	\begin{equation}\label{estimate:Nf}
		\begin{aligned}
			\bm{N}_f \cdot \bm{n}_* \geq 
			& -\frac{1}{2} \frac{\hat{\alpha}_1^{LF} \Delta t}{\Delta x}\left[\left(\mathcal{\overline{U}}_{i+\frac{1}{2}, j}^{n,R}+\mathcal{\overline{U}}_{i-\frac{1}{2}, j}^{n,L}+\mathcal{\overline{U}}_{i+\frac{1}{2}, j}^{n,L}+\mathcal{\overline{U}}_{i-\frac{1}{2}, j}^{n,R}\right) \cdot \bm{n}_*+2\left|\bm{B}_*\right|^2\right] \\
			& -\frac{1}{2} \frac{\hat{\alpha}_2^{LF} \Delta t}{\Delta x}\left[\left(\mathcal{\overline{U}}_{i, j+\frac{1}{2}}^{n,U}+\mathcal{\overline{U}}_{i, j-\frac{1}{2}}^{n,D}+\mathcal{\overline{U}}_{i, j+\frac{1}{2}}^{n,D}+\mathcal{\overline{U}}_{i, j-\frac{1}{2}}^{n,U}\right) \cdot \bm{n}_*+2\left|\bm{B}_*\right|^2\right] -\Delta t\left(\nabla_{i,j} \cdot \average{\hat{\bm{B}}^n} \right)\left(\bm{v}_* \cdot \bm{B}_*\right),
		\end{aligned}
	\end{equation}
	where 
	\begin{equation*}
		\begin{aligned}
			\nabla_{i,j} \cdot \average{\hat{\bm{B}}^n} &:= \frac{1}{6} ( \frac{\average{\hat{B}_1^n}_{i+\halfone,j+\halfone}^{D} - \average{\hat{B}_1^n}_{i-\halfone,j+\halfone}^{D}}{\Delta x}  +\frac{\average{\hat{B}_2^n}_{i+\halfone,j+\halfone}^{L} - \average{\hat{B}_2^n}_{i+\halfone,j-\halfone}^{L}}{\Delta y} ) \\
			&+ \frac{4}{6} ( \frac{\average{\hat{B}_1^n}_{i+\halfone,j} - \average{\hat{B}_1^n}_{i-\halfone,j}}{\Delta x} +\frac{\average{\hat{B}_2^n}_{i,j+\halfone} - \average{\hat{B}_2^n}_{i,j-\halfone}}{\Delta y} )\\
			&+ \frac{1}{6} (\frac{\average{\hat{B}_1^n}_{i+\halfone,j-\halfone}^U - \average{\hat{B}_1^n}_{i-\halfone,j-\halfone}^U}{\Delta x} +\frac{\average{\hat{B}_2^n}_{i-\halfone,j+\halfone}^R - \average{\hat{B}_2^n}_{i-\halfone,j-\halfone}^R}{\Delta y}).
		\end{aligned}
	\end{equation*}
Using \eqref{eq:WU2} for the source term $\bm{S}_{i,j}(\hat{\bm U}^{\,n})$ gives
	$$
	\begin{aligned}
		-[\![ \hat{B}_{1}^n ]\!]_{i+\frac{1}{2}, j} \bm{S}\left(\average{\hat{\bm{U}}^n}_{i+\frac{1}{2}, j}\right) \cdot \bm{n}_{*} &\geq [\![ \hat{B}_{1}^n ]\!]_{i+\frac{1}{2}, j}\left(\bm{v}_{*} \cdot \bm{B}_{*}\right) -\frac{\left|[\![ \hat{B}_{1}^n ]\!]_{i+\frac{1}{2}, j}\right|}{\sqrt{\average{\hat{\rho}^n}_{i+\frac{1}{2}, j}}}\left(\average{\hat{\bm{U}}^n}_{i+\frac{1}{2}, j} \cdot \bm{n}_{*}+\frac{\left| \bm{B}_* \right|^2}{2}\right) \\
		&\geq [\![ \hat{B}_{1}^n ]\!]_{i+\frac{1}{2}, j}\left(\bm{v}_{*} \cdot \bm{B}_{*}\right)-2 \beta_{1}\left(\average{\hat{\bm{U}}^n}_{i+\frac{1}{2}, j} \cdot \bm{n}_{*}+\frac{\left| \bm{B}_* \right|^2}{2}\right) .
	\end{aligned}
	$$
For the specially designed discretization of the Godunov--Powell source term \eqref{Dsource}, we obtain
	\begin{equation}\label{estimate:S}
		\begin{aligned}
			\Delta t_n \bm{S}_{i,j} (\hat{\bm{U}}^n) \cdot \bm{n}_* \geq 
			& -\frac{\beta_1\Delta t}{2\Delta x}\left[\left(\mathcal{\overline{U}}_{i+\frac{1}{2},j}^{n,R} + \mathcal{\overline{U}}_{i-\frac{1}{2},j}^{n,L} + \mathcal{\overline{U}}_{i+\frac{1}{2},j}^{n,L} + \mathcal{\overline{U}}_{i-\frac{1}{2},j}^{n,R}\right)\cdot\bm{n}_* + 2\left| \bm{B}_* \right|^2\right] \\
			& -\frac{\beta_2\Delta t}{2\Delta y}\left[\left(\mathcal{\overline{U}}_{i,j+\frac{1}{2}}^{n,U} + \mathcal{\overline{U}}_{i,j-\frac{1}{2}}^{n,D} + \mathcal{\overline{U}}_{i,j+\frac{1}{2}}^{n,D} + \mathcal{\overline{U}}_{i,j-\frac{1}{2}}^{n,U}\right)\cdot\bm{n}_* + 2\left| \bm{B}_* \right|^2\right]  +\frac{\Delta t}{2}[\![\hat{\bm{B}}^n]\!]_{i,j}(\bm{v}_*\cdot\bm{B}_*),
		\end{aligned}
	\end{equation}
	where 
	\begin{equation*}
		\begin{aligned}
			[\![\hat{\bm{B}}^n]\!]_{i,j} &:= \frac{1}{6} ( \frac{[\![\hat{B}_1^n]\!]_{i+\halfone,j+\halfone}^{D} + [\![\hat{B}_1^n]\!]_{i-\halfone,j+\halfone}^{D}}{\Delta x}  +\frac{[\![\hat{B}_2^n]\!]_{i+\halfone,j+\halfone}^{L} + [\![\hat{B}_2^n]\!]_{i+\halfone,j-\halfone}^{L}}{\Delta y} ) \\
			&+ \frac{4}{6} (\frac{[\![\hat{B}_1^n]\!]_{i+\halfone,j} + [\![\hat{B}_1^n]\!]_{i-\halfone,j}}{\Delta x} +\frac{[\![\hat{B}_2^n]\!]_{i,j+\halfone} + [\![\hat{B}_2^n]\!]_{i,j-\halfone}}{\Delta y})\\
			&+ \frac{1}{6} (\frac{[\![\hat{B}_1^n]\!]_{i+\halfone,j-\halfone}^U + [\![\hat{B}_1^n]\!]_{i-\halfone,j-\halfone}^U}{\Delta x} +\frac{[\![\hat{B}_2^n]\!]_{i-\halfone,j+\halfone}^R + [\![\hat{B}_2^n]\!]_{i-\halfone,j-\halfone}^R}{\Delta y}).
		\end{aligned} 
	\end{equation*}
Combining \eqref{cellDecom}, \eqref{estimate:Nf}, and \eqref{estimate:S} yields
	\begin{equation}\label{GQL2}
		\begin{aligned}
			\overline{\bm{U}}^{n+1}_{i,j} \cdot \bm{n}_*+\frac{\left|\bm{B}_*\right|^2}{2} 
			&>  \Delta t_n \bm{S}_{i,j}(\hat{\bm{U}}^n) \cdot \bm{n}_*
			+ \frac{1}{2} \kappa_1 \left( \mathcal{\overline{U}}_{i+\frac{1}{2},j}^{n,R} + \mathcal{\overline{U}}_{i-\frac{1}{2},j}^{n,L} - \mathcal{\overline{U}}_{i+\frac{1}{2},j}^{n,L} - \mathcal{\overline{U}}_{i-\frac{1}{2},j}^{n,R}  \right) \cdot \bm{n}_* \\ 
			&+ \frac{1}{2} \kappa_2 \left( \mathcal{\overline{U}}_{i,j+\frac{1}{2}}^{n,U} + \mathcal{\overline{U}}_{i,j-\frac{1}{2}}^{n,D} - \mathcal{\overline{U}}_{i,j+\frac{1}{2}}^{n,D} - \mathcal{\overline{U}}_{i,j-\frac{1}{2}}^{n,U}  \right) \cdot \bm{n}_*
			+ \bm{N}_f \cdot \bm{n}_* \\
			& + \frac{1}{6 \kappa}\left[\kappa_1\left(\mathcal{\overline{U}}_{i-\frac{1}{2}, j}^{n,R}+\mathcal{\overline{U}}_{i+\frac{1}{2}, j}^{n,L}\right)+\kappa_2\left(\mathcal{\overline{U}}_{i, j-\frac{1}{2}}^{n,U}+\mathcal{\overline{U}}_{i, j+\frac{1}{2}}^{n,D}\right)\right] \cdot \bm{n}_*+\frac{1}{6} \left|\bm{B}_*\right|^2 \\
			&\geq  \mathcal{D}  +\Delta t_n \left(\bm{v}_* \cdot \bm{B}_*\right)\left(\frac{1}{2} [\![\hat{\bm{B}}^n]\!]_{i,j} -   \nabla_{i,j} \cdot \average{\hat{\bm{B}}^n}\right) > \Delta t_n ( \bm{v}_* \cdot \bm{B}_* ) \left( \frac{1}{2} [\![\hat{\bm{B}}^n]\!]_{i,j} -   \nabla_{i,j} \cdot \average{\hat{\bm{B}}^n} \right) \\
			&= \Delta t_n ( \bm{v}_* \cdot \bm{B}_* ) \left( -\nabla_{i,j} \cdot \hat{\bm{B}}^n \right) = 0,
		\end{aligned}
	\end{equation}
where the first inequality follows from \eqref{eq:PPcon} and the positivity of
	\begin{align*}
		\mathcal{D} &= \frac{1}{2}\left(\kappa_1-\left(\hat{\alpha}_1^{LF}+\beta_1\right) \frac{\Delta t_n}{\Delta x}\right)\left[\left(\mathcal{\overline{U}}_{i+\frac{1}{2}, j}^{n, R}+\mathcal{\overline{U}}_{i-\frac{1}{2}, j}^{n,L}\right) \cdot \bm{n}_*+\left|\bm{B}_*\right|^2\right] \\
		& +\frac{1}{2}\left(\kappa_2-\left(\hat{\alpha}_2^{LF}+\beta_2\right) \frac{\Delta t_n}{\Delta y}\right)\left[\left(\mathcal{\overline{U}}_{i, j+\frac{1}{2}}^{n,U}+\mathcal{\overline{U}}_{i, j-\frac{1}{2}}^{n,D}\right) \cdot \bm{n}_*+\left|\bm{B}_*\right|^2\right] \\
		& +\frac{1}{2}\left(\frac{ \kappa_1}{3 \kappa}-\left(\kappa_1+\left(\hat{\alpha}_1^{LF}+\beta_1\right) \frac{\Delta t_n}{\Delta x}\right)\right)\left[\left(\mathcal{\overline{U}}_{i+\frac{1}{2}, j}^{n,L}+\mathcal{\overline{U}}_{i-\frac{1}{2}, j}^{n,R}\right) \cdot \bm{n}_*+\left|\bm{B}_*\right|^2\right] \\
		& +\frac{1}{2}\left(\frac{ \kappa_2}{3 \kappa}-\left(\kappa_2+\left(\hat{\alpha}_2^{LF}+\beta_2\right) \frac{\Delta t_n}{\Delta y}\right)\right)\left[\left(\mathcal{\overline{U}}_{i, j+\frac{1}{2}}^{n,D}+\mathcal{\overline{U}}_{i, j-\frac{1}{2}}^{n,U}\right) \cdot \bm{n}_*+\left|\bm{B}_*\right|^2\right] > 0,
	\end{align*}
which is guaranteed by the CFL condition \eqref{eq:PPCFL}.
We emphasize that the last term in \eqref{estimate:S} is crucial: it exactly cancels the additional magnetic-field contribution arising in the flux estimate under the DDF condition \eqref{key11122}, ensured by \textbf{Step~1}.  
With the GQL representation \eqref{GQL}, together with \eqref{GQL1} and \eqref{GQL2}, we obtain the following theorem.

	\begin{theorem}[PP Cell Average Update]\label{them:PPcfl}
Assume that $\overline{\bm U}^{\,n}_{i,j} \in \mathcal G$ and conditions \eqref{eq:PPcon}--\eqref{eq:CAD3} hold.  
Then the updated cell averages produced by \eqref{EulerForwardCellAvg} are PP; i.e.,
		\begin{equation}\label{eq:PPscheme}
			\overline{{\bm U}}^{n+1}_{i,j} = \overline{{\bm U}}^{n}_{i,j} + \Delta t_n {\bm L}_{0,i,j}(\hat{{\bm U}}^n) \in \mathcal{G}, \quad \forall\, i, j,
		\end{equation}
		under the CFL condition
		\begin{equation}\label{eq:PPCFL}
			0<\Delta t_n \left( \frac{\alpha_1^{LF}}{\Delta x}+\frac{\alpha_2^{LF}}{\Delta y}\right)<\frac{1}{6}.
		\end{equation}
	\end{theorem}

	\subsection{A Lax-type entropy criterion as shock indicator}\label{trouble-cell}

	In this subsection, we introduce a simple yet effective indicator to identify troubled cells near discontinuities. 
	The COE procedure in \Cref{OE} will be applied selectively to these identified cells rather than all cells, reducing both computational cost and numerical dissipation. 
	Inspired by~\cite{LiuFeng2021}, we adopt a Lax-type entropy criterion based on wave speed jumps...
	
	The fast magneto-acoustic speed in the $x$- or $y$-direction is defined as
	\begin{equation}\label{def:cf2}
		c_{f,\ell}( {\bm U} )= \sqrt{\frac{1}{2} \left( \frac{\gamma {p} + |{\bm{B}}|^2}{{\rho}} + \sqrt{ \left( \frac{\gamma {p} + |{\bm{B}}|^2}{{\rho}} \right)^2 - 4 \frac{\gamma {p} ({B_{\ell}})^2}{{\rho}^2} } \right)}, \quad \ell =1,2.
	\end{equation}
	
	This indicator offers two key advantages:
	\begin{itemize}
		\item \textbf{Low computational cost:} Only two characteristic speeds are computed in each direction.
		\item \textbf{Symmetry preservation:} The indicator preserves symmetry by detecting troubled cells in a symmetric fashion.
	\end{itemize}

	\begin{remark}
		The MHD system admits eight characteristic fields. However, our results demonstrate that using only the fast wave speeds is sufficient for reliable detection of discontinuities. Including additional speeds tends to mark unnecessary cells as troubled, without yielding noticeable improvements (see \Cref{Ex:ShockCloud}).
	\end{remark}

	\subsection{COE: Convex Oscillation Elimination} \label{OE}
	
	High-order PAMPA schemes are prone to nonphysical oscillations near discontinuities, especially in the presence of strong shocks or contact discontinuities. To robustly suppress such artifacts while preserving high-order accuracy in smooth regions, we introduce a COE approach. The COE procedure acts as a convex blending limiter, applicable either globally or locally—restricted to cells flagged as troubled by a shock indicator.

	To describe the COE procedure clearly, we begin by defining its key components:
	
	\begin{itemize}
		\item Special $\star$-Norm: To measure the distance between states, we define a new squared norm:
		\begin{equation}
			\label{eq:star-norm}
			\begin{aligned}
				\| \bm{U} - \check{\bm{U}} \|_{\star}^2 &:= (E - \check{E})^2 + \frac{1}{8} \left( |\bm{v}|^4 + |\check{\bm{v}}|^4 \right)(\rho - \check{\rho})^2 \\
				&\quad + \frac{1}{8} \sum_{\ell=1}^{3} (v_\ell^2 + \check{v}_\ell^2)(m_\ell - \check{m}_\ell)^2 + \frac{1}{8} \sum_{\ell=1}^{3} (B_\ell^2 + \check{B}_\ell^2)(B_\ell - \check{B}_\ell)^2.
			\end{aligned}
		\end{equation}

		\item Maximum Local Wave Speed: To ensure the evolution invariance property introduced later, 
		For each spatial direction $\ell \in \{1,2\}$ (corresponding to $x$ and $y$), we define:
		\begin{equation}
			\label{eq:lambda}
			\lambda_{\ell,i,j} := |\overline{v}_{\ell,i,j}| + c_{f,\ell}(\overline{\bm{U}}_{i,j}),
		\end{equation}
		where $c_{f,\ell}$ is the fast magnetoacoustic speed in direction $\ell$ as defined in \eqref{def:cf2}, and $\overline{v}_{\ell,i,j} = \overline{m}_{\ell,i,j} / \overline{\rho}_{i,j}$.

	\end{itemize}
	
	The COE procedure comprises the following steps:
	
	\paragraph{Step 1: Normalization Factor}
	We introduce a global scaling factor to normalize discrepancy measures:
	\[
	\Theta := \max_{i,j} \Theta_{i,j}, \qquad
	\Theta_{i,j} := \left( \int_{I_{i,j}} \left\| \check{\bm{P}}_{i,j}(x,y) - \overline{\bm{U}}^\Omega \right\|_{\star}^2 \,\mathrm{d}x\,\mathrm{d}y \right)^{1/2},
	\]
	where $\overline{\bm{U}}^\Omega := \frac{1}{N_x N_y} \sum_{i,j} \overline{\bm{U}}_{i,j}$ denotes the spatial average over the entire domain.
	
	\paragraph{Step 2: Directional Damping Coefficients}
	We compute directional damping coefficients $\sigma_{1,i,j}$ and $\sigma_{2,i,j}$ as follows:

	\paragraph{Step 3: Convex Blending of States}
	For each state $\check{\bm{U}} \in \check{\mathbb{U}}_{i,j} \cup \check{\bm{U}}^c_{i,j}$, we define the damped state as a convex combination:
	\begin{equation}
		\hat{\bm{U}} := \theta^{\mathrm{OE}}_{i,j} \left(\check{\bm{U}} - \overline{\bm{U}}_{i,j}\right) + \overline{\bm{U}}_{i,j},
	\end{equation}
	where the damping factor is
	\[
	\theta^{\mathrm{OE}}_{i,j} := \exp\left( -C_0 \left[ \sigma_{1,i,j} \frac{\Delta t_n}{\Delta x} + \sigma_{2,i,j} \frac{\Delta t_n}{\Delta y} \right] \right).
	\]
	The scalar $C_0 > 0$ controls the overall strength of dissipation and is chosen empirically. Importantly, due to the scale-invariance (\Cref{thm:scale}) and evolution-invariance (\Cref{thm:evol}) of the COE mechanism, $C_0$ is not problem-specific. In all numerical experiments presented in this paper, we fix $C_0 = 43$.
	
	\begin{remark}
		All cell integrals involved in the COE procedure are approximated using Gauss-type tensor product quadrature. For instance, the directional discrepancy $\eta^{L}_{i,j}$ is computed as
		\[
		\eta^{L}_{i,j} \approx \left( \sum_{\mu=1}^q \sum_{\nu=1}^q \omega_\mu \omega_\nu \left\| \check{\bm{P}}_{i,j}(x_\mu, y_\nu) - \check{\bm{P}}^{L}_{i,j}(x_\mu, y_\nu) \right\|_{\star}^2 \right)^{1/2},
		\]
		where $\{\omega_\mu\}$ are the quadrature weights and $\{(x_\mu, y_\nu)\}$ are the Gauss points in $I_{i,j}$.
	\end{remark}


	Although the proposed COE procedure is conceptually distinct from the original OE method introduced in \cite{peng2023oedg}—which operates via modal filtering—the COE framework retains two critical structural properties: {\it scale invariance} and {\it evolution invariance}. These invariance properties play a key role in ensuring that the scheme behaves consistently under changes of physical units or time rescaling, and that it preserves important symmetries inherent to the governing MHD system.
	
	\subsubsection{Scale Invariance}
	
	We begin by establishing the invariance of the COE procedure under scaling transformations of the conservative variables. Let $\mu > 0$ and define the block-diagonal scaling matrix
	\begin{equation}
		P := \begin{bmatrix}
			\mu \bm{I}_4 & & \\
			& \sqrt{\mu} \bm{I}_3 & \\
			& & \mu
		\end{bmatrix},
	\end{equation}
	which induces the transformation $\bm{U} \mapsto P \bm{U}$, scaling density, momentum, magnetic field, and energy consistently with the known scaling laws of ideal MHD.
	
	\begin{theorem}[Scale Invariance] \label{thm:scale}
		Let $P$ be defined as above. Then for each spatial direction $\ell = 1, 2$, the directional damping coefficient $\sigma_{\ell, i,j}$ is invariant under the scaling $\bm{U} \mapsto P\bm{U}$, i.e.,
		\begin{equation}
			\sigma_{\ell, i,j}(P\bm{U}) = \sigma_{\ell, i,j}(\bm{U}).
		\end{equation}
	\end{theorem}
	
	\begin{proof}
		The result follows from two key observations:
		
		(i) The $\star$-norm defined in \eqref{eq:star-norm} satisfies quadratic homogeneity under the scaling transformation:
		\[
		\| P\bm{U}_1 - P\bm{U}_2 \|^2_{\star} = \mu^2 \| \bm{U}_1 - \bm{U}_2 \|^2_{\star}.
		\]
		
		(ii) The wave speeds defined by \eqref{eq:lambda} are invariant under scaling:
		\[
		\lambda_{\ell}(P\bm{U}) = \lambda_{\ell}(\bm{U}), \qquad \ell = 1, 2.
		\]
		
		Since both the wave speeds and the discrepancy terms scale homogeneously and consistently in numerator and denominator, the damping ratio $\sigma_{\ell, i,j}$ remains unchanged.
	\end{proof}
	
	This scale-invariant design confers several important benefits. First, it guarantees that the damping coefficient $\sigma_{\ell, i,j}$ is dimensionless and thus independent of the units chosen to represent physical quantities (e.g., grams vs. kilograms). Consequently, the strength of the COE damping effect remains consistent across problems of varying physical scales.
	
	More importantly, the COE procedure inherits and preserves the scaling symmetry of the continuous MHD system. Let $\mathcal{S}_t$ denote the exact solution operator of the MHD equations, i.e., $\bm{U}(\bm{x}, t) = \mathcal{S}_t(\bm{U}_0)$, and let $\mathcal{P}$ denote the scaling operator associated with matrix $P$. It is well known that the ideal MHD equations satisfy the commutative scaling relation:
	\[
	\mathcal{S}_t \mathcal{P} = \mathcal{P} \mathcal{S}_t.
	\]
	Thanks to Theorem~\ref{thm:scale}, the numerical solution operator $\mathcal{E}_n$ of the COE-based PAMPA scheme also satisfies a discrete analog:
	\[
	\mathcal{E}_n \mathcal{P} = \mathcal{P} \mathcal{E}_n.
	\]
	This demonstrates that the numerical method respects the scale-commutative structure of the underlying physical system.
	
	\begin{remark}
		The admissible set $\mathcal{G}$ for the MHD system is invariant under the transformation
		\[
		\bm{U}_\mu := (\mu \rho, \mu \bm{m}, \sqrt{\mu} \bm{B}, \mu E)^\top, \qquad \mu > 0.
		\]
		This reflects the intrinsic scale invariance of the MHD equations; see, e.g., \cite{WuTangM3AS}.
	\end{remark}
	
	\subsubsection{Evolution Invariance}
	
	The second important structural property concerns invariance under time-rescaling of the governing equations. Consider the family of modified MHD systems
	\begin{equation}
		\label{eq:muMHD}
		\bm{U}_t + \nabla \cdot \left( \mu \bm{F}(\bm{U}) \right) = 0, \qquad \mu > 0,
	\end{equation}
	which may be interpreted as a time-rescaled version of the original system. Let $\mathcal{S}_t^\mu$ denote the corresponding exact solution operator. It is straightforward to verify that
	\[
	\mathcal{S}_t^\mu = \mathcal{S}_{\mu t}, \qquad \forall\, t > 0,\ \mu > 0.
	\]
	This identity expresses the fact that increasing the flux strength $\mu$ is equivalent to accelerating the time evolution by a factor of $\mu$.
	
	We now show that the COE-based scheme preserves this invariance at the discrete level.
	
	\begin{theorem}[Evolution Invariance] \label{thm:evol}
		Let $\mathcal{E}_n^\mu$ denote the numerical solution operator applied to \eqref{eq:muMHD} using time step $\Delta t_\mu = \Delta t / \mu$, such that the CFL number remains fixed. Then for any initial condition $\bm{U}^0_h$, we have
		\begin{equation}
			\mathcal{E}_n^\mu(\bm{U}^0_h) = \mathcal{E}_n(\bm{U}^0_h), \qquad \forall\, \mu > 0.
		\end{equation}
	\end{theorem}
	
	\begin{proof}
		Let $\lambda_{i,j}^\mu$, $\sigma_{\ell,i,j}^\mu$, and $\theta_{i,j}^{\mathrm{OE},\mu}$ denote the local wave speed, damping coefficient, and blending factor used in the COE procedure when solving \eqref{eq:muMHD}. Note that
		\[
		\lambda_{i,j}^\mu = \mu \lambda_{i,j}, \qquad \sigma_{\ell, i,j}^\mu = \mu \sigma_{\ell, i,j}.
		\]
		However, the effective argument of the exponential blending factor becomes
		\[
		\sigma_{\ell, i,j}^\mu \cdot \frac{\Delta t_\mu}{\Delta x} = \mu \sigma_{\ell, i,j} \cdot \frac{\Delta t}{\mu \Delta x} = \sigma_{\ell, i,j} \cdot \frac{\Delta t}{\Delta x},
		\]
		and similarly for the $y$-direction. Thus,
		\[
		\theta_{i,j}^{{OE},\mu} = \theta_{i,j}^{{OE}}.
		\]
		Therefore, the COE modification remains unchanged under variation of $\mu$, and the discrete evolution operator $\mathcal{E}_n^\mu$ produces the same solution as $\mathcal{E}_n$, establishing the claim.
	\end{proof}
	
	This invariance result confirms that the COE procedure exhibits consistent dynamical behavior for problems with differing characteristic speeds, provided that the time step is scaled accordingly to maintain a fixed CFL condition. This is particularly important in MHD simulations where physical regimes may span several orders of magnitude in wave speeds.

	\subsection{High-order time discretization}\label{sec:time}
	A high-order SSP Runge–Kutta method can be employed for the time discretization of the ODE systems governing the point values, namely~\eqref{eq:pointevolve1}, \eqref{eq:pointevolve2}, and \eqref{eq:pointevolve3}, as well as the cell average update~\eqref{eq:avgGL}.
	For instance, the third-order SSP Runge–Kutta method is given by:
	\begin{equation*}
		\begin{aligned}
			{\bm q}^{(1)}_{0,i,j} &= {\bm q}^{n}_{0,i,j} + \Delta t_n {\bm L}_{0,i,j}(\hat{{\bm U}}^n),  
			& {\bm q}^{(1)}_{r,i,j} &= {\bm q}^{n}_{r,i,j} + \Delta t_n {\bm L}_{r,i,j}({\bm W}^n, \overline{{\bm U}}^n, \hat{{\bm U}}^n), \\
			{\bm q}^{(2)}_{0,i,j} &= \frac{3}{4}{\bm q}^{n}_{0,i,j} + \frac{1}{4}  ({\bm q}^{(1)}_{0,i,j} + \Delta t_n {\bm L}_{0,i,j}(\hat{{\bm U}}^{(1)})),  
			& {\bm q}^{(2)}_{r,i,j} &= \frac{3}{4}{\bm q}^{(1)}_{r,i,j} + \frac{1}{4} ( {\bm q}^{(1)}_{r,i,j}  + \Delta t_n {\bm L}_{r,i,j}({\bm W}^{(1)}, \overline{{\bm U}}^{(1)}, \hat{{\bm U}}^{(1)})), \\
			{\bm q}^{n+1}_{0,i,j} &= \frac{1}{3}{\bm q}^{n}_{0,i,j} + \frac{2}{3}  ({\bm q}^{(2)}_{0,i,j} + \Delta t_n {\bm L}_{0,i,j}(\hat{{\bm U}}^{(2)})),  
			& {\bm q}^{n+1}_{r,i,j} &= \frac{1}{3}{\bm q}^{n}_{r,i,j} + \frac{2}{3} ( {\bm q}^{(2)}_{r,i,j}  + \Delta t_n {\bm L}_{r,i,j}({\bm W}^{(2)}, \overline{{\bm U}}^{(2)}, \hat{{\bm U}}^{(2)})),  
		\end{aligned}
	\end{equation*}
	where $r = 1, \dots, 8$.

	\section{Numerical Experiments}\label{S.5}
	
	This section presents several numerical examples to demonstrate the accuracy, effectiveness, and robustness of our structure-preserving PAMPA method on uniform Cartesian meshes. Unless otherwise specified, the adiabatic index is set to $\gamma = \frac{5}{3}$. For time integration, we employ the third-order SSP Runge--Kutta method described in \Cref{sec:time}, using a CFL number of $C_{\rm CFL} = 0.1$.

	We begin by verifying the accuracy of the proposed scheme using a smooth solution, followed by a series of benchmark problems and challenging test cases to demonstrate its effectiveness and robustness.
	
	\begin{expl}[Alfv\'{e}n Wave]\label{EX:Alfen} \rm
		This test involves the propagation of a circularly polarized Alfv\'{e}n wave traveling at a constant Alfv\'{e}n speed. The exact solution remains smooth for all time, making it well-suited for assessing the convergence properties of the numerical methods for MHD. The initial conditions are specified as
		\begin{align*}
			&\rho = 1, \quad v_1 = -0.1\sin\alpha\sin(2\pi\beta), \quad v_2 = 0.1\cos\alpha\sin(2\pi\beta), \quad v_3 = 0.1\cos(2\pi\beta), \\
			&B_1 = \cos\alpha + v_1, \quad B_2 = \sin\alpha + v_2, \quad B_3 = v_3, \quad p = 0.1,
		\end{align*}
		where $\alpha = \pi/4$ and $\beta = x\cos\alpha + y\sin\alpha$. 
		The computational domain is $[0, 1/\cos\alpha] \times [0, 1/\sin\alpha]$, and the wave returns to its initial configuration at every integer time $t$. The domain is discretized using an $N \times N$ uniform Cartesian mesh with periodic boundary conditions imposed in both directions. 
		We compute the $l^1$, $l^2$, and $l^\infty$ errors of the velocity field $(v_1, v_2)$ and magnetic field $(B_1, B_2)$ at $t = 1$, as reported in \Cref{tab:Ex-Alfen}. The results confirm that the proposed scheme attains the expected third-order accuracy.

		\begin{table}[!thb] 
			\centering
			\belowrulesep=0pt
			\aboverulesep=0pt
			\caption{Errors in $(v_1, v_2)$ and $(B_1, B_2)$ at $t = 1$ on an $N \times N$ mesh for the Alfv\'{e}n wave problem.}
			\label{tab:Ex-Alfen}
			\setlength{\tabcolsep}{3mm}{
				\begin{tabular}{c|cccccccc}
					\toprule[1.5pt]
					\multirow{2}{*}{ } &
					\multirow{2}{*}{$N$} &
					\multicolumn{2}{c}{$l^1$ norm} &
					\multicolumn{2}{c}{$l^2$ norm} &
					\multicolumn{2}{c}{$l^\infty$ norm} \\
					\cmidrule(r){3-4} \cmidrule(r){5-6} \cmidrule(l){7-8}
					& & error & order &  error & order &  error & order \\    
					\midrule[1.5pt]
					\multirow{5}{*}{$(v_1, v_2)$} & $20$ & 2.25e-04 & -- &  1.77e-04 & -- &  1.83e-04 & --   \\
					& $40$ & 2.90e-05 & 2.95 & 2.28e-05 & 2.95 & 2.38e-05 & 2.95 \\
					& $80$ & 3.67e-06 & 2.98 & 2.88e-06 & 2.98 & 3.01e-06 & 2.98 \\
					& $160$ & 4.61e-07 & 2.99 & 3.62e-07 & 2.99 & 3.78e-07 & 2.99 \\
					& $320$ & 5.77e-08 & 3.00 & 4.54e-08 & 3.00 & 4.74e-08 & 3.00 \\
					\midrule
					\multirow{5}{*}{$(B_1, B_2)$} & $20$ & 2.28e-04 & -- & 1.79e-04 & -- & 1.80e-04 & --   \\
					& $40$ & 2.92e-05 & 2.97 & 2.30e-05 & 2.97 & 2.31e-05 & 2.96 \\
					& $80$ & 3.68e-06 & 2.99 & 2.89e-06 & 2.99 & 2.91e-06 & 2.99 \\
					& $160$ & 4.61e-07 & 3.00 & 3.63e-07 & 3.00 & 3.64e-07 & 3.00 \\
					& $320$ & 5.77e-08 & 3.00 & 4.54e-08 & 3.00 & 4.56e-08 & 3.00 \\
					\bottomrule[1.5pt]
				\end{tabular}
			}
		\end{table}
	\end{expl}

	\begin{expl}[Smooth Vortex]\label{Ex:vortex} \rm 
		This widely used benchmark assesses both the accuracy and the PP capability of numerical schemes for MHD. The initial conditions are specified as
		\begin{equation*}\label{initial:vortex}
			\left( \rho, {\bm v}, {\bm B}, p \right) = \left( 1, 1 + \delta v_1, 1 + \delta v_2, 0, \delta B_1, \delta B_2, 0, 1 + \delta p \right),
		\end{equation*}
		where the perturbations are defined by
		\[
		\left( \delta v_1, \delta v_2 \right) = \frac{\mu}{\sqrt{2}\pi} e^{0.5(1 - r^2)}(-y, x), \quad
		\left( \delta B_1, \delta B_2 \right) = \frac{\mu}{2\pi} e^{0.5(1 - r^2)}(-y, x),
		\]
		\[
		\delta p = -\frac{\mu^2(1 + r^2)}{8\pi^2} e^{1 - r^2}, \quad \text{with} \quad r^2 = x^2 + y^2.
		\]
		Here, $\mu$ denotes the vortex strength. As $\mu$ increases, the central pressure decreases rapidly; for example, when $\mu = 5.389489439$, the minimum pressure drops to approximately $5.3 \times 10^{-12}$. At such extreme values, numerical schemes lacking a PP mechanism typically fail due to the generation of nonphysical negative pressures. 
		The problem is solved on an $N \times N$ uniform Cartesian mesh with periodic boundary conditions, and the solution is evolved to time $t = 0.05$. 
		
		\Cref{tab:Ex-Vortex} presents the $l^1$ errors and convergence rates of the velocity and magnetic fields for three vortex strengths: $\mu = 1$, $5.38948943$, and $5.389489439$. Across all cases, the observed convergence rates range from approximately $3.81$ to $4.13$, surpassing the scheme’s formal third-order accuracy. This superconvergence is likely due to the high-order damping effect of the COE procedure dominating the error, as also noted in \cite{peng2023oedg}. 
		It is noted that, for the mild case of $\mu = 1$, the PP limiter remains inactive throughout the simulation. In contrast, for the demanding cases $\mu = 5.38948943$ and $5.389489439$, the simulation fails almost immediately without the PP limiter due to the onset of negative pressures. These results clearly demonstrate that the PP limiter and the COE procedure enhance robustness without compromising the scheme’s accuracy.
		\begin{table}[!thb] 
			\centering
			\belowrulesep=0pt
			\aboverulesep=0pt
			\caption{$l^1$ errors in ${\bm v}$ and ${\bm B}$ at $t = 0.05$ on an $N \times N$ mesh for the vortex problem with vortex strengths $\mu_1 = 1.0$, $\mu_2 = 5.38948943$, and $\mu_3 = 5.389489439$.}
			\label{tab:Ex-Vortex}
			\setlength{\tabcolsep}{3mm}{
				\begin{tabular}{c|cccccccc}
					\toprule[1.5pt]
					\multirow{2}{*}{ } &
					\multirow{2}{*}{$N$} &
					\multicolumn{2}{c}{$ \mu_{1} $} &
					\multicolumn{2}{c}{$ \mu_{2} $} &
					\multicolumn{2}{c}{$ \mu_{3} $} \\
					\cmidrule(r){3-4} \cmidrule(r){5-6} \cmidrule(l){7-8}
					& & error & order &  error & order &  error & order \\    
					\midrule[1.5pt]
					\multirow{5}{*}{${\bm v}_h$} 
					& $20$ & 6.23e-03 & --    & 5.19e-03 & --    & 5.27e-03 & -- \\
					& $40$ & 2.21e-03 & 1.49  & 1.28e-03 & 2.02  & 1.28e-03 & 2.05 \\
					& $80$ & 2.66e-04 & 3.05  & 9.30e-05 & 3.79  & 9.28e-05 & 3.78 \\
					& $160$& 1.52e-05 & 4.13  & 6.58e-06 & 3.82  & 6.58e-06 & 3.82 \\
					& $320$& 9.61e-07 & 3.98  & 4.33e-07 & 3.93  & 4.32e-07 & 3.93 \\
					\midrule
					\multirow{5}{*}{${\bm B}_h$} 
					& $20$ & 4.40e-03 & --    & 3.67e-03 & --    & 3.72e-03 & -- \\
					& $40$ & 1.56e-03 & 1.50  & 9.02e-04 & 2.02  & 8.97e-04 & 2.05 \\
					& $80$ & 1.88e-04 & 3.05  & 6.54e-05 & 3.79  & 6.52e-05 & 3.78 \\
					& $160$& 1.08e-05 & 4.13  & 4.65e-06 & 3.81  & 4.65e-06 & 3.81 \\
					& $320$& 6.93e-07 & 3.96  & 3.08e-07 & 3.92  & 3.08e-07 & 3.92 \\
					\bottomrule[1.5pt]
				\end{tabular}
			}
		\end{table}
	\end{expl}

	\begin{expl}[Orszag--Tang Problem]\label{Ex:OT} \rm  
		The Orszag--Tang vortex is a classical benchmark problem used to assess the shock-capturing capability of numerical methods for MHD. Although the initial conditions are smooth, the solution rapidly evolves into a complex configuration involving multiple interacting shocks and discontinuities. 
		The initial conditions are given by
		\begin{equation*}
			(\rho, {\bm v}, {\bm B}, p) = \left(\gamma^2, -\sin y, \sin x, 0, -\sin y, \sin(2x), 0, \gamma \right).
		\end{equation*}
		The computational domain $\left[0, 2\pi\right]^2$ is discretized using a $400 \times 400$ uniform Cartesian mesh with periodic boundary conditions.
		
		\Cref{fig:Ex-OT-density} shows the density contours at times $t = 3$ and $t = 4$. The top row displays results obtained without the COE procedure, while the bottom row includes COE. As observed, the simulation without COE exhibits noticeable spurious oscillations, whereas the inclusion of COE significantly suppresses such artifacts, leading to cleaner and more physically consistent results. 
		The proposed structure-preserving PAMPA method demonstrates strong performance in capturing both discontinuities and smooth features of the flow. The results are consistent with those reported in~\cite{Li2011,Guillet2019,DingWu2024SISCMHD}. As the solution evolves from $t = 3$ to $t = 4$, the flow becomes increasingly intricate.  Prior studies (e.g.,~\cite{jiang1999high}) have reported the onset of nonphysical negative pressures near $t \approx 3.9$ due to intense shock interactions and steep gradients. In contrast, our structure-preserving PAMPA scheme, equipped with a PP mechanism, successfully avoids such issues and maintains physical admissibility throughout the simulation.

		\begin{figure}[!thb]
			\centering
			\begin{subfigure}{0.48\textwidth}
				\includegraphics[width=\textwidth]{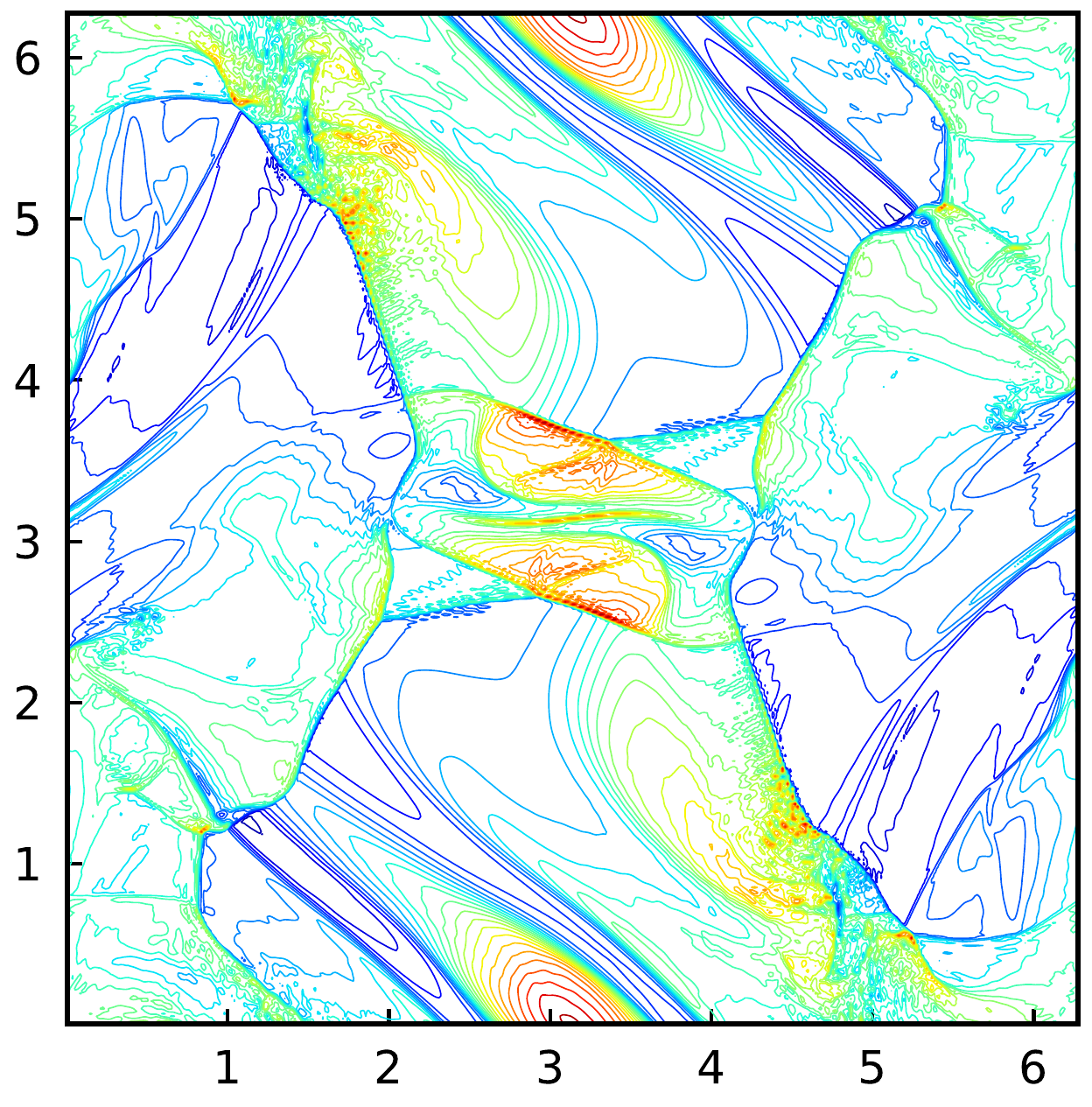}
			\end{subfigure}
			\hfill
			\begin{subfigure}{0.48\textwidth}
				\includegraphics[width=\textwidth]{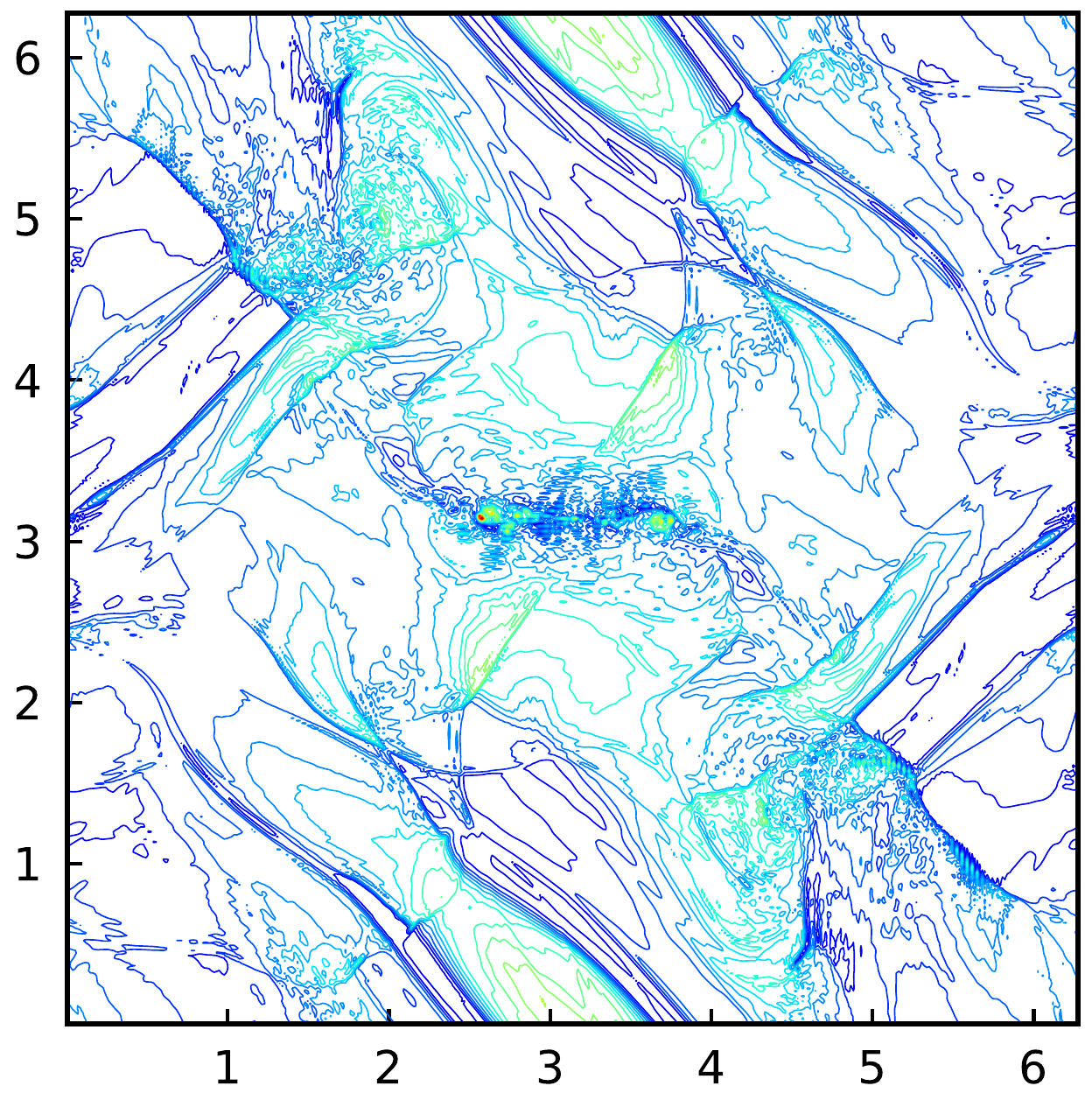}
			\end{subfigure}
			
			\begin{subfigure}{0.48\textwidth}
				\includegraphics[width=\textwidth]{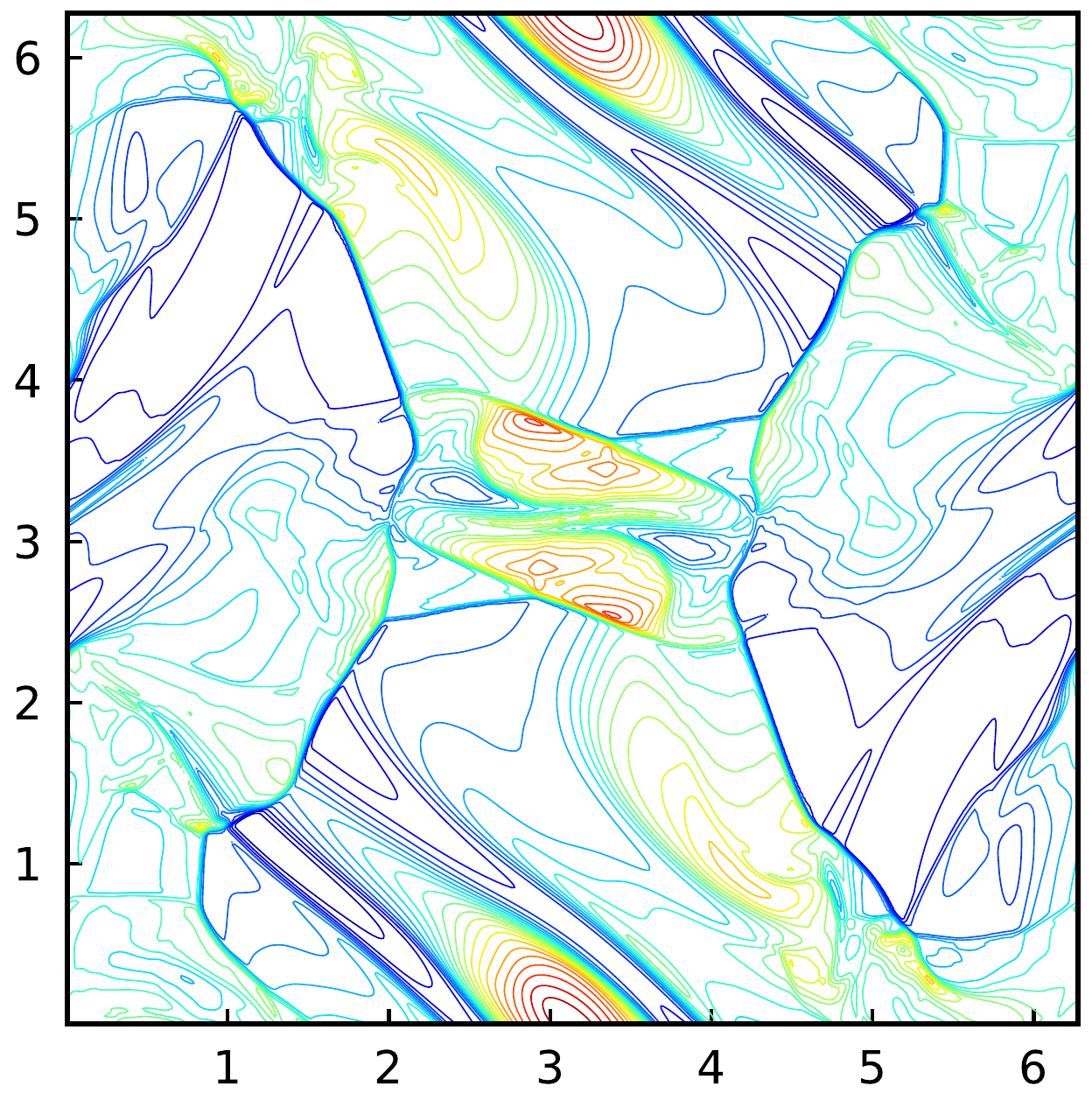}
			\end{subfigure}
			\hfill
			\begin{subfigure}{0.48\textwidth}
				\includegraphics[width=\textwidth]{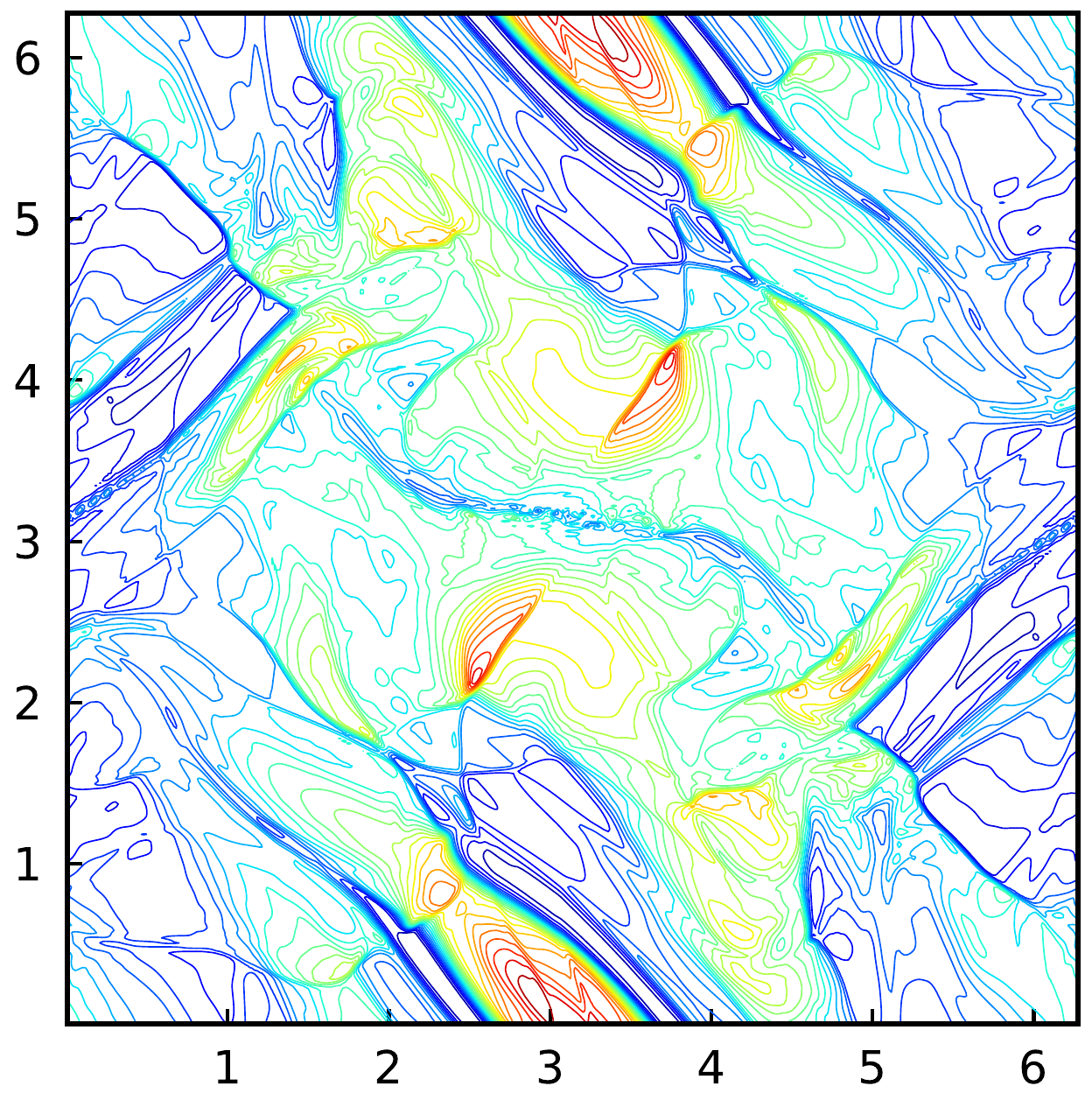}
			\end{subfigure}
			\caption{\Cref{Ex:OT}: Density contours for the Orszag--Tang problem at $t = 3$ (left) and $t = 4$ (right). Twenty-four contour levels are shown. Top: without COE; bottom: with COE.}

			\label{fig:Ex-OT-density}
		\end{figure}
	\end{expl}

	\begin{expl}[Rotor Problem]\label{Ex:Rotor} \rm
		We next consider the rotor problem, a widely used benchmark for evaluating the shock-capturing capability and robustness of MHD schemes. It models a rapidly spinning, high-density core embedded in a uniform, magnetized ambient plasma. The initial conditions are defined as
		\begin{equation*}
			(\rho, {\bm v}, {\bm B}, p)= \begin{cases}
				\left(10,-(y-0.5) / r_1,(x-0.5) / r_1, 0,2.5 / \sqrt{4 \pi}, 0,0,0.5\right), & r \leq r_1, \\
				(1+9 \phi,-\phi(y-0.5) / r, \phi(x-0.5) / r, 0,2.5 / \sqrt{4 \pi}, 0,0,0.5), & r_1<r \leq r_2, \\
				(1,0,0,0,2.5 / \sqrt{4 \pi}, 0,0,0.5), & r_2<r,
			\end{cases}
		\end{equation*}
		where $r = \sqrt{(x - 0.5)^2 + (y - 0.5)^2}$, $\phi = (r_2 - r)/(r_2 - r_1)$, $r_1 = 0.1$, and $r_2 = 0.115$. 
		The computational domain is $[0, 1]^2$, discretized using a $200 \times 200$ uniform Cartesian mesh with outflow boundary conditions. The simulation is advanced in time up to $t = 0.295$.
		
		\begin{figure}[!thb]
			\centering		
			\begin{subfigure}{0.48\textwidth}
				\includegraphics[width=\textwidth]{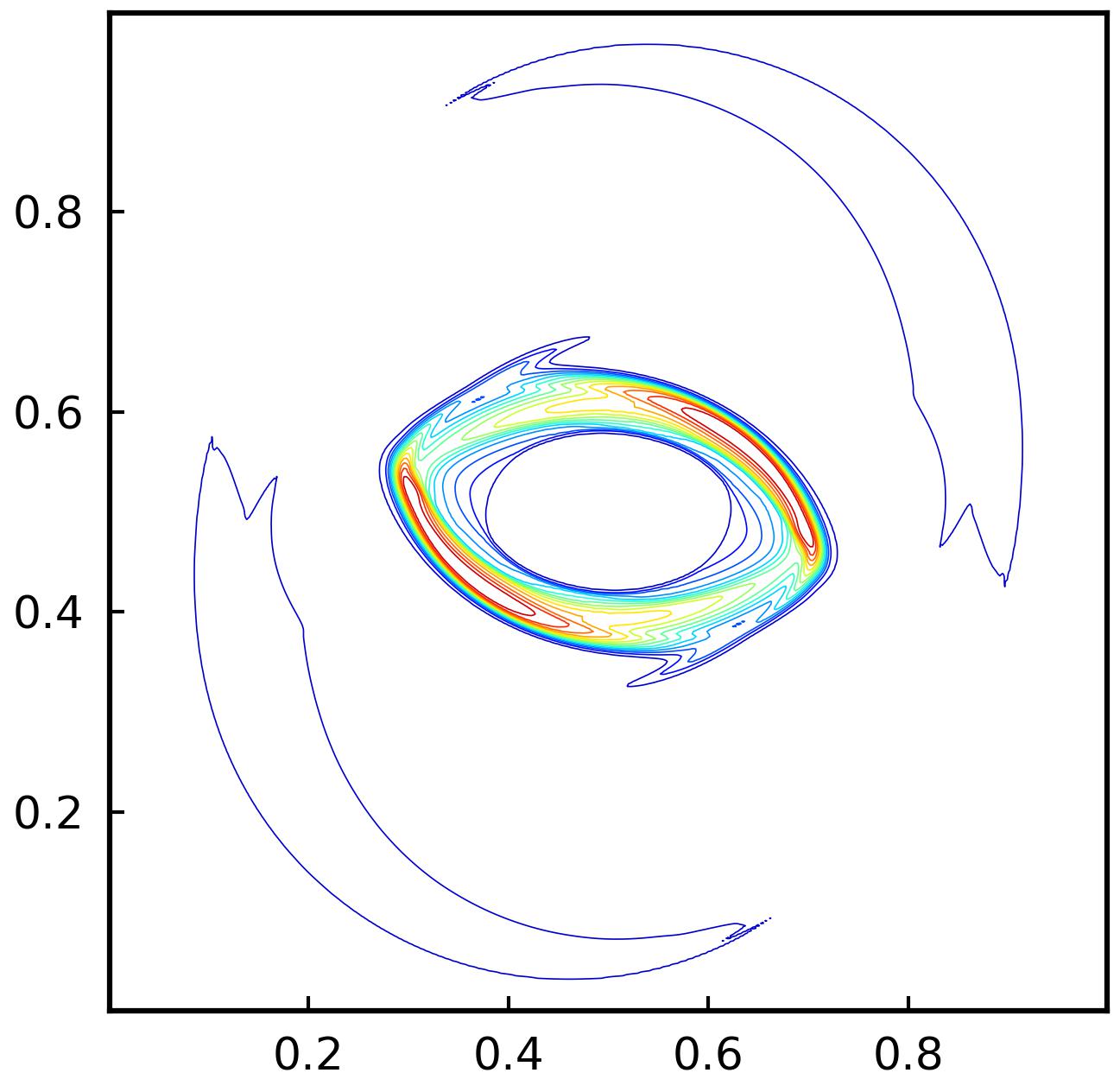}
			\end{subfigure}
			\hfill
			\begin{subfigure}{0.48\textwidth}
				\includegraphics[width=\textwidth]{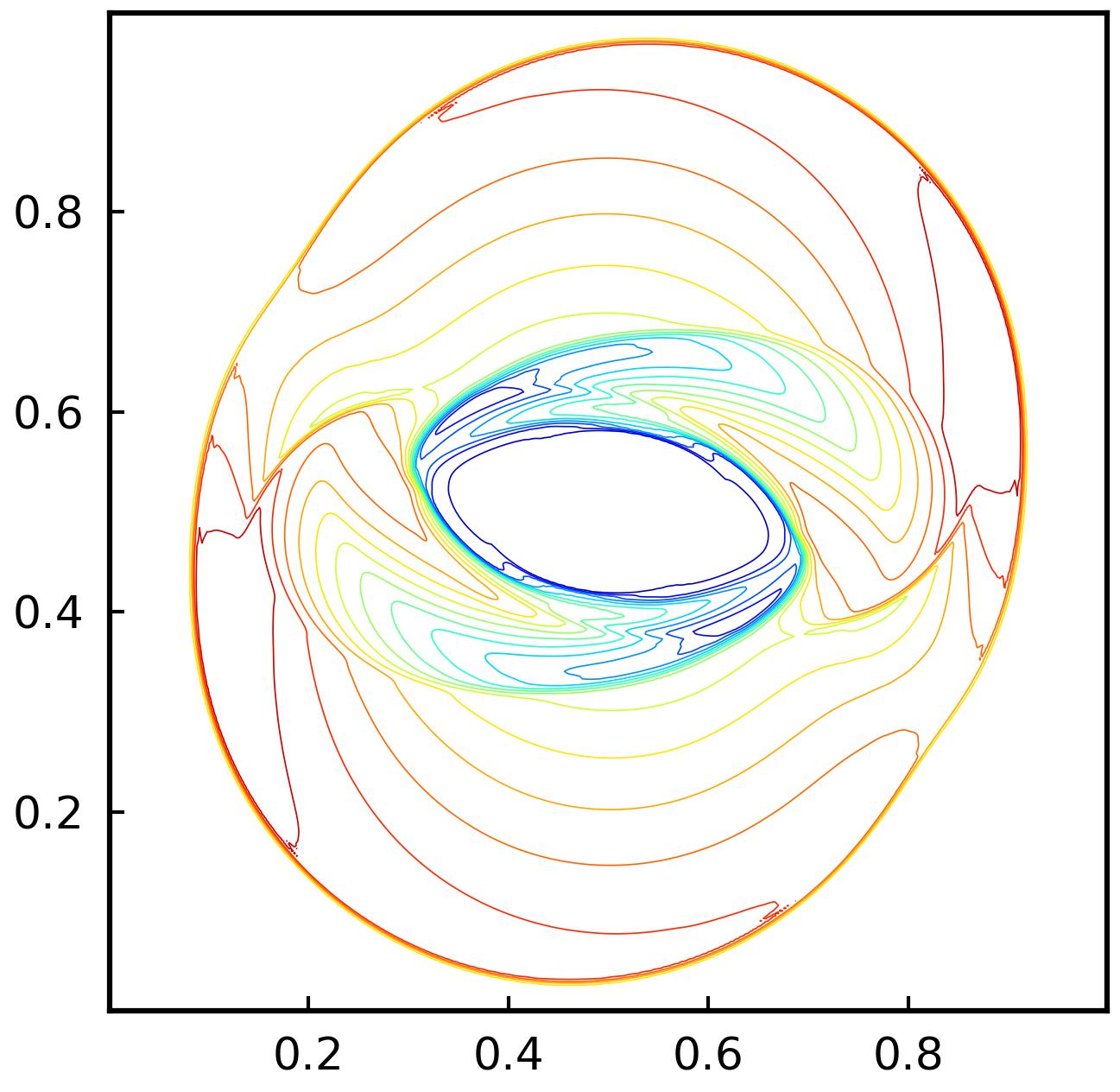}
			\end{subfigure}
			\begin{subfigure}{0.48\textwidth}
				\includegraphics[width=\textwidth]{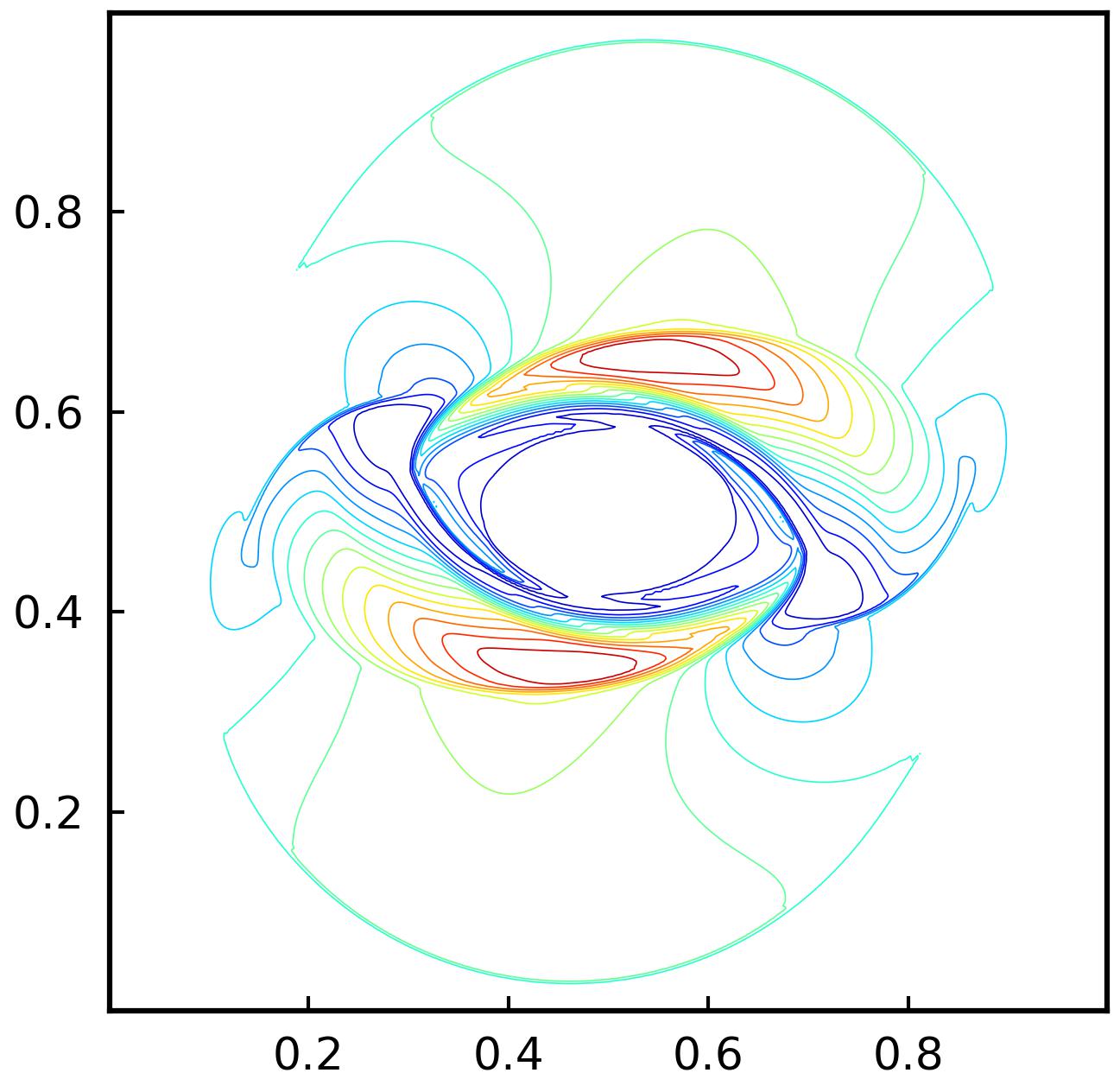}
			\end{subfigure}
			\hfill
			\begin{subfigure}{0.48\textwidth}
				\includegraphics[width=\textwidth]{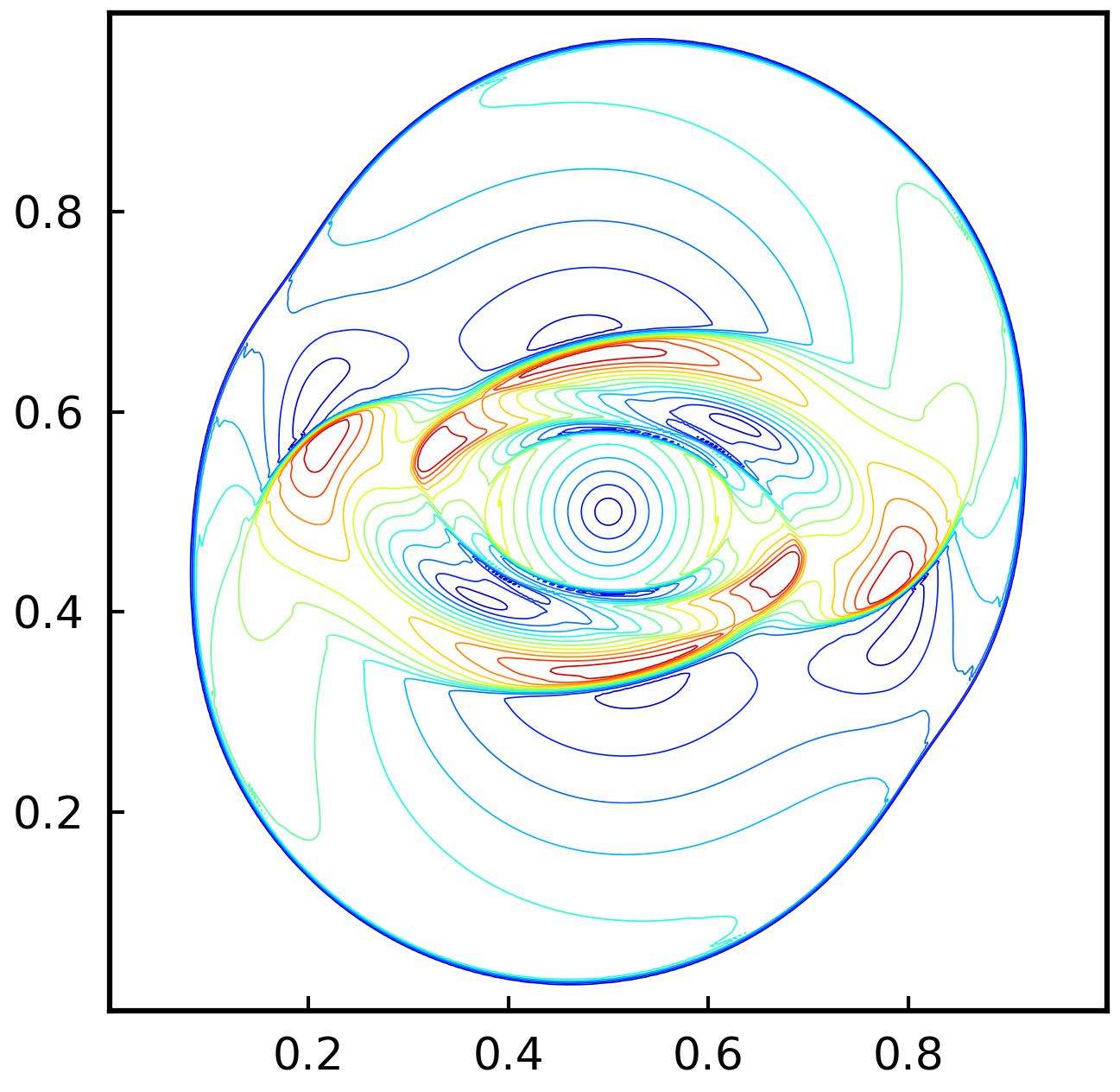}
			\end{subfigure}
			\caption{\Cref{Ex:Rotor}: The density (top-left), thermal pressure (top-right), magnetic pressure (bottom-left), and velocity magnitude (bottom-right) for the rotor problem at $t = 0.295$. 
			}
			\label{fig:Ex-Rotor1}
		\end{figure}

		\Cref{fig:Ex-Rotor1} displays contour plots of the density $\rho$, thermal pressure $p$, magnetic pressure $\frac{\left | {\bm B} \right |^2} 2$, and the velocity magnitude $\left | {\bm v} \right |$. The simulation successfully preserves the expected circular rotational structure, which has been identified by T\'oth~\cite{Toth2000} as a challenging feature for many MHD solvers. 
		\Cref{fig:Ex-RotorDDF} presents Mach number contours at $t = 0.295$. The top row corresponds to results obtained without the DDF projection, while the bottom row shows results with DDF applied. It is well known that significant divergence errors in the magnetic field can introduce visible distortions, particularly near the center of the Mach number profile~\cite{Li2005}. 
		In our results, the solution without DDF exhibits noticeable nonphysical oscillations, especially in the vicinity of the rotating core. In contrast, the application of DDF effectively suppresses these artifacts, demonstrating its critical role in maintaining the stability and accuracy of the PAMPA scheme. 
		Moreover, due to the strong gradients in the solution, negative pressure values may arise near the center. However, no such nonphysical behavior is observed in our simulation, owing to the provably PP property of the proposed PAMPA scheme.

		\begin{figure}[!thb]
			\centering		
			\begin{subfigure}{0.48\textwidth}
				\includegraphics[width=\textwidth]{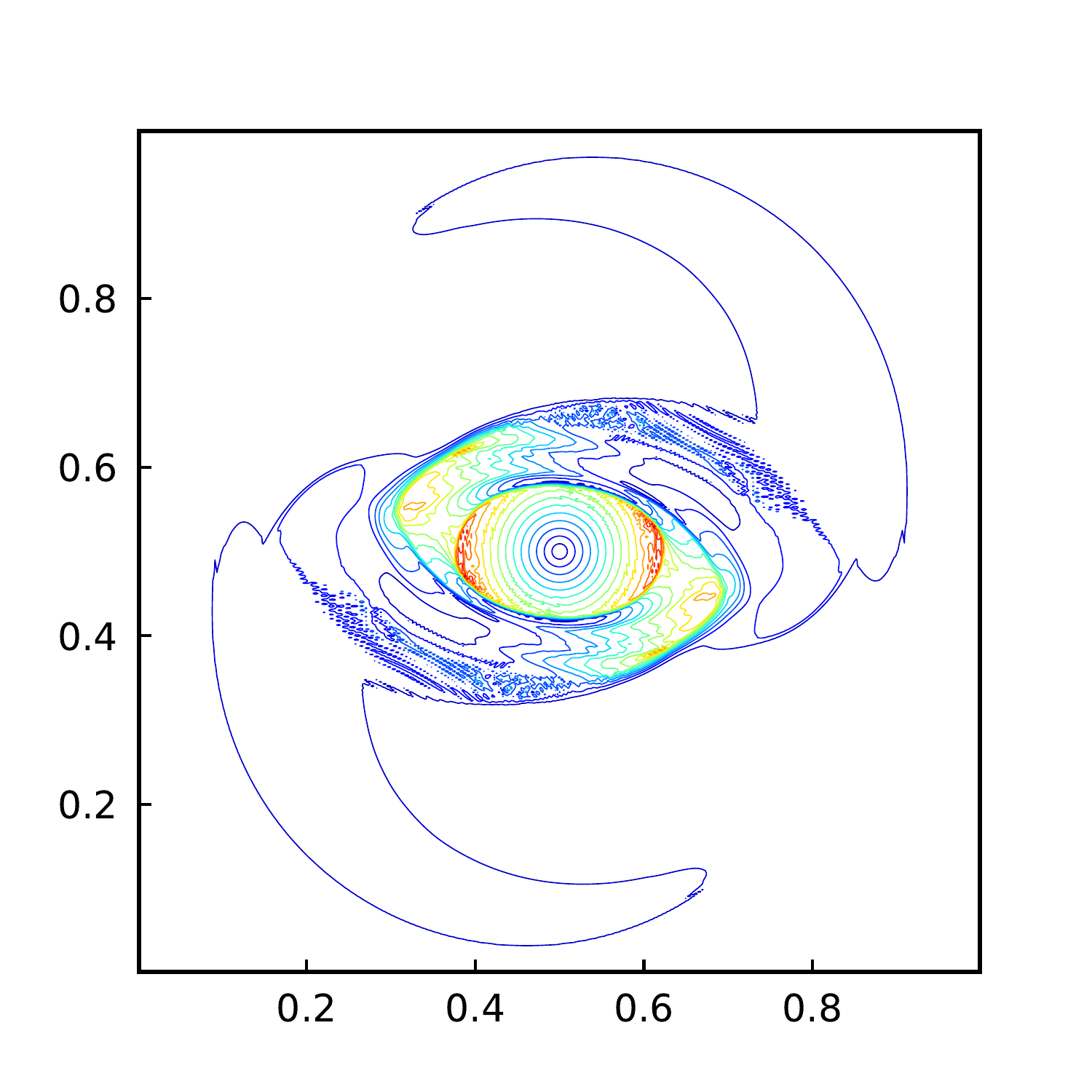}
			\end{subfigure}
			\hfill
			\begin{subfigure}{0.48\textwidth}
				\includegraphics[width=\textwidth]{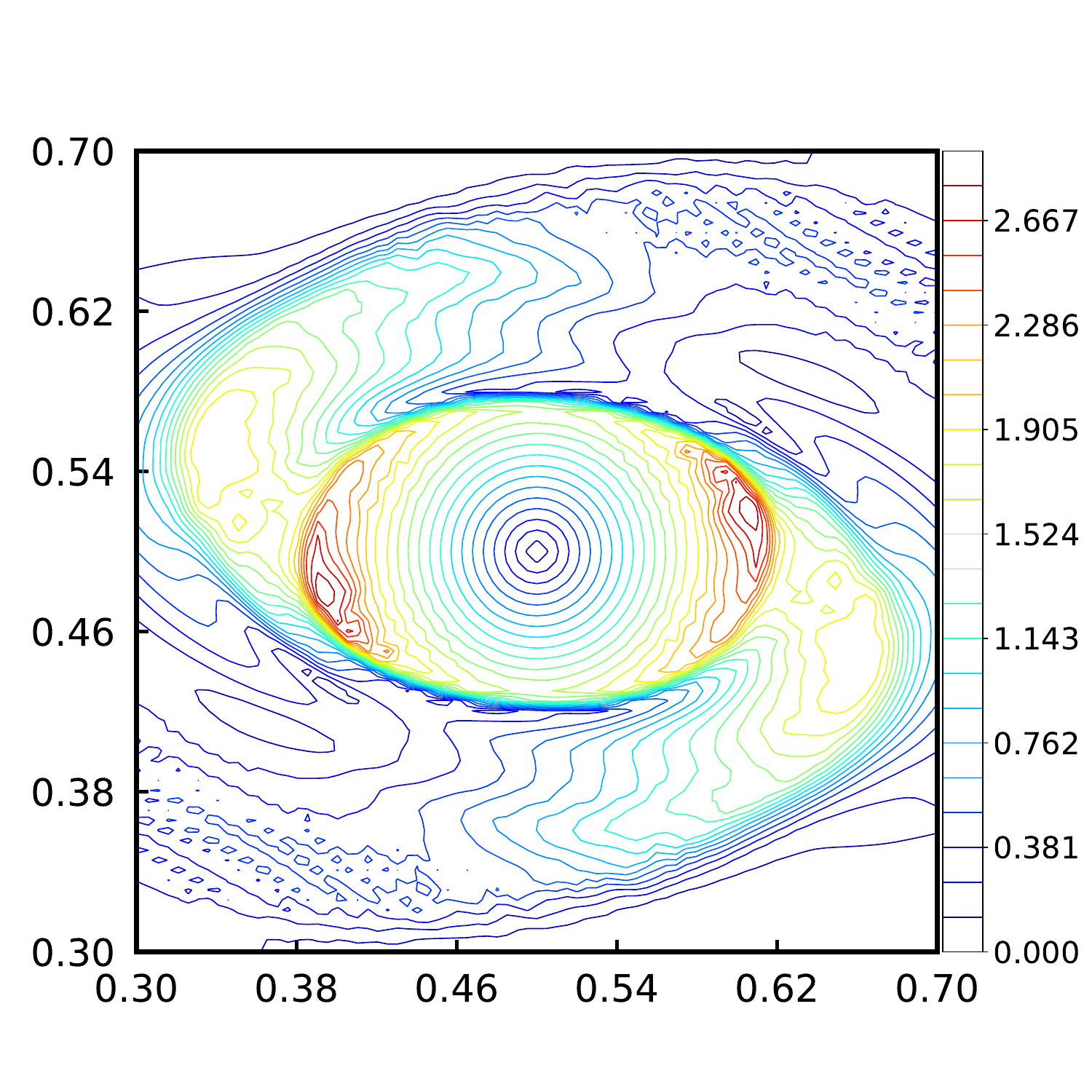}
			\end{subfigure}
			\begin{subfigure}{0.48\textwidth}
				\includegraphics[width=\textwidth]{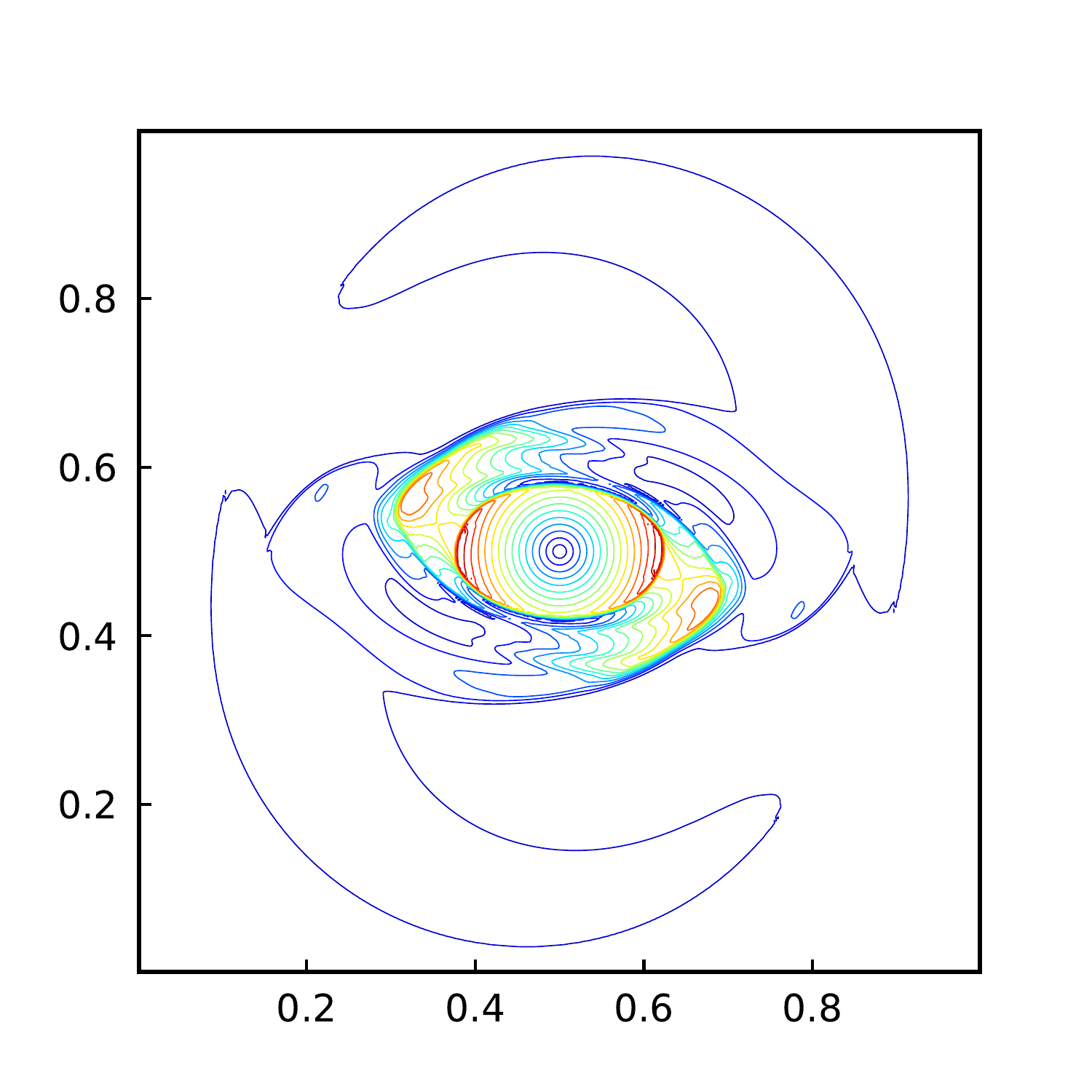}
			\end{subfigure}
			\hfill
			\begin{subfigure}{0.48\textwidth}
				\includegraphics[width=\textwidth]{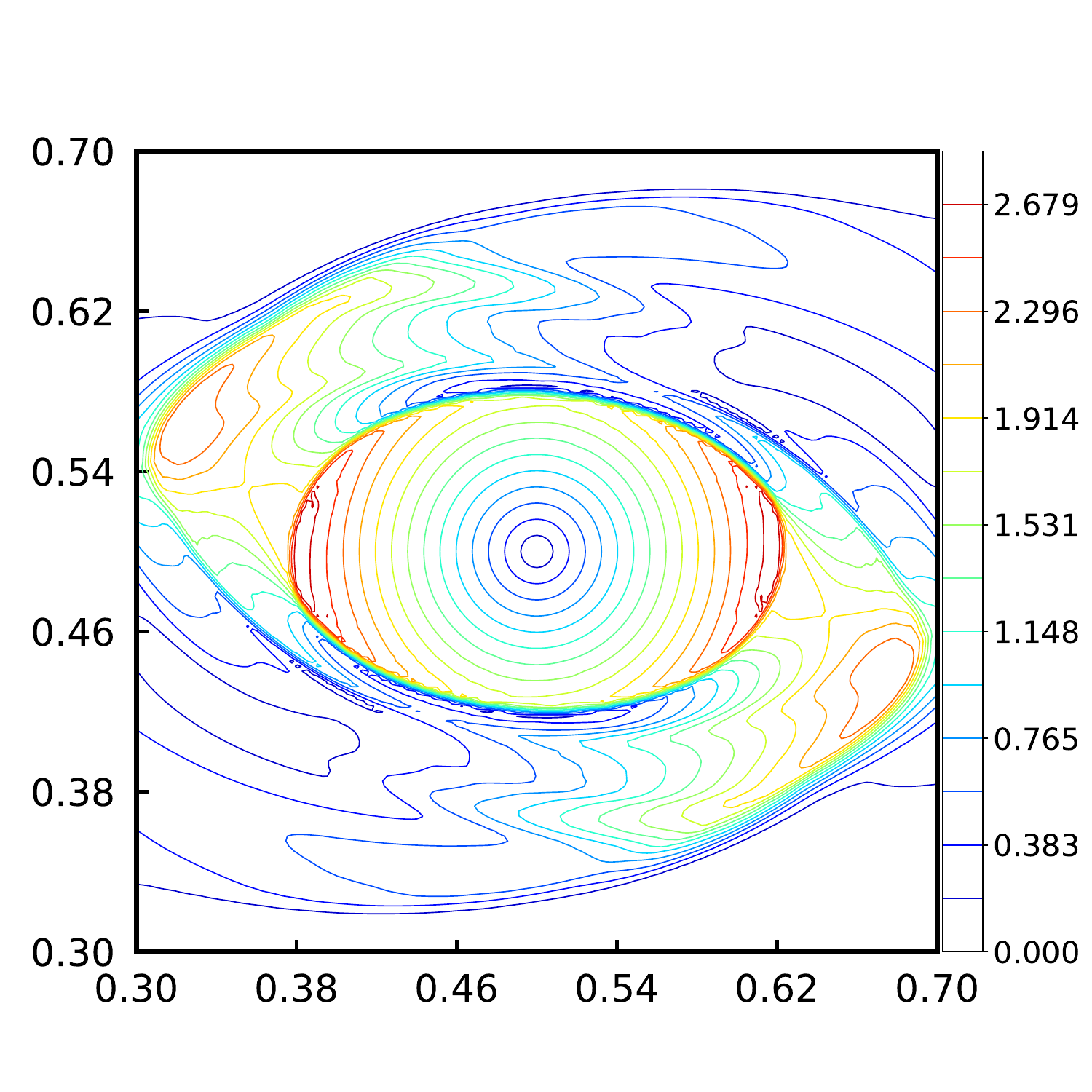}
			\end{subfigure}
			\caption{\Cref{Ex:Rotor}:Mach number contours for the rotor problem at $t = 0.295$. Top: without DDF projection; bottom: with DDF projection. Left: global view; right: zoomed-in view.}
			\label{fig:Ex-RotorDDF}
		\end{figure}
		
		To evaluate the effectiveness of our shock indicator, \Cref{fig:Rotor-BadCells} shows the distribution of detected troubled cells at six representative time instances. It is evident that the identified cells are well correlated with regions of strong discontinuities, such as shock fronts. This demonstrates that the indicator accurately distinguishes between smooth and non-smooth regions of the solution, thereby enabling the localized application of the COE procedure. As a result, computational cost and numerical dissipation are significantly reduced in smooth regions, while ensuring stability near discontinuities.

		\begin{figure}[!thb]
			\centering
			\begin{subfigure}[b]{0.3\textwidth}
				\includegraphics[width=\linewidth]{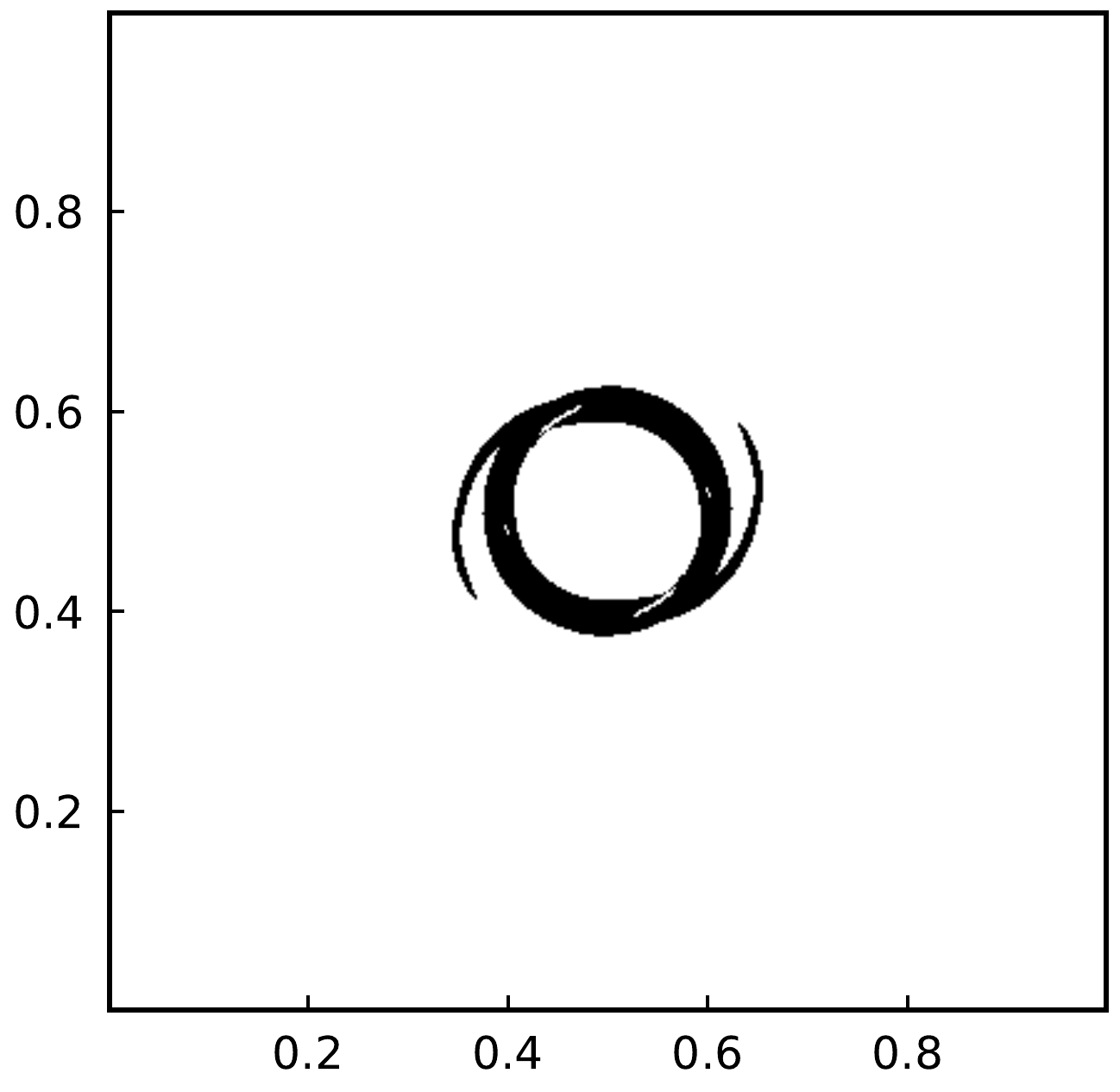}
				\label{fig:p1}
			\end{subfigure}
			\hfill
			\begin{subfigure}[b]{0.3\textwidth}
				\includegraphics[width=\linewidth]{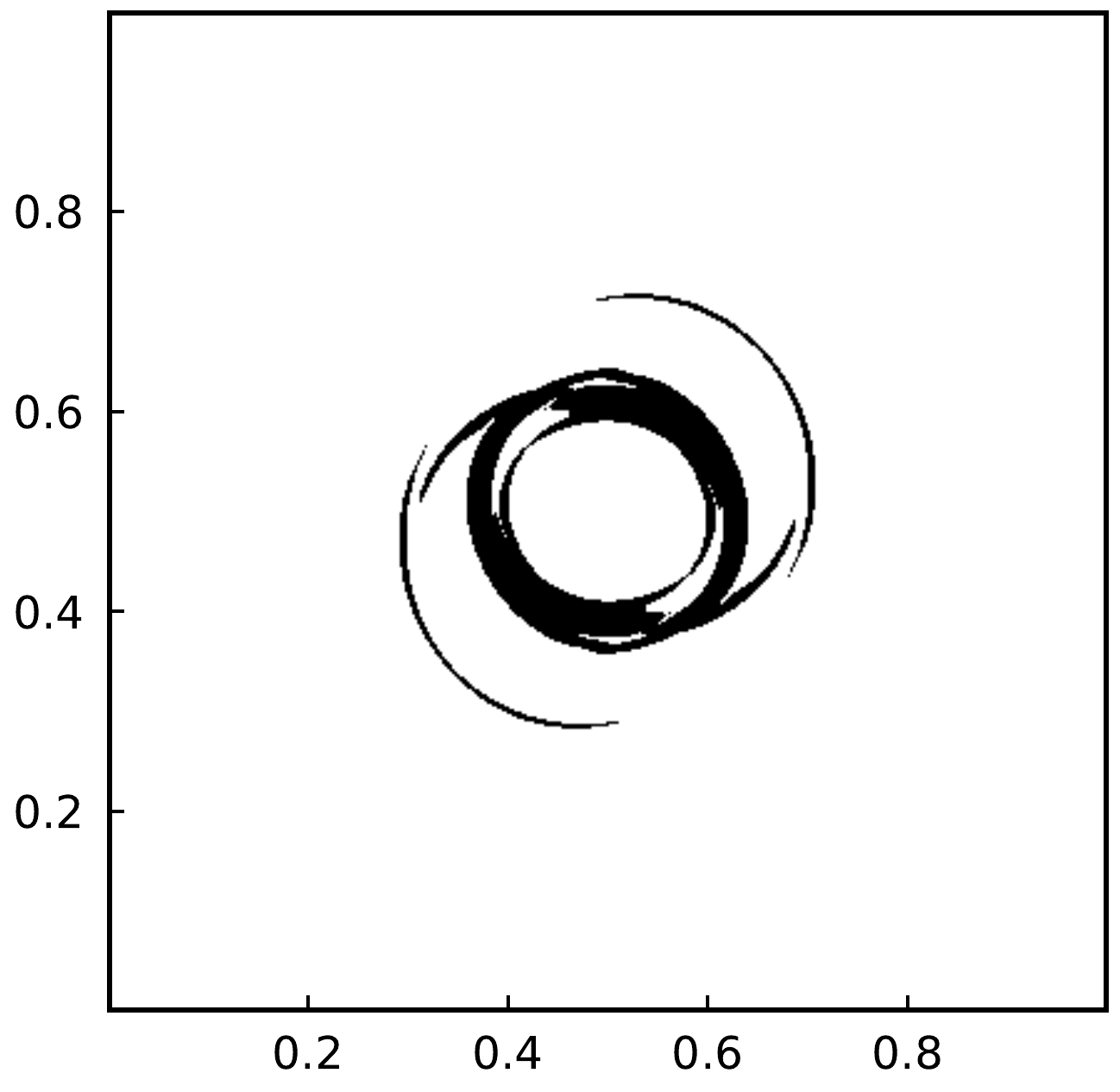}
				\label{fig:p2}
			\end{subfigure}
			\hfill
			\begin{subfigure}[b]{0.3\textwidth}
				\includegraphics[width=\linewidth]{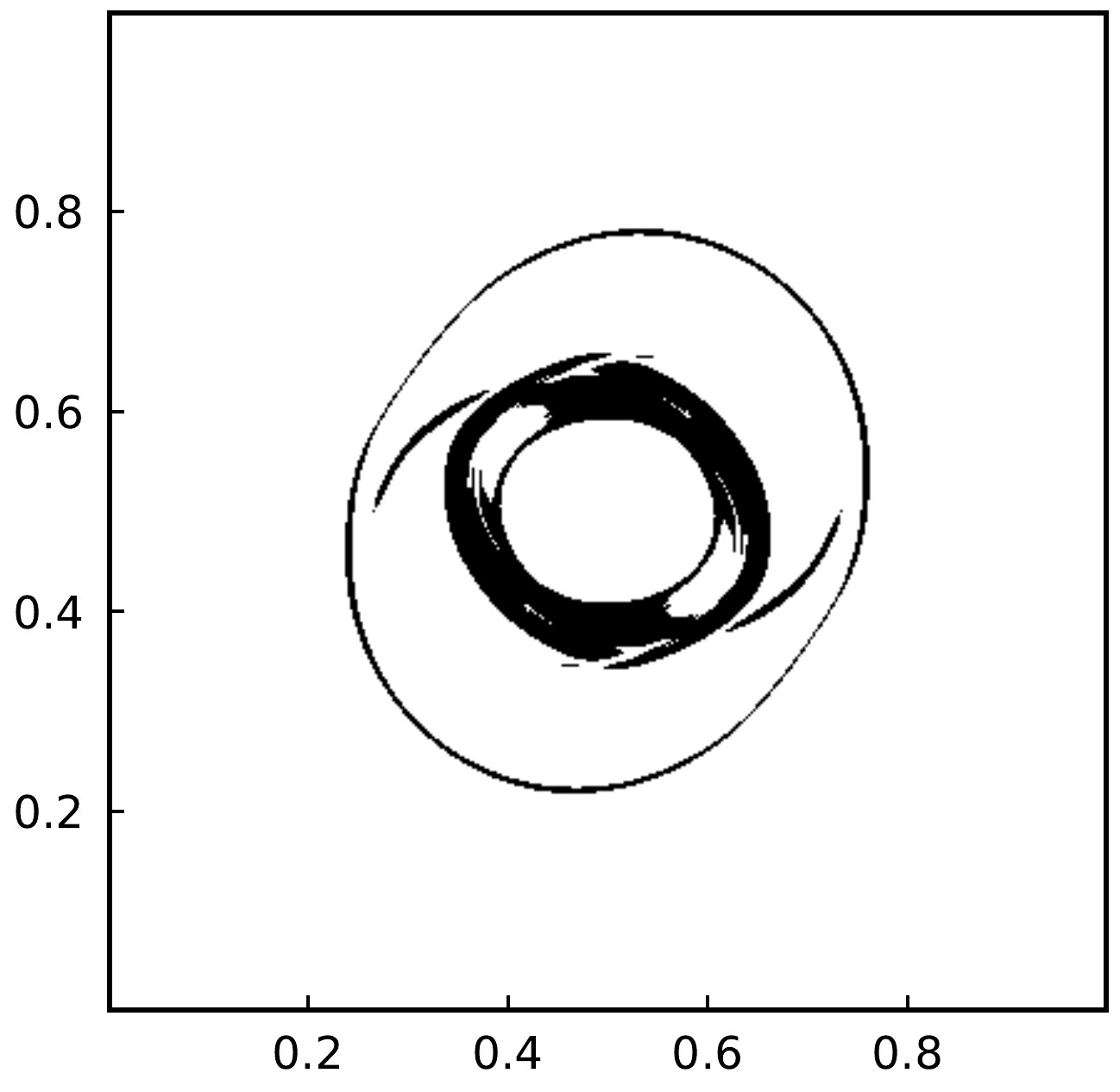}
				\label{fig:p3}
			\end{subfigure}
			
			\vspace{0.01cm}
			
			\begin{subfigure}[b]{0.3\textwidth}
				\includegraphics[width=\linewidth]{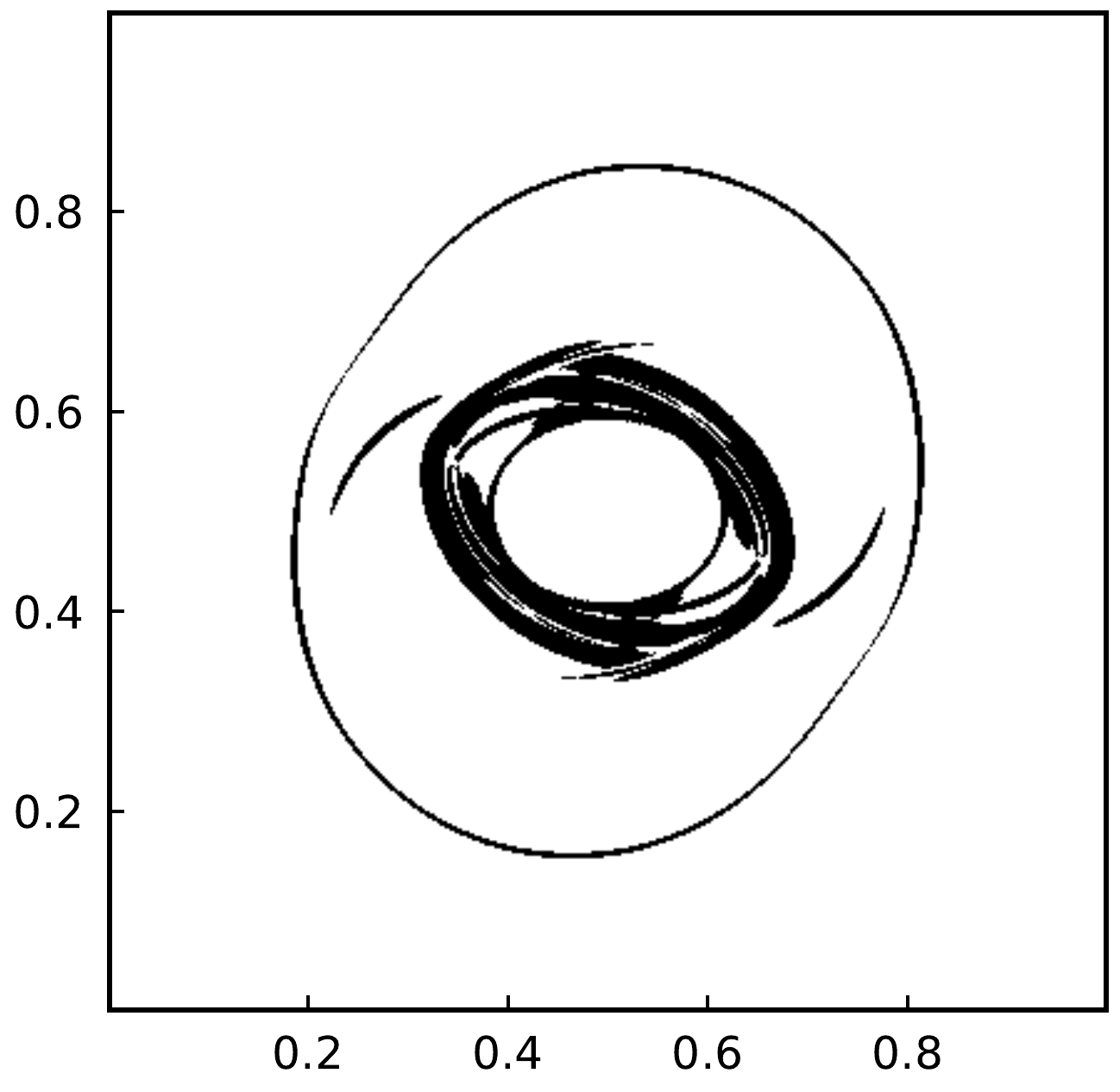}
				\label{fig:p4}
			\end{subfigure}
			\hfill
			\begin{subfigure}[b]{0.3\textwidth}
				\includegraphics[width=\linewidth]{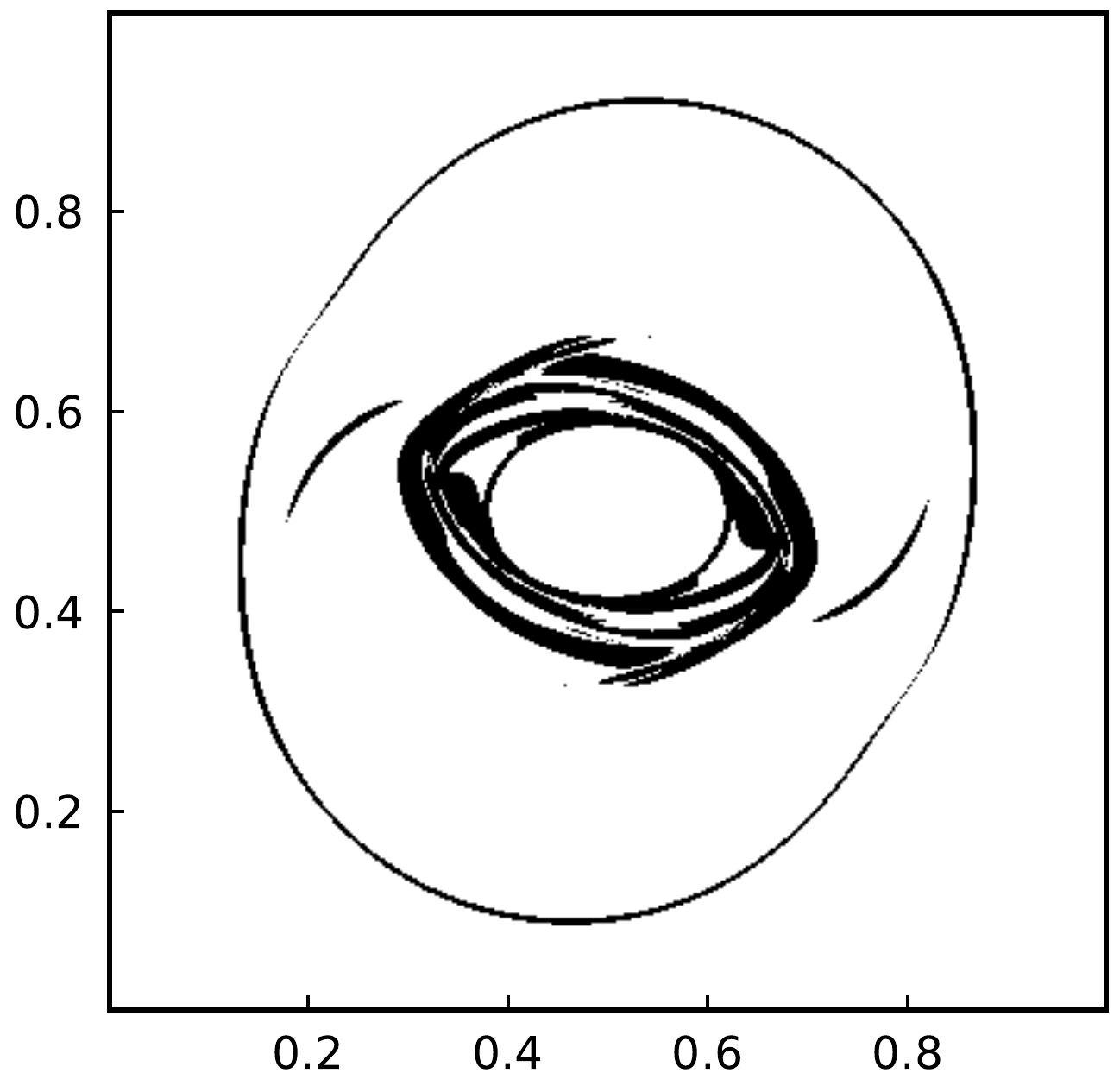}
				\label{fig:p5}
			\end{subfigure}
			\hfill
			\begin{subfigure}[b]{0.3\textwidth}
				\includegraphics[width=\linewidth]{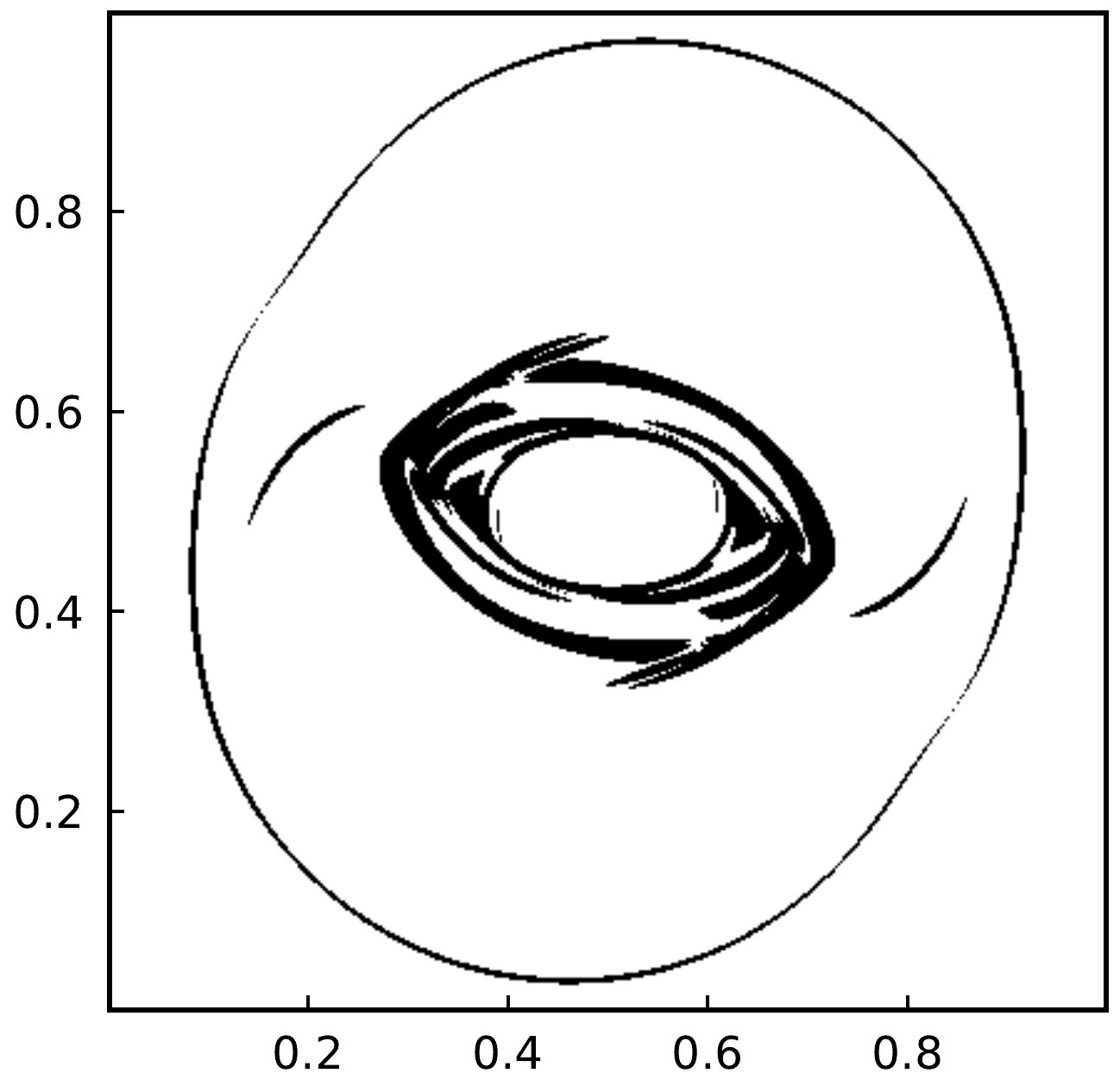}
				\label{fig:p6}
			\end{subfigure}
			\caption{\Cref{Ex:Rotor}: Distribution of detected troubled cells at times $t = 0.05$, $0.1$, $0.15$ (top row, left to right) and $t = 0.2$, $0.25$, $0.295$ (bottom row, left to right).  Black regions indicate the locations of troubled cells identified by our shock indicator.}
			\label{fig:Rotor-BadCells}
		\end{figure}	
		
	\end{expl}

	\begin{expl}[Blast Problem \uppercase\expandafter{\romannumeral1}]\label{Ex:Blast} \rm
		A widely studied and notoriously challenging test case in MHD is the blast wave problem, originally proposed by Balsara and Spicer~\cite{BalsaraSpicer1999}. Due to its extreme difficulty, it has been extensively revisited in the literature~\cite{BalsaraSpicer1999,Li2011,Li2012,WuShu2018,WuShu2019}. One of the primary challenges lies in the presence of steep pressure gradients induced by the strong explosion, which can easily lead to nonphysical negative pressures, particularly near shock fronts. Consequently, this problem serves as a stringent benchmark for assessing the robustness and stability of numerical MHD solvers. 
		The simulation is conducted on a square domain $[-0.5, 0.5]^2$, discretized with a uniform $400 \times 400$ Cartesian mesh. Outflow boundary conditions are applied on all sides. The initial conditions follow the classical configuration in~\cite{BalsaraSpicer1999}:
		\begin{equation*}
			(\rho, {\bm v}, {\bm B}, p) = 
			\begin{cases} 
				\left( 1, 0, 0, 0, B_0, 0, 0, p_0 \right), & \text{if } \sqrt{x^2 + y^2} \leq 0.1, \\ 
				\left( 1, 0, 0, 0, B_0, 0, 0, 0.1 \right), & \text{otherwise,} 
			\end{cases}
		\end{equation*}
		where the adiabatic index is set to $\gamma = 1.4$, the central high pressure is $p_0 = 10^3$, and the uniform magnetic field is initialized as $B_0 = \frac{100}{\sqrt{4\pi}}$. This setup results in a very low plasma beta, $\beta \approx 2.51 \times 10^{-4}$, creating a magnetically dominated environment that poses significant challenges for numerical stability.

		\begin{figure}[!thb]
			\centering		
			\begin{subfigure}{0.48\textwidth}
				\includegraphics[width=\textwidth]{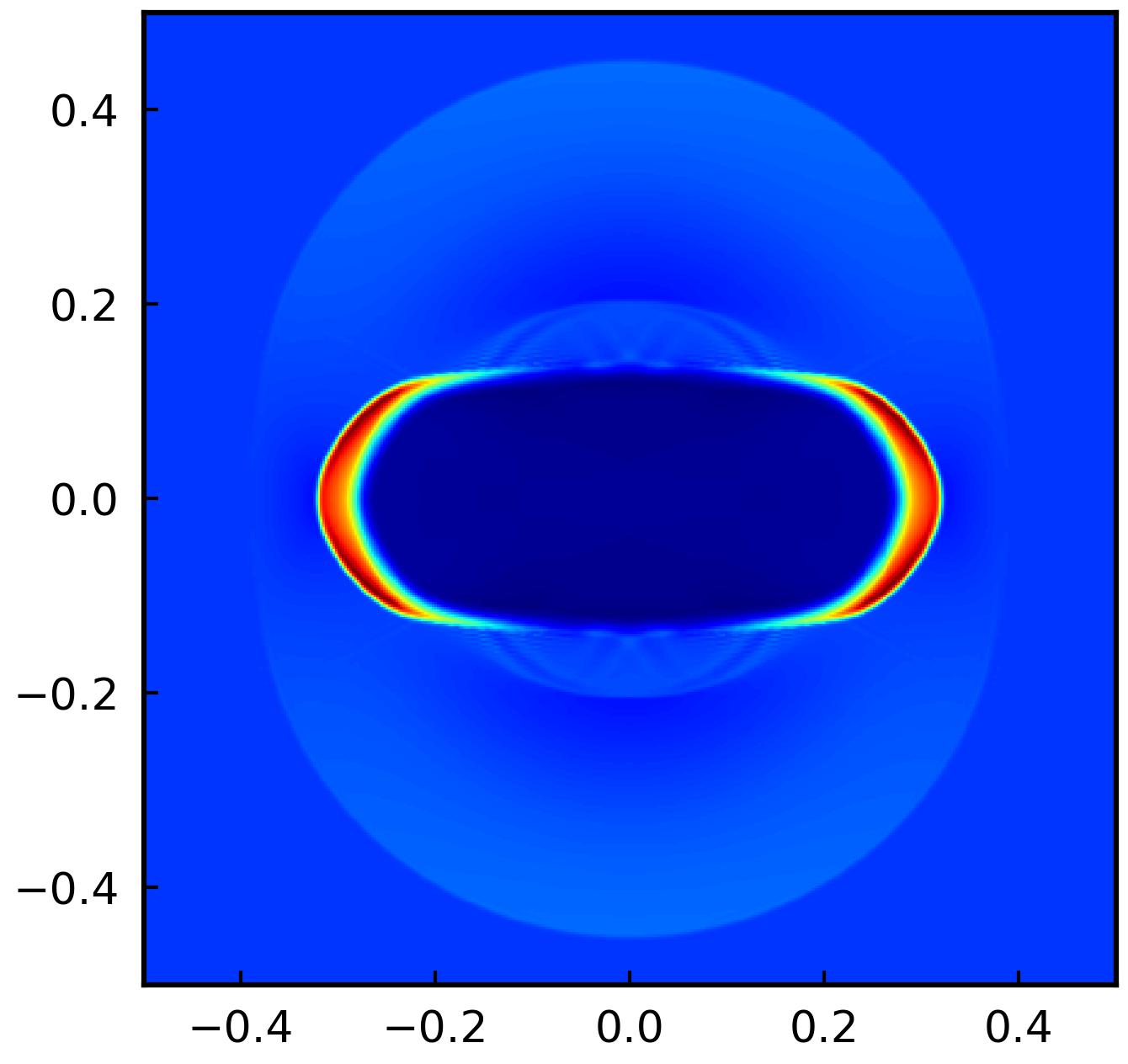}
			\end{subfigure}
			\hfill
			\begin{subfigure}{0.48\textwidth}
				\includegraphics[width=\textwidth]{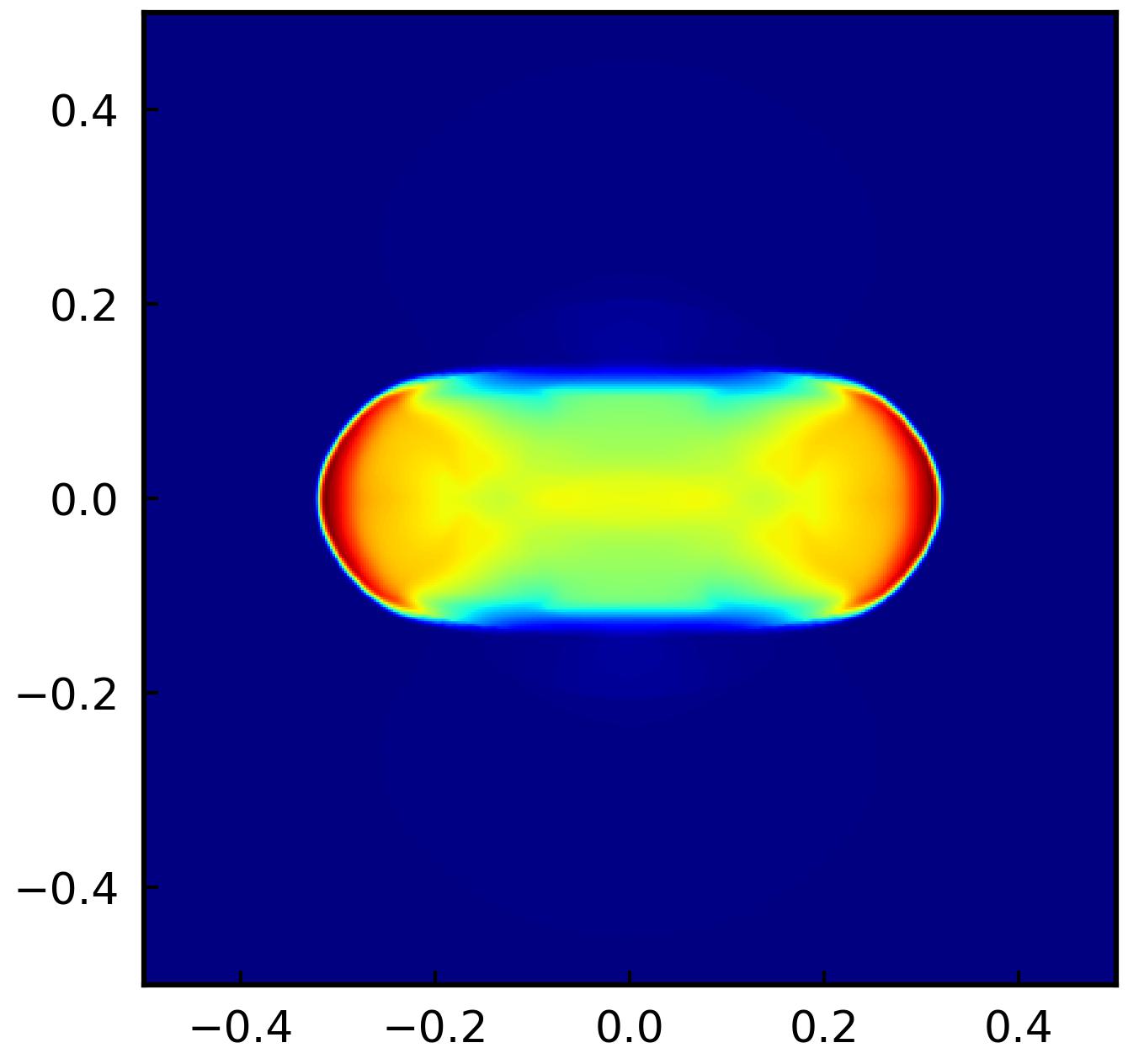}
			\end{subfigure}
			
			\begin{subfigure}{0.48\textwidth}
				\includegraphics[width=\textwidth]{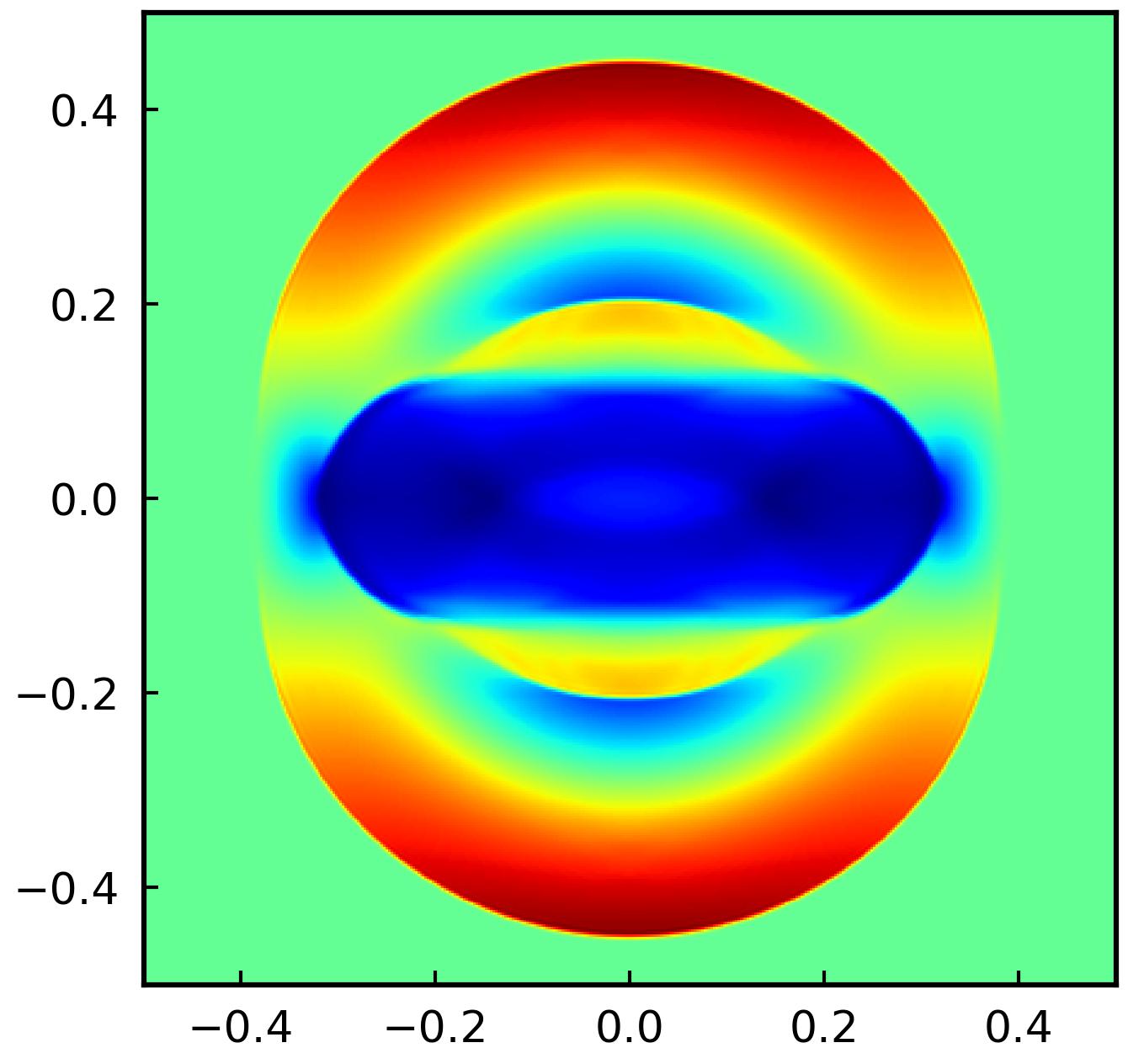}
			\end{subfigure}
			\hfill
			\begin{subfigure}{0.48\textwidth}
				\includegraphics[width=\textwidth]{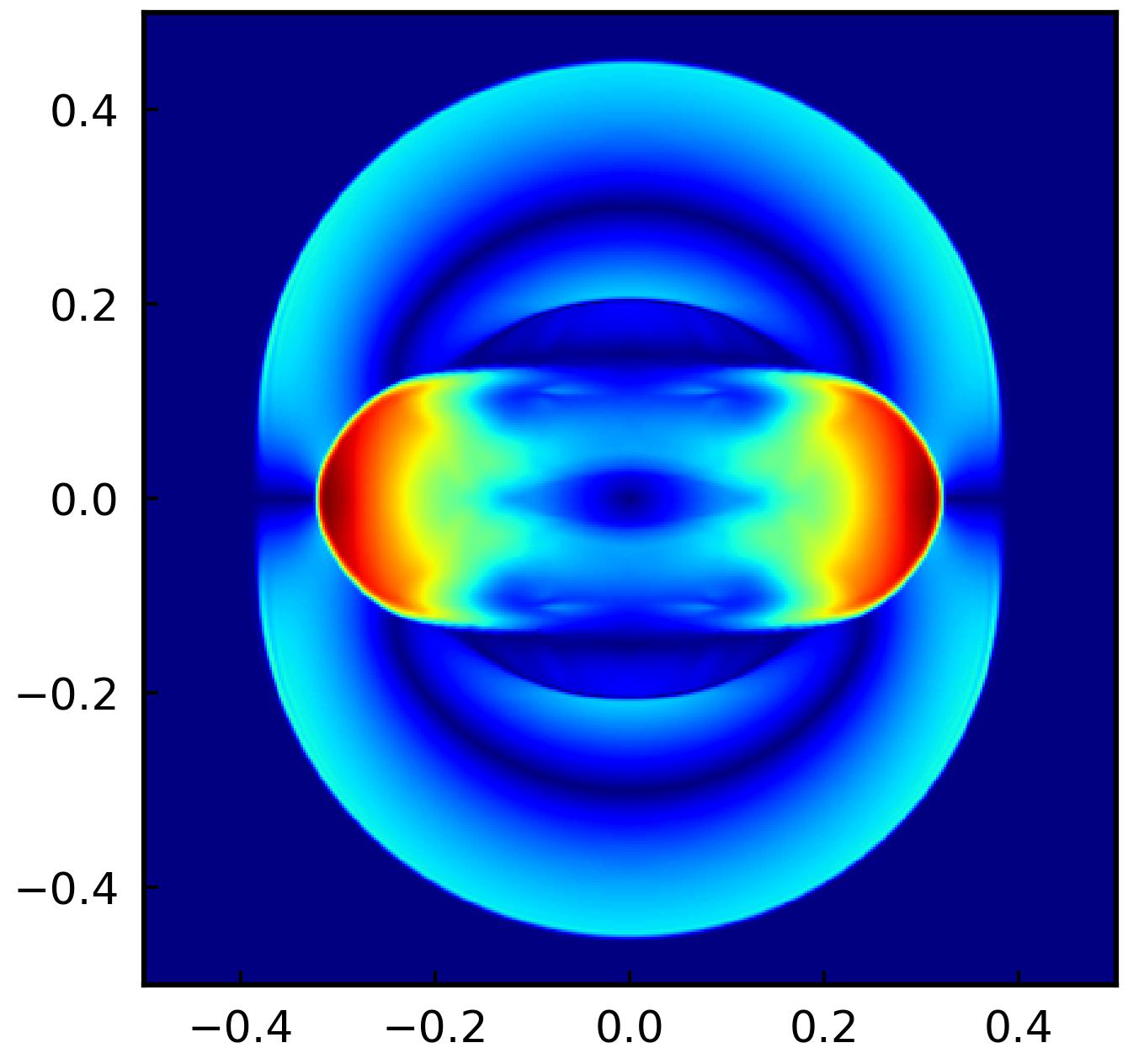}
			\end{subfigure}
			\caption{\Cref{Ex:Blast}: Plots of density (top-left), thermal pressure (top-right), magnetic pressure (bottom-left), and velocity magnitude (bottom-right) for the blast problem at $t = 0.01$.}
			\label{fig:Ex-Blast}
		\end{figure}
		
		\Cref{fig:Ex-Blast} shows the solution contours at time $t = 0.01$. The numerical results are consistent with those reported in previous studies~\cite{BalsaraSpicer1999,Christlieb2015PP,Li2011,WuShu2018,WuShu2019}. As observed, the solution features a radially expanding blast wave accompanied by an inward-propagating rarefaction. The shock fronts are sharply resolved, and no nonphysical negative pressures are encountered during our simulation. 
		These results highlight the robustness and accuracy of the proposed PAMPA scheme in capturing strongly magnetized, shock-dominated MHD flows while preserving physical admissibility.

		To further demonstrate the effectiveness of our shock indicator, we compare it with two existing approaches in \Cref{fig:Ex-Blast-troubled-comparison}. The left panel shows the troubled cells identified by the neural network-based indicator proposed in~\cite{DEEPRAY2019}. For this comparison, we use the publicly available pre-trained model provided by the authors of \cite{DEEPRAY2019} in their GitHub repository.\footnote{\url{https://deepray.github.io/codes}} The center panel displays the results obtained using the original indicator developed by Feng and Liu~\cite{LiuFeng2021}. 
		As observed, both existing methods tend to flag an excessive number of troubled cells, including regions where the solution remains smooth. In contrast, the proposed indicator, shown in the right panel, accurately captures the key discontinuities while substantially reducing false positives in smooth regions. This leads to improved computational efficiency and minimizes unnecessary numerical dissipation. 
		Moreover, we note that the troubled-cell distribution produced by the neural network-based indicator lacks the axial symmetry of the underlying problem---an artifact of its data-driven nature. By contrast, our method preserves the symmetry inherent to the blast wave configuration, further demonstrating its reliability and robustness.

		\begin{figure}[!thb]
			\centering
			\begin{subfigure}{0.3\textwidth}
				\includegraphics[width=\linewidth]{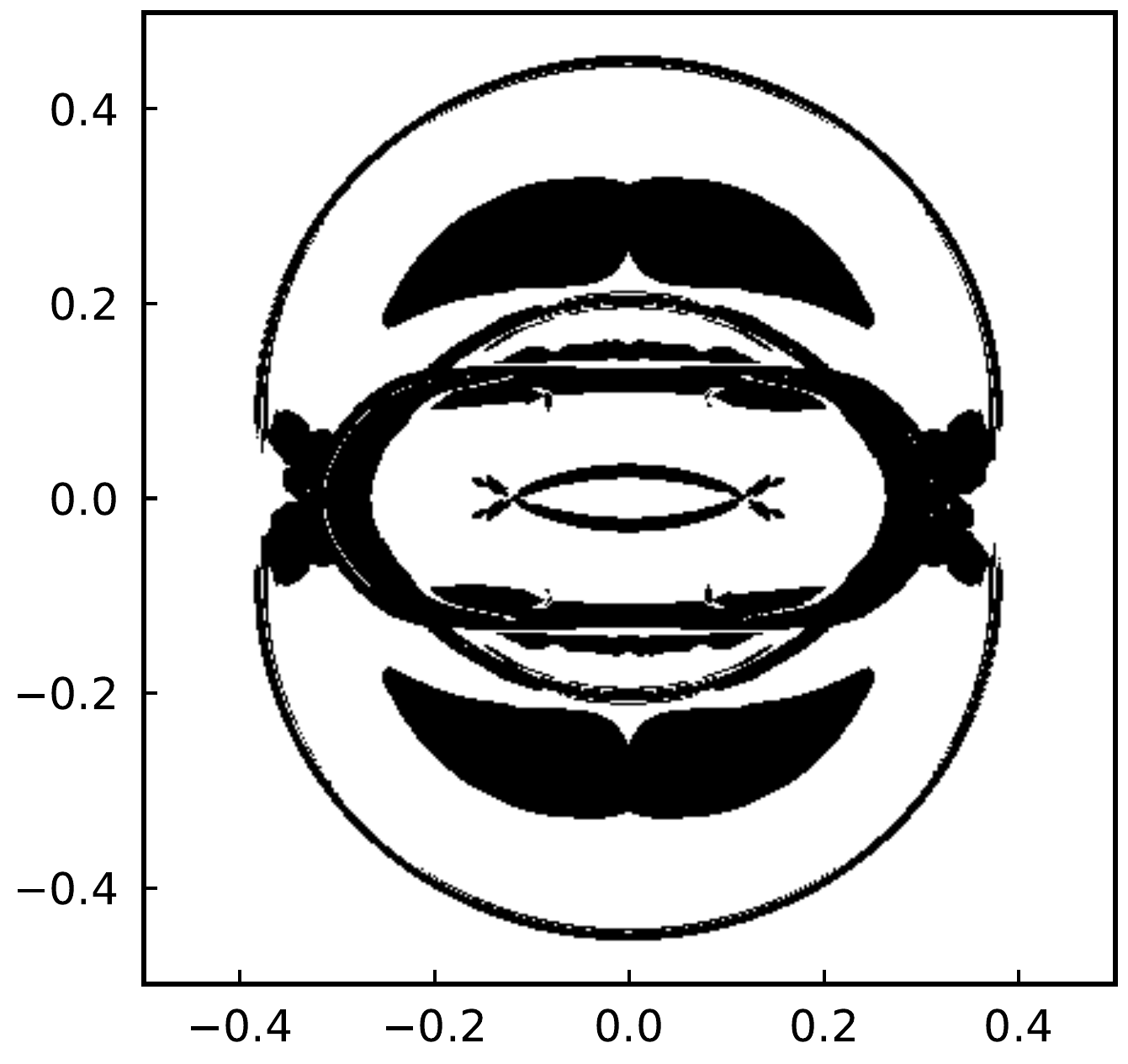}
				\caption{Neural network-based indicator~\cite{DEEPRAY2019}}
				\label{fig:trouble_cell_DeepRay}
			\end{subfigure}
			\hfill
			\begin{subfigure}{0.3\textwidth}
				\includegraphics[width=\linewidth]{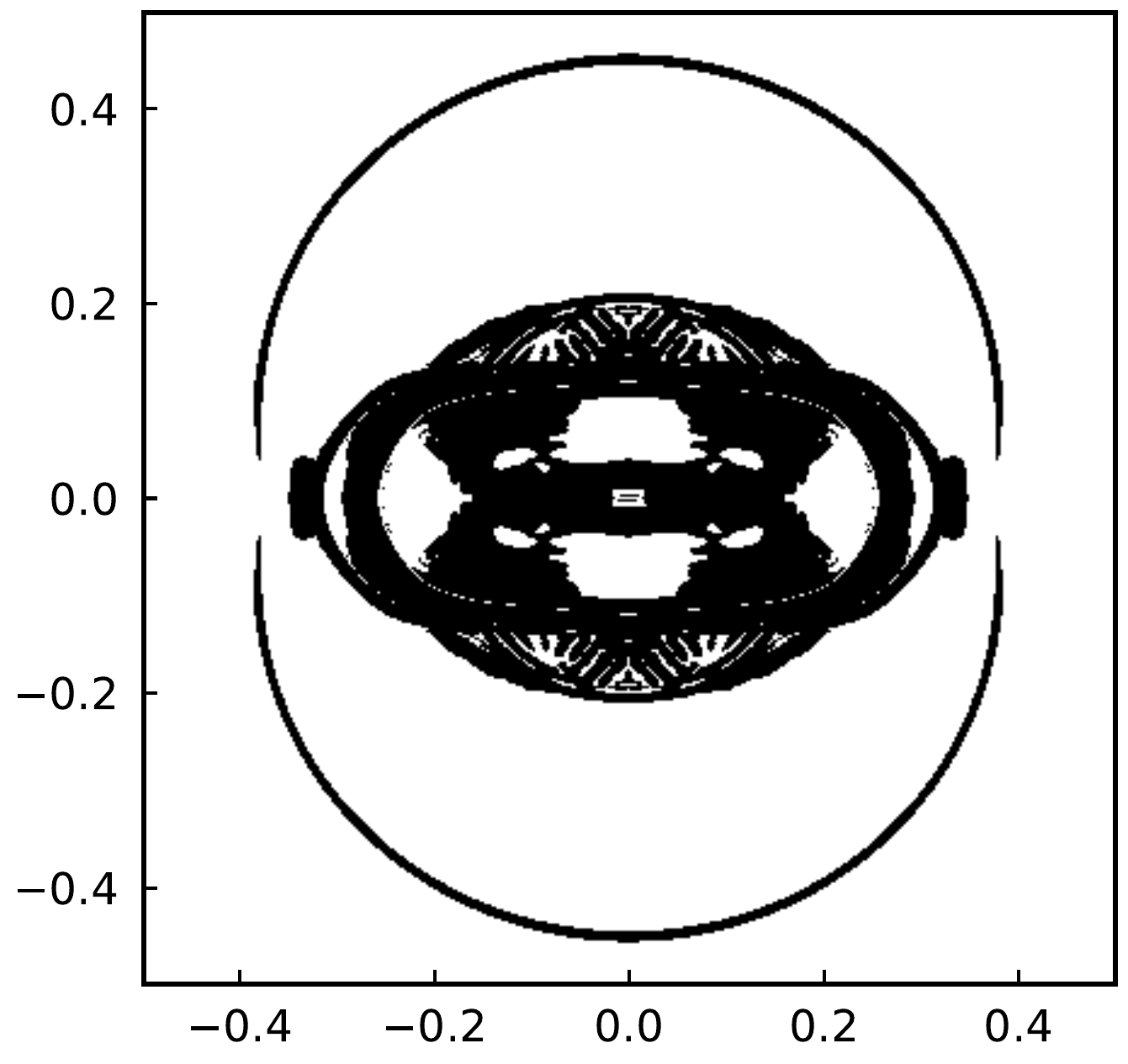}
				\caption{Indicator by Feng and Liu~\cite{LiuFeng2021}}
				\label{fig:trouble_cell_Liu}
			\end{subfigure}
			\hfill
			\begin{subfigure}{0.3\textwidth}
				\includegraphics[width=\linewidth]{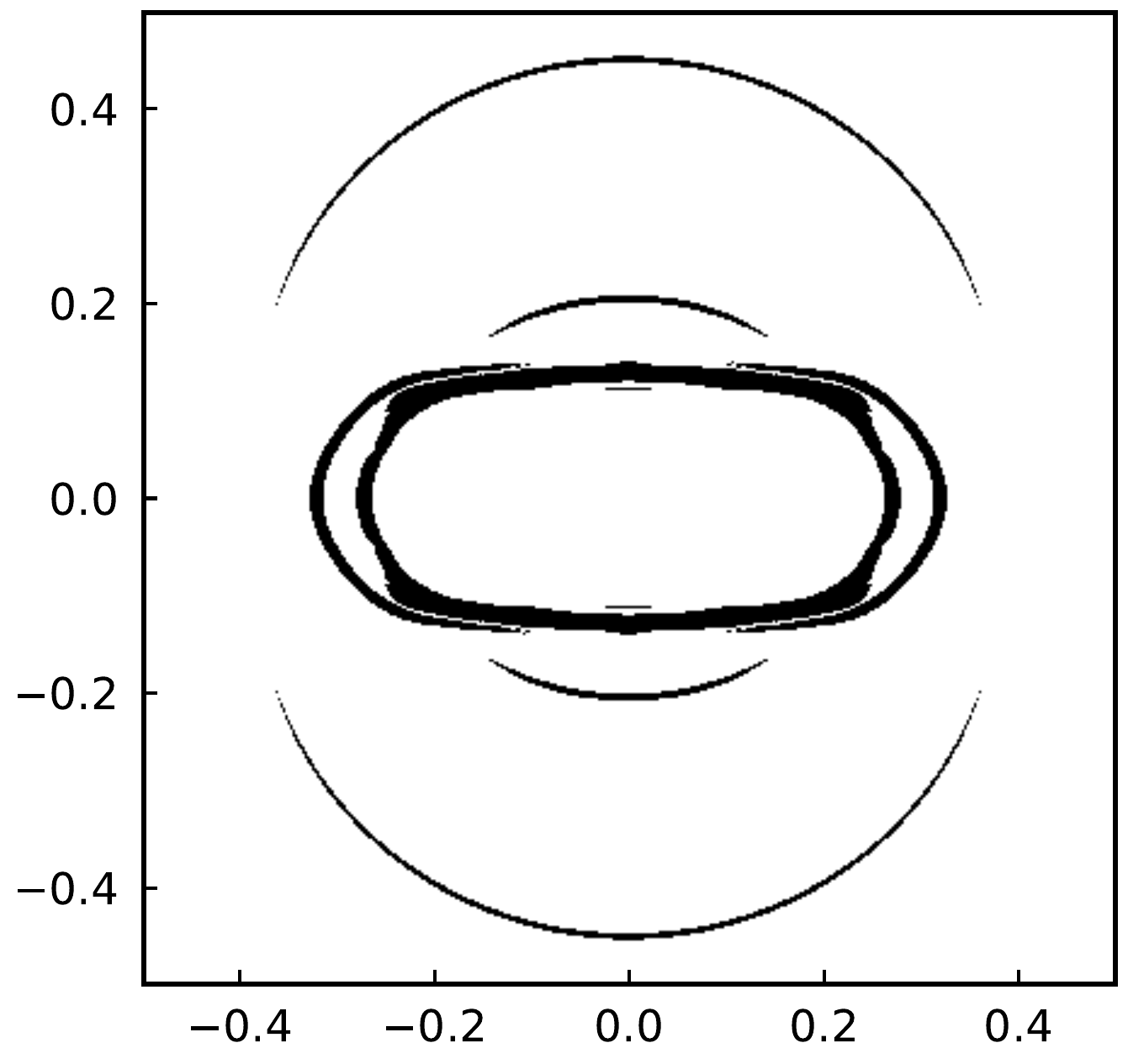}
				\caption{Proposed indicator}
				\label{fig:trouble_cell_Our}
			\end{subfigure}
			\caption{\Cref{Ex:Blast}: Comparison of troubled-cell indicators for the MHD blast wave problem at $t = 0.01$. The proposed indicator accurately identifies key discontinuities while minimizing false positives in smooth regions, demonstrating improved selectivity and symmetry preservation.}
			\label{fig:Ex-Blast-troubled-comparison}
		\end{figure}
		
	\end{expl}

	\begin{expl}[Blast Problem \uppercase\expandafter{\romannumeral2}]\label{Ex:Blast-strong} \rm
		We now consider a more demanding variant of the blast wave problem, originally proposed by Wu and Shu~\cite{WuShu2018}. This setup is similar to that of \Cref{Ex:Blast}, but with significantly increased explosion strength and magnetization: the initial central pressure is raised to $p_0 = 10^4$ and the magnetic field strength to $B_0 = \frac{1000}{\sqrt{4\pi}}$. These changes result in much stronger discontinuities and an extremely low plasma-beta, approximately $\beta \approx 2.51 \times 10^{-6}$---only 1\% of that in the standard configuration in \Cref{Ex:Blast}.
		
		The simulation is performed up to time $t = 0.001$ on a uniform $400 \times 400$ Cartesian grid. \Cref{fig:Ex-Blast-strong} displays the computed results. As shown, the proposed COE procedure effectively suppresses numerical oscillations. No nonphysical values of pressure or density are observed throughout the simulation. 
		Due to the greatly increased magnetization, the external fast shock becomes significantly weaker and is no longer visible in the density contours. This behavior is consistent with previous observations reported in~\cite{WuShu2018,WuShu2019}. This example underscores the robustness of our scheme under strongly magnetized, low-plasma-beta conditions.

		\begin{figure}[!thb]
			\centering		
			\begin{subfigure}{0.48\textwidth}
				\includegraphics[width=\textwidth]{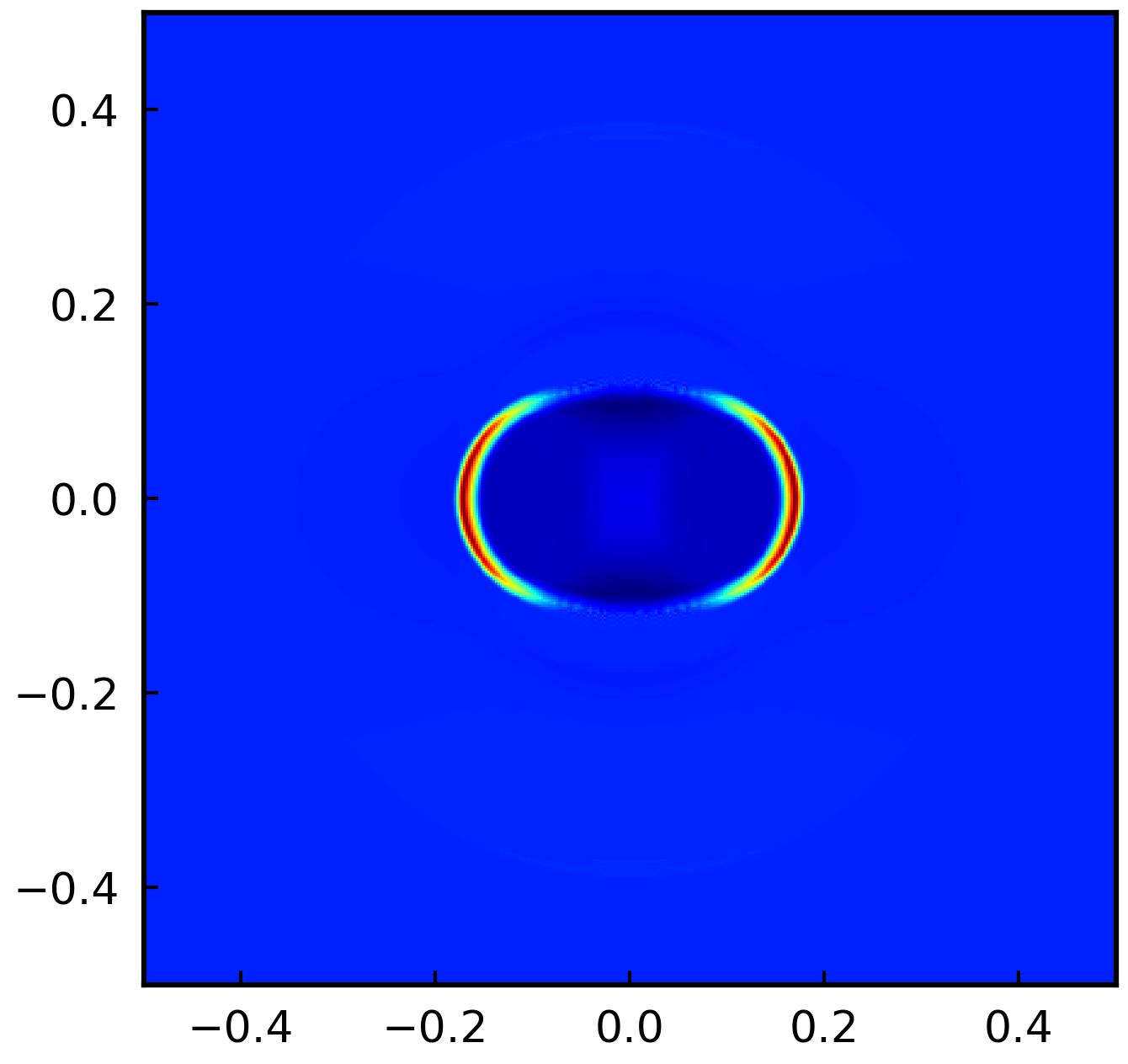}
			\end{subfigure}
			\hfill
			\begin{subfigure}{0.48\textwidth}
				\includegraphics[width=\textwidth]{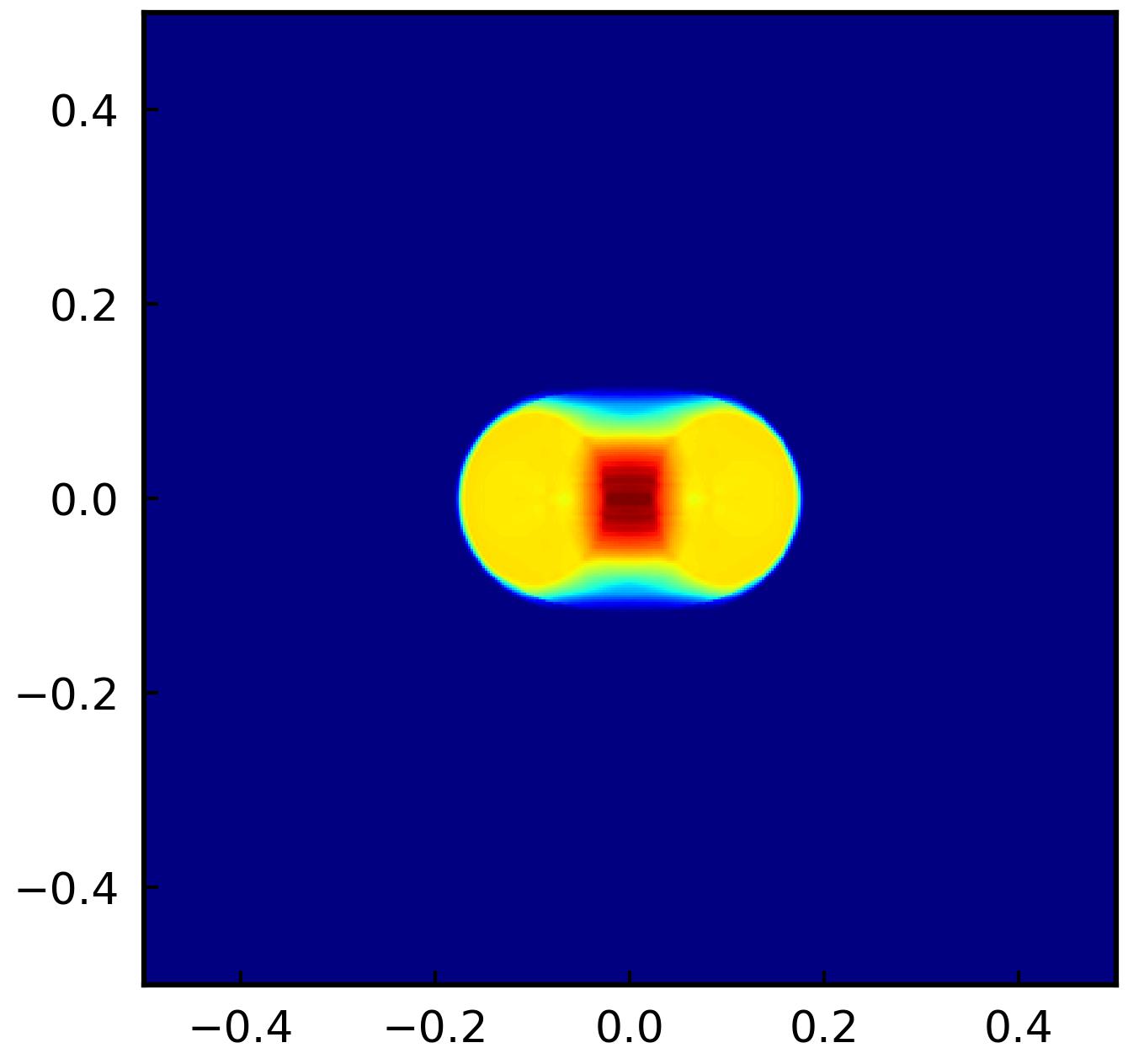}
			\end{subfigure}
			
			\begin{subfigure}{0.48\textwidth}
				\includegraphics[width=\textwidth]{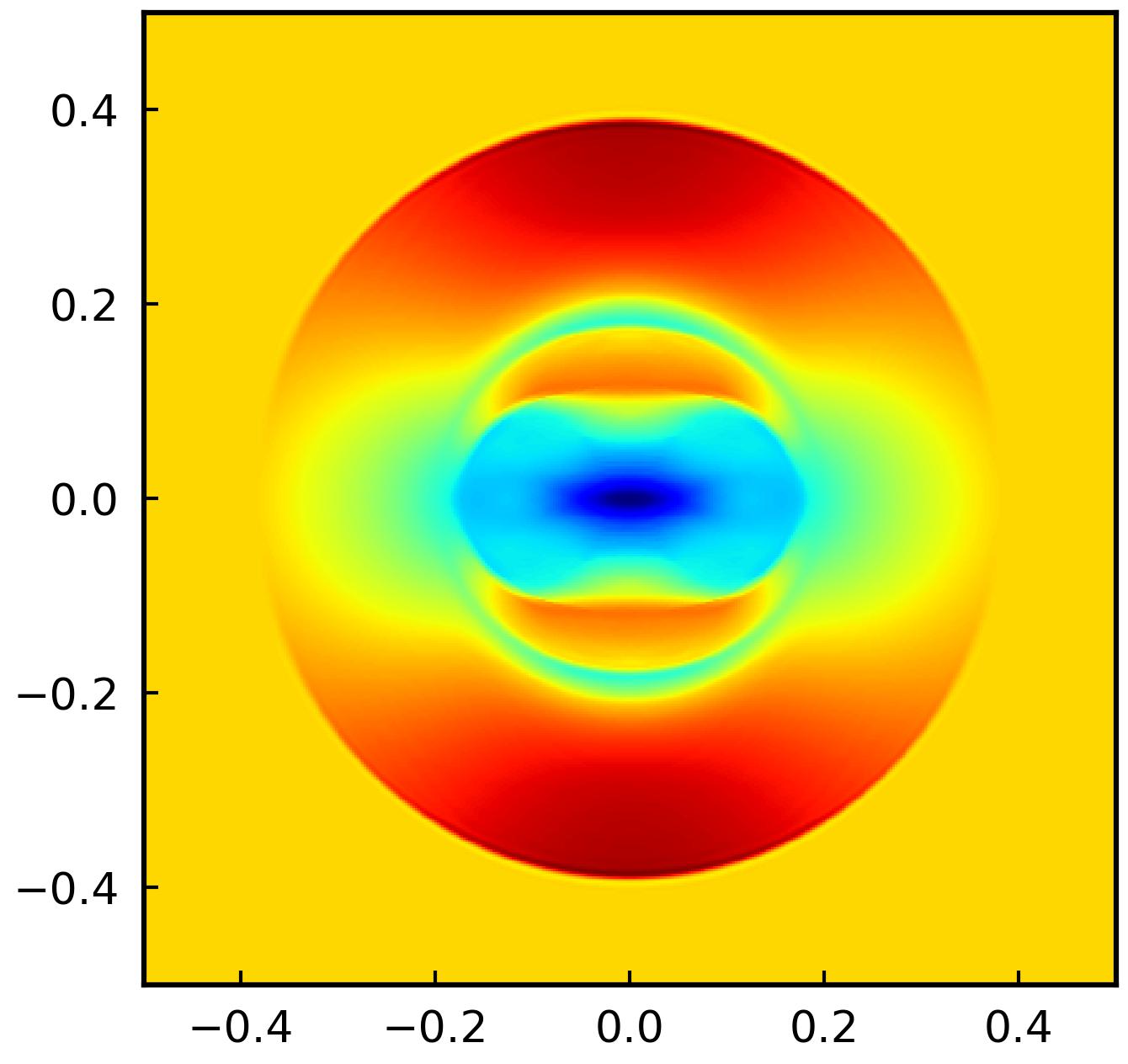}
			\end{subfigure}
			\hfill
			\begin{subfigure}{0.48\textwidth}
				\includegraphics[width=\textwidth]{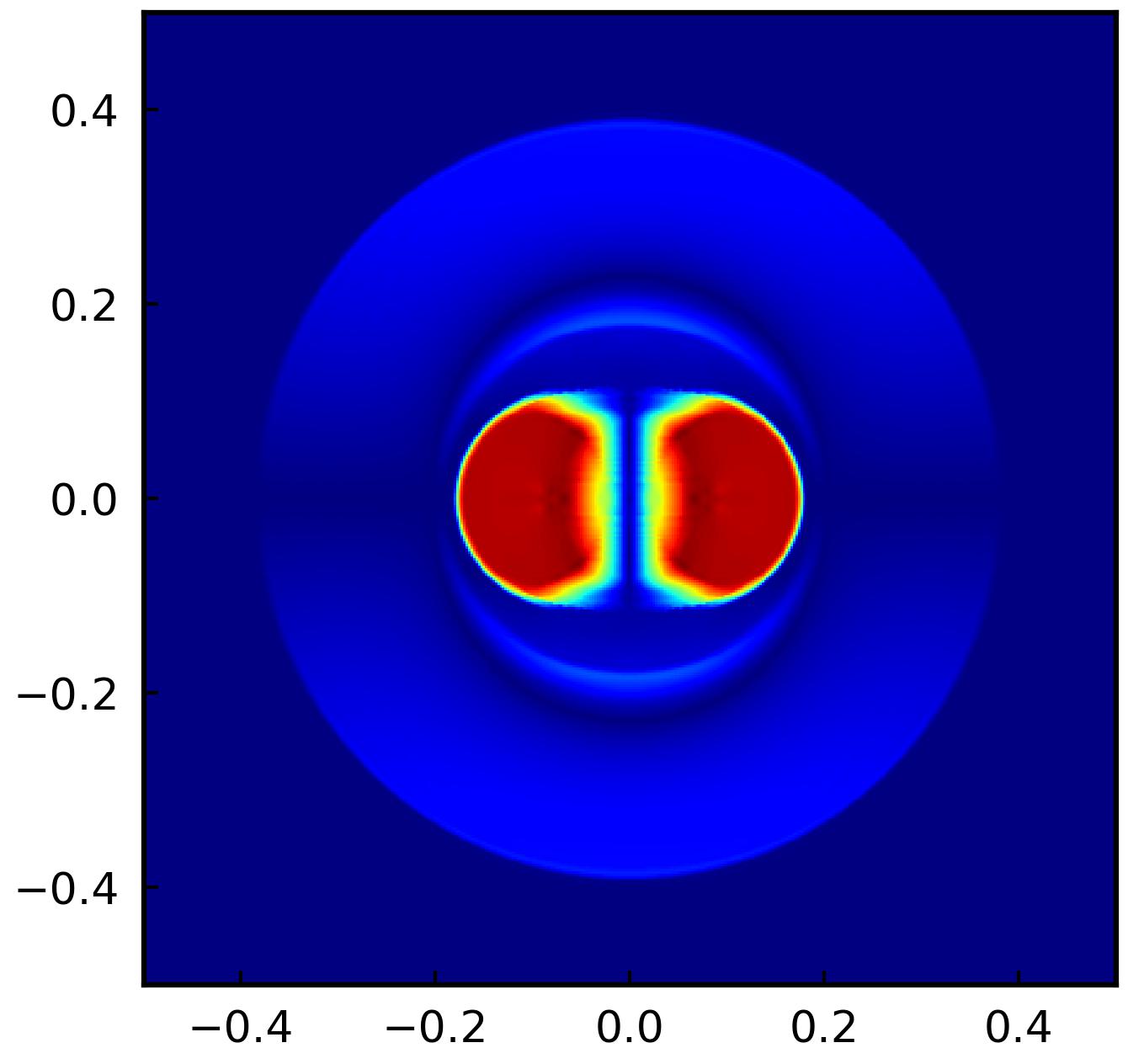}
			\end{subfigure}
			\caption{\Cref{Ex:Blast-strong}: Plots of density (top-left), thermal pressure (top-right), magnetic pressure (bottom-left), and velocity magnitude (bottom-right) for the blast problem \uppercase\expandafter{\romannumeral2} at $t = 0.001$.}
			\label{fig:Ex-Blast-strong}
		\end{figure}
	\end{expl}

	\begin{expl}[Shock–Cloud Interaction]\label{Ex:ShockCloud} \rm
		The shock–cloud interaction problem models the disruption of a high-density cloud by a strong shock wave, resulting in the formation of multiple discontinuities and the development of small-scale flow instabilities. We adopt the configuration from~\cite{Dai1998,jiang1998nonoscillatory,WuShu2018,WuShu2019}. The initial condition consists of left and right states separated by a vertical discontinuity at $x = 0.6$:
		\begin{equation*}
			(\rho, {\bm v}, {\bm B}, p) = 
			\begin{cases}
				(3.86859,\, 0,\, 0,\, 0,\, 0,\, 2.1826182,\ -2.1826182,\, 167.345), & x < 0.6, \\
				(1,\, -11.2536,\, 0,\, 0,\, 0,\, 0.56418958,\ 0.56418958,\, 1), & x \ge 0.6.
			\end{cases}
		\end{equation*}
		A circular cloud of elevated density ($\rho = 10$) is embedded in the right state, centered at $(0.8, 0.5)$ with radius $0.15$.
		The computational domain is $[0, 1]^2$, discretized using a uniform $400 \times 400$ Cartesian mesh. Inflow boundary conditions are imposed on the right boundary, and outflow conditions on the remaining three boundaries. 
		
		The numerical results at time $t = 0.06$ are shown in \Cref{fig:Ex-ShockCloud}. The simulation captures complex flow structures—including bow shocks, reflected shocks, contact discontinuities, and small-scale features—with high resolution. The results are in good agreement with previous studies~\cite{Toth2000,jiang1998nonoscillatory,Balbas2006,WuShu2019}. 
		In addition, \Cref{fig:Ex-ShockcloudTrouble} shows the distribution of troubled cells identified by our shock indicator. Using only the two fast magnetoacoustic wave speeds effectively captures strong discontinuities while keeping the false detection rate low in smooth regions. By contrast, including additional characteristic waves (e.g., the advection wave) leads to over-detection, flagging troubled cells even in relatively smooth areas. This comparison suggests that a fast-wave-only strategy in our indicator offers a balanced and effective design choice for shock detection. 
		As in previous test cases, no negative pressure or density values are observed during the simulation, further confirming the robustness and PP property of our scheme.

		\begin{figure}[!thb]
			\centering		
			\begin{subfigure}{0.48\textwidth}
				\includegraphics[width=\textwidth]{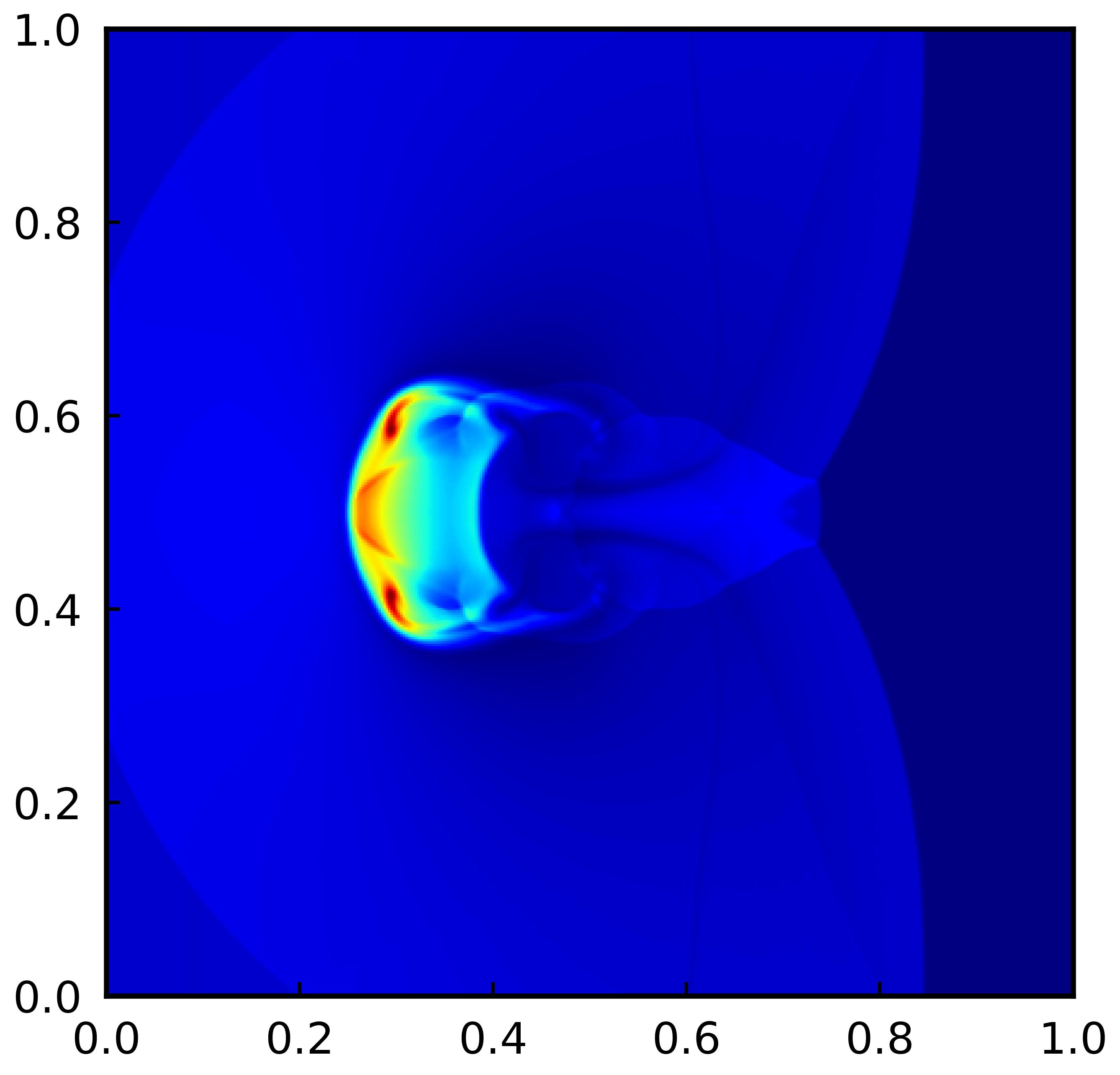}
			\end{subfigure}
			\hfill
			\begin{subfigure}{0.48\textwidth}
				\includegraphics[width=\textwidth]{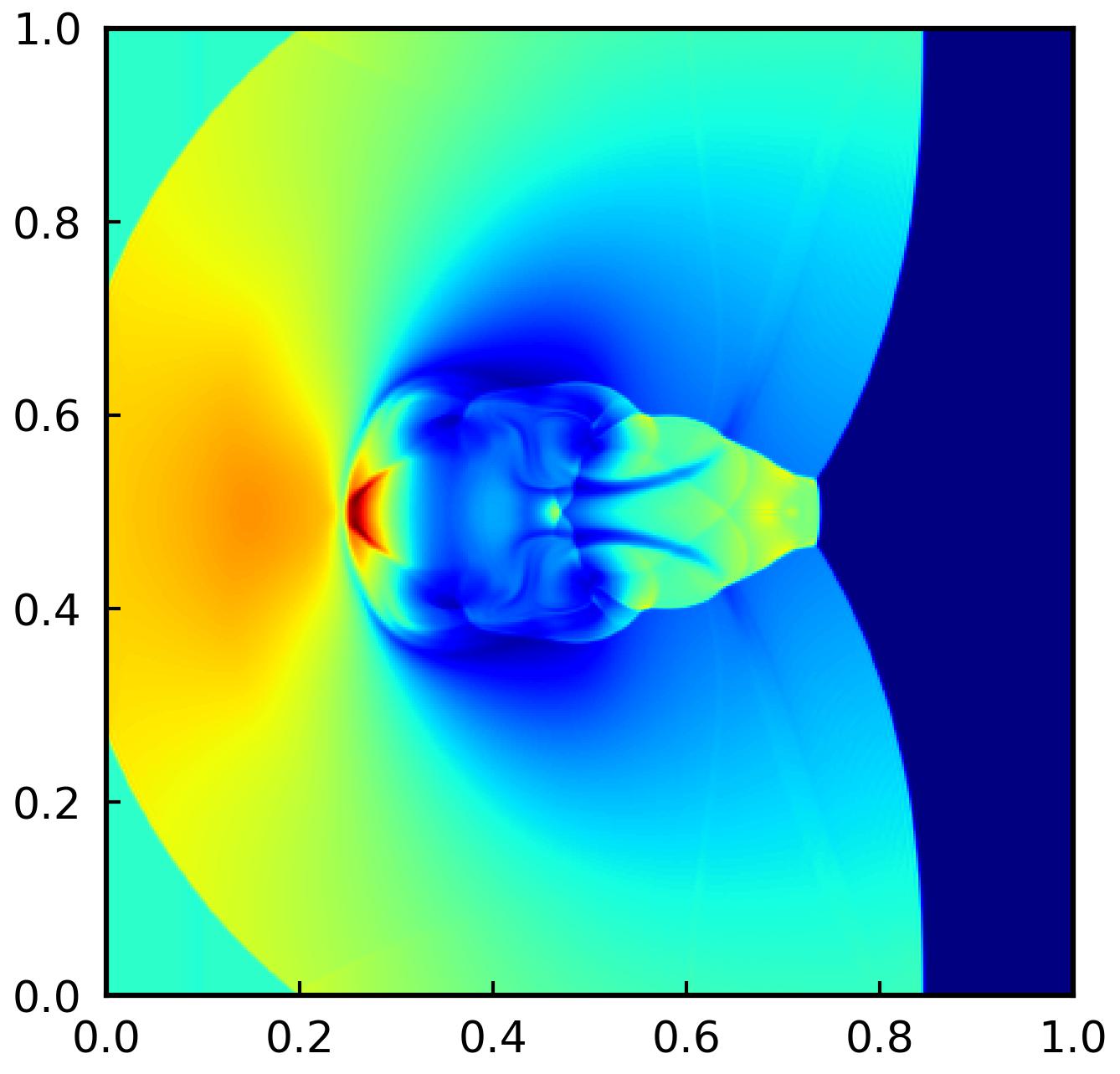}
			\end{subfigure}
			
			\begin{subfigure}{0.48\textwidth}
				\includegraphics[width=\textwidth]{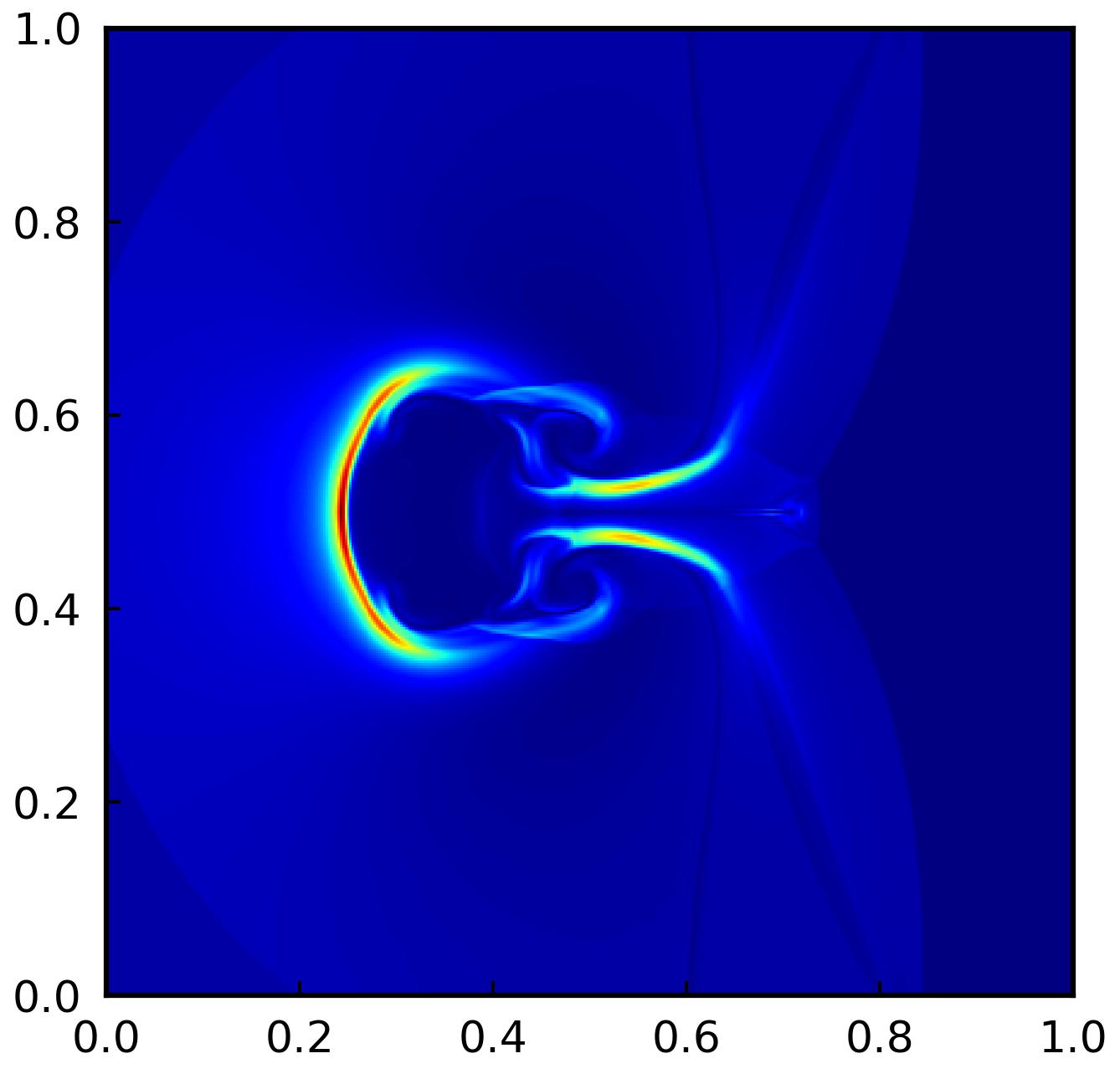}
			\end{subfigure}
			\hfill
			\begin{subfigure}{0.48\textwidth}
				\includegraphics[width=\textwidth]{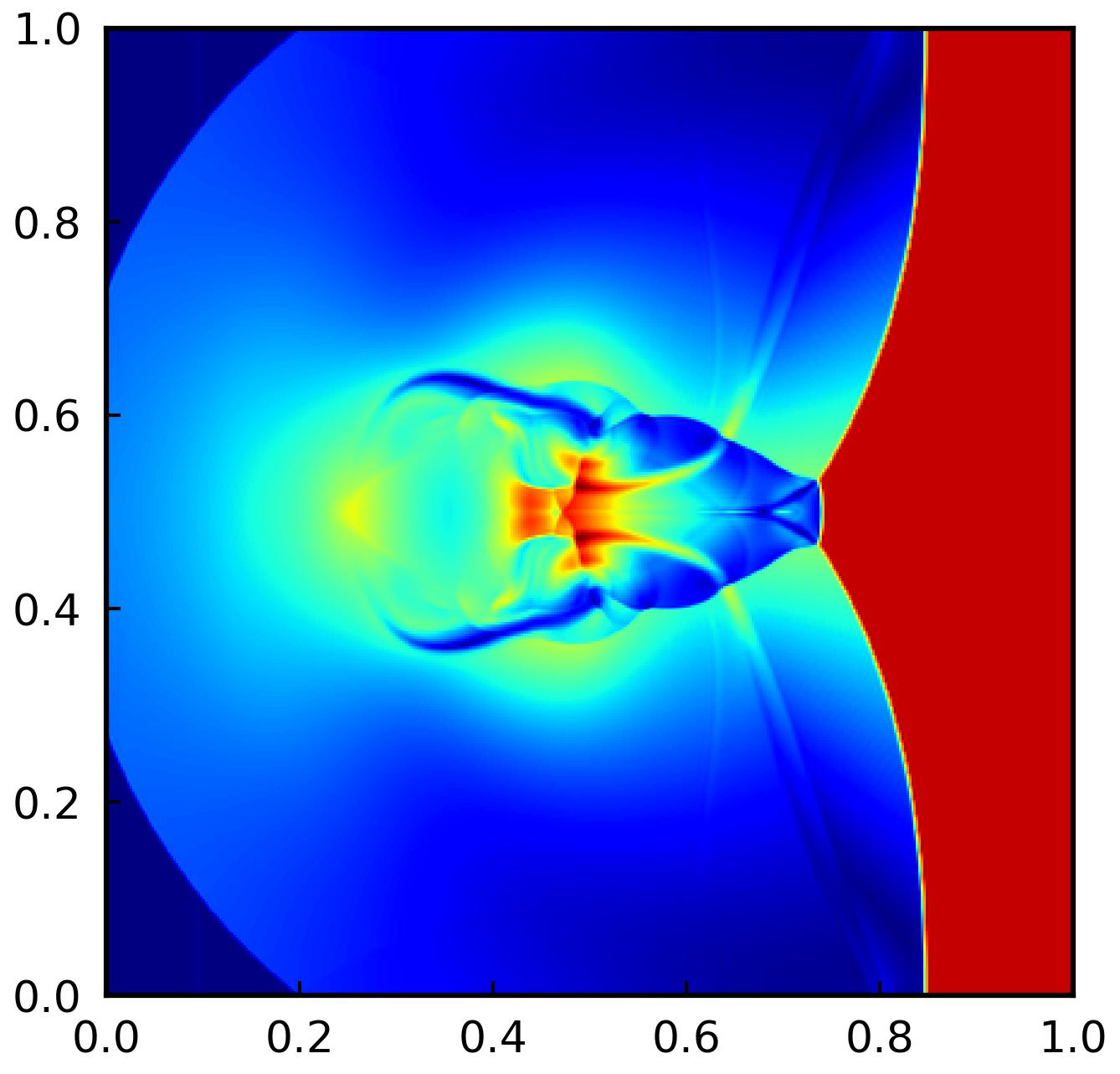}
			\end{subfigure}
			\caption{\Cref{Ex:ShockCloud}: The density (top-left), thermal pressure (top-right), magnetic pressure (bottom-left), and velocity magnitude (bottom-right) for the shock–cloud interaction problem at $t = 0.06$. 
			}
			\label{fig:Ex-ShockCloud}
		\end{figure}
		
		\begin{figure}[!htb]
			\centering		
			\begin{subfigure}{0.48\textwidth}
				\includegraphics[width=\textwidth]{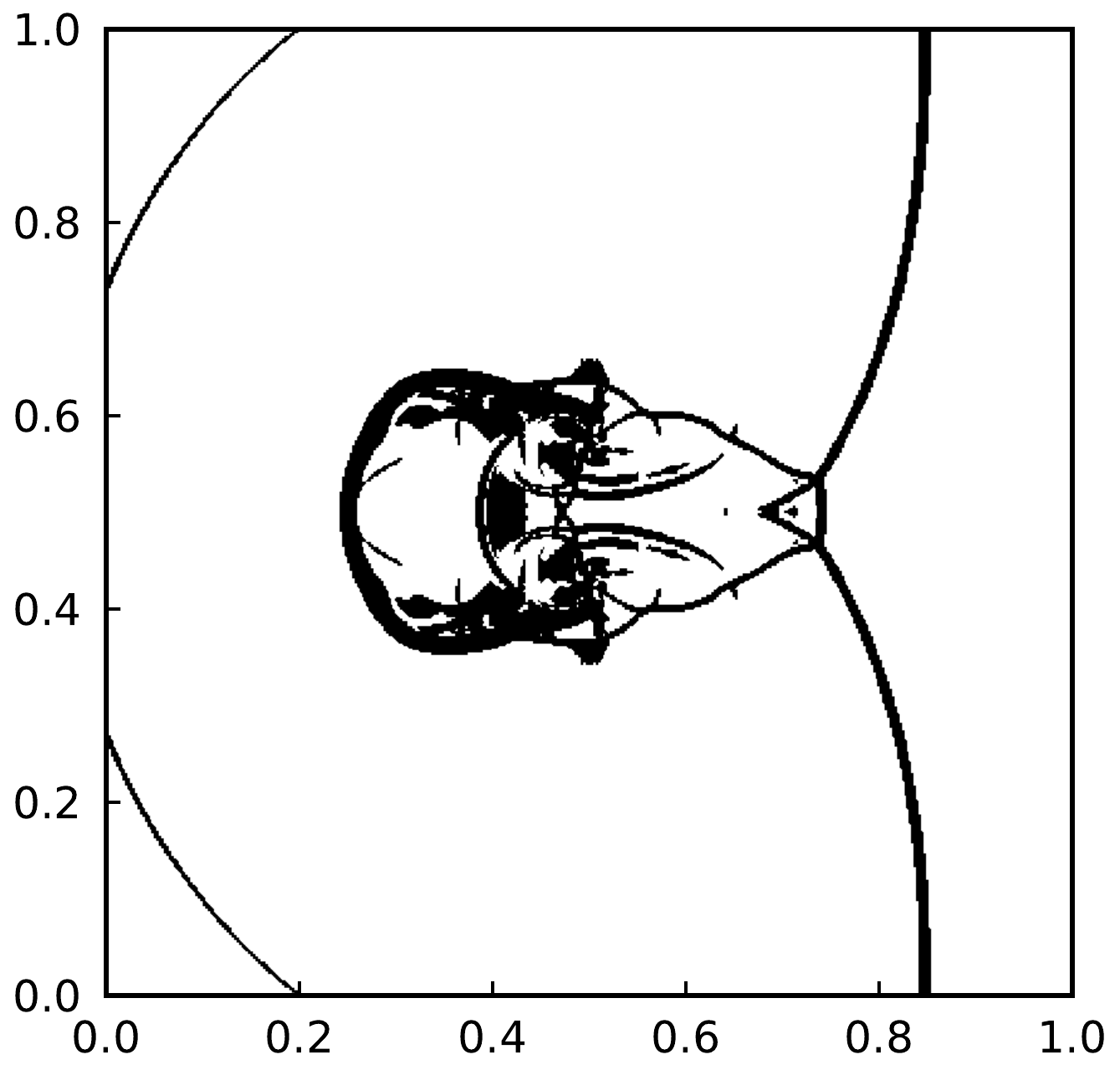}
			\end{subfigure}
			\hfill
			\begin{subfigure}{0.48\textwidth}
				\includegraphics[width=\textwidth]{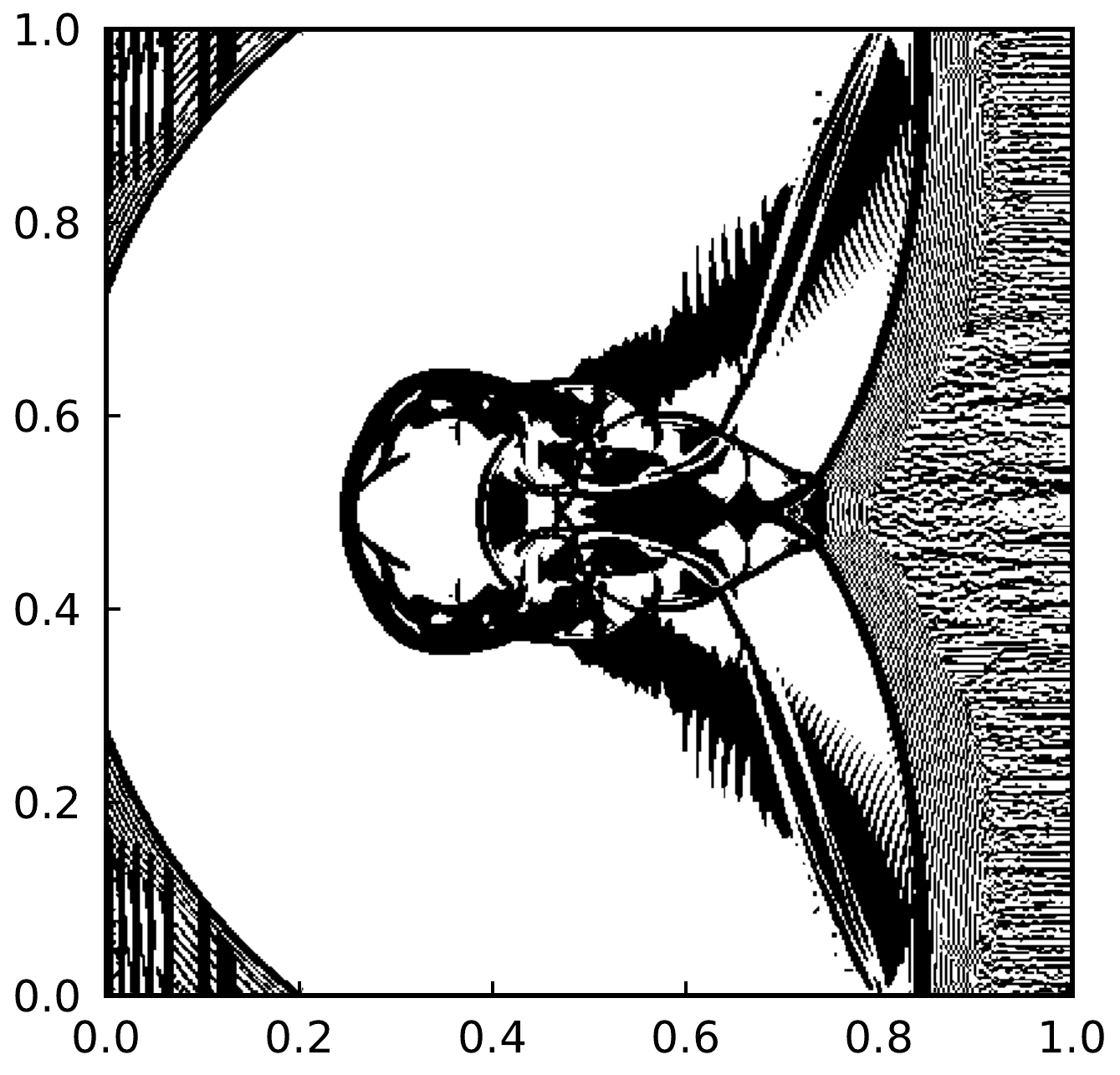}
			\end{subfigure}
			\caption{\Cref{Ex:ShockCloud}: Troubled-cell distributions identified based on two fast wave speeds (left) and three wave speeds (right) at $t = 0.06$.}
			\label{fig:Ex-ShockcloudTrouble}
		\end{figure} 
	\end{expl}

	\begin{figure}[!thb]
		\centering
		\begin{subfigure}{0.32\textwidth}
			\includegraphics[scale=0.32]{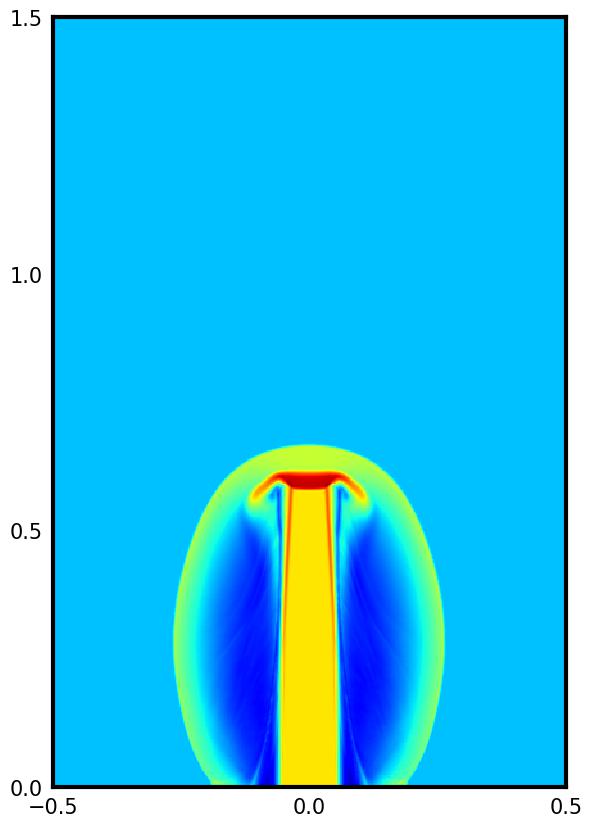}
		\end{subfigure}
		\hfill
		\begin{subfigure}{0.32\textwidth}
			\includegraphics[scale=0.32]{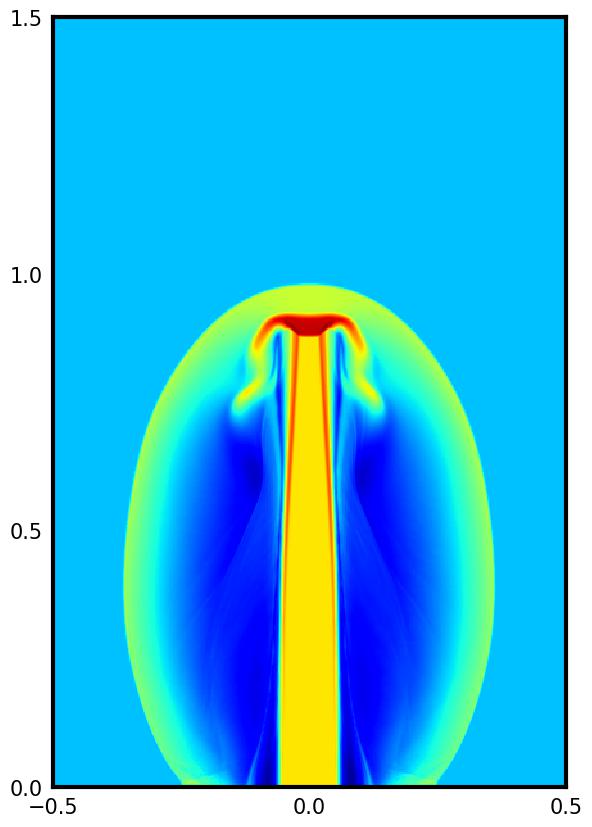}
		\end{subfigure}
		\hfill
		\begin{subfigure}{0.32\textwidth}
			\includegraphics[scale=0.32]{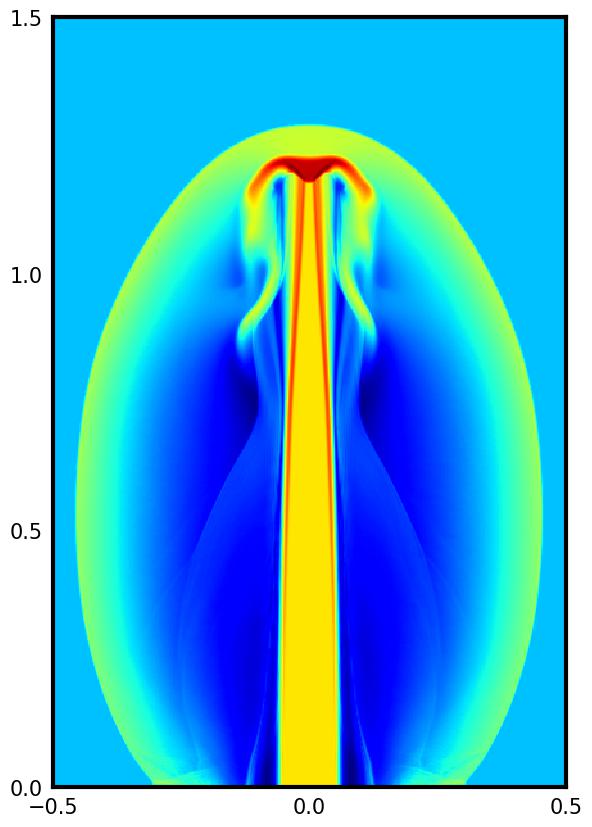}
		\end{subfigure}
		
		\begin{subfigure}{0.32\textwidth}
			\includegraphics[scale=0.32]{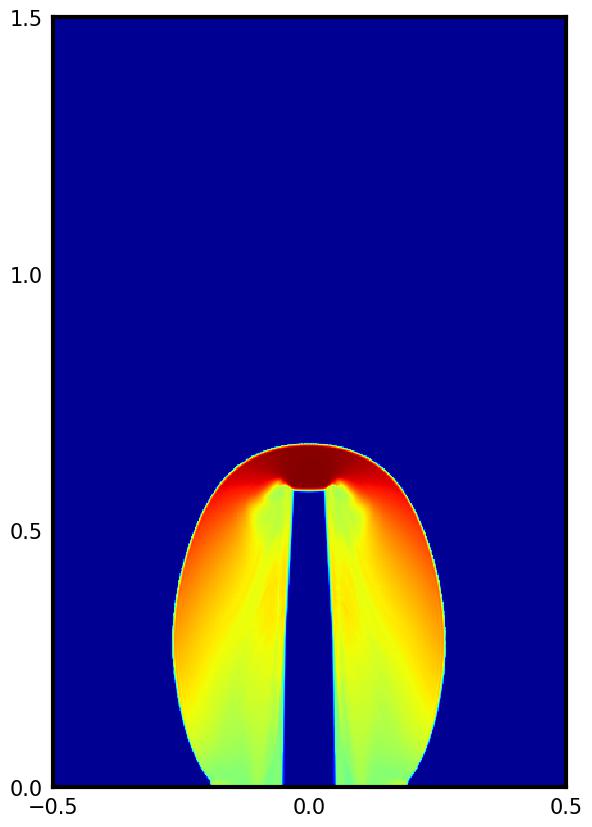}
		\end{subfigure}
		\hfill
		\begin{subfigure}{0.32\textwidth}
			\includegraphics[scale=0.32]{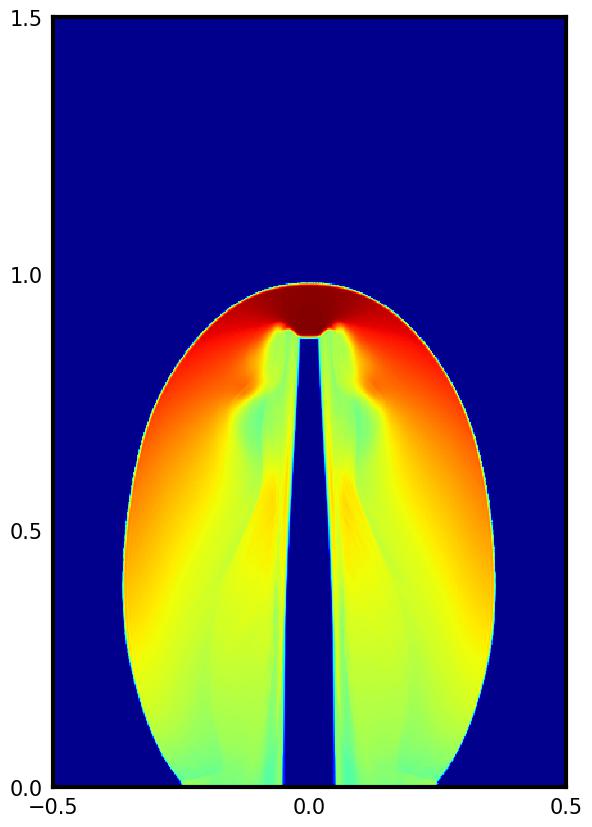}
		\end{subfigure}
		\hfill
		\begin{subfigure}{0.32\textwidth}
			\includegraphics[scale=0.32]{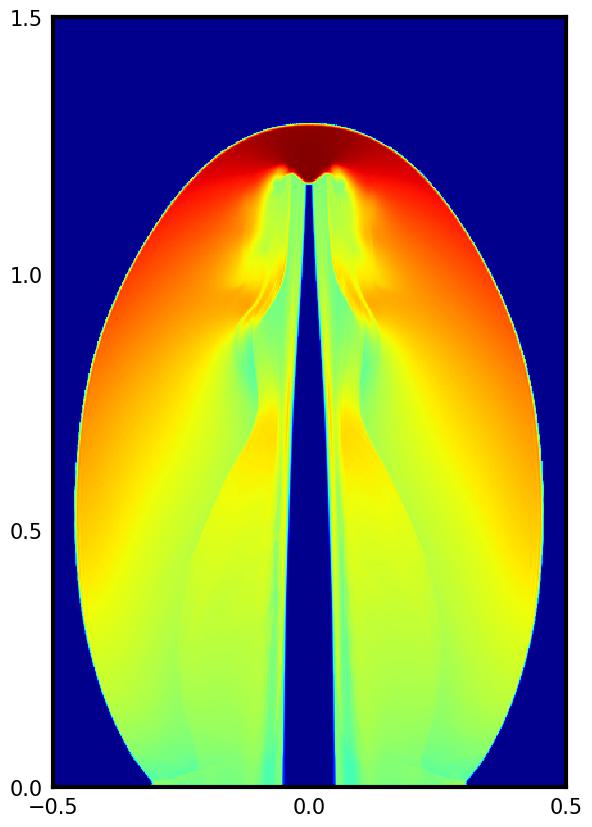}
		\end{subfigure}
		\caption{Mach 800 jet problem with $B_0 = \sqrt{200}$: Density logarithm and pressure logarithm (bottom) at $t = 0.001, 0.0015$, and $0.002$ (from left to right).
		}
		\label{fig:Ex-Jet_800_200}
	\end{figure} 
	
	\begin{figure}[!htb]
		\centering	
		\begin{subfigure}{0.32\textwidth}
			\includegraphics[scale=0.32]{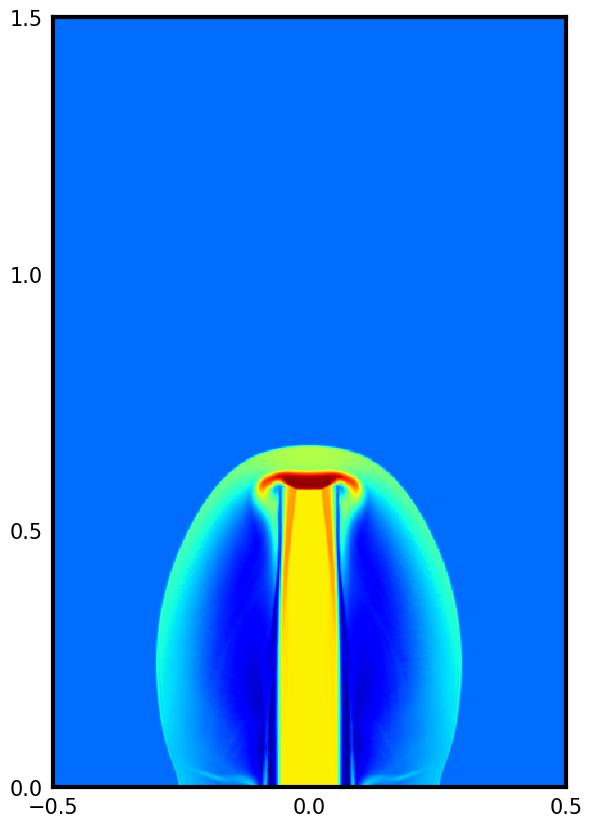}
		\end{subfigure}
		\hfill
		\begin{subfigure}{0.32\textwidth}
			\includegraphics[scale=0.32]{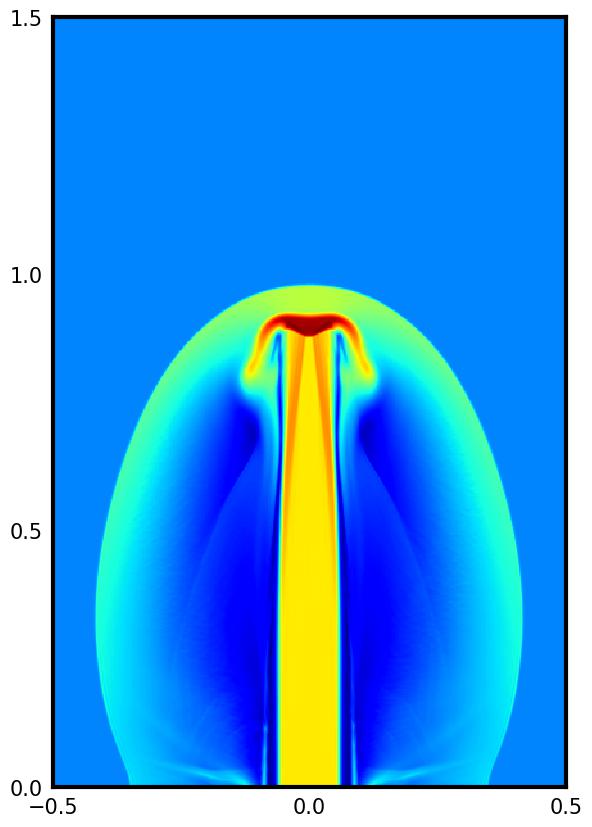}
		\end{subfigure}
		\hfill
		\begin{subfigure}{0.32\textwidth}
			\includegraphics[scale=0.32]{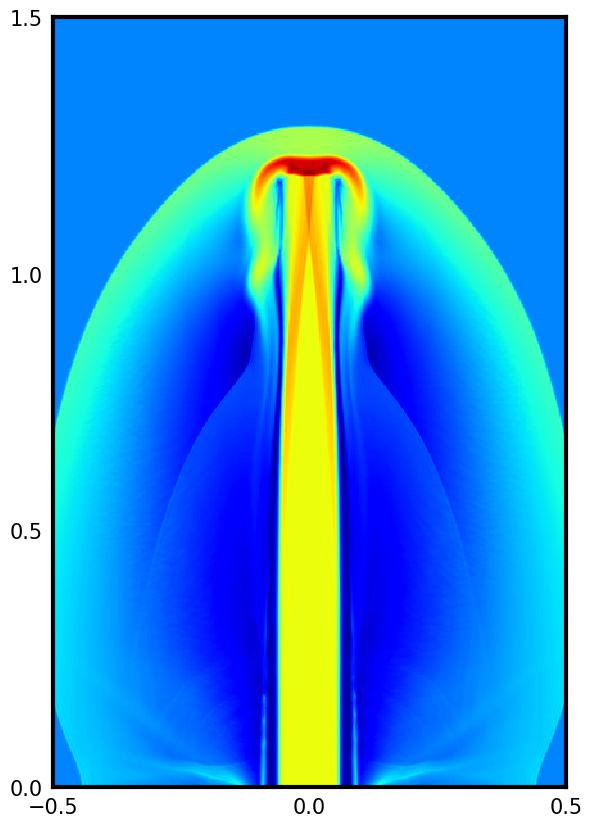}
		\end{subfigure}
		
		\begin{subfigure}{0.32\textwidth}
			\includegraphics[scale=0.32]{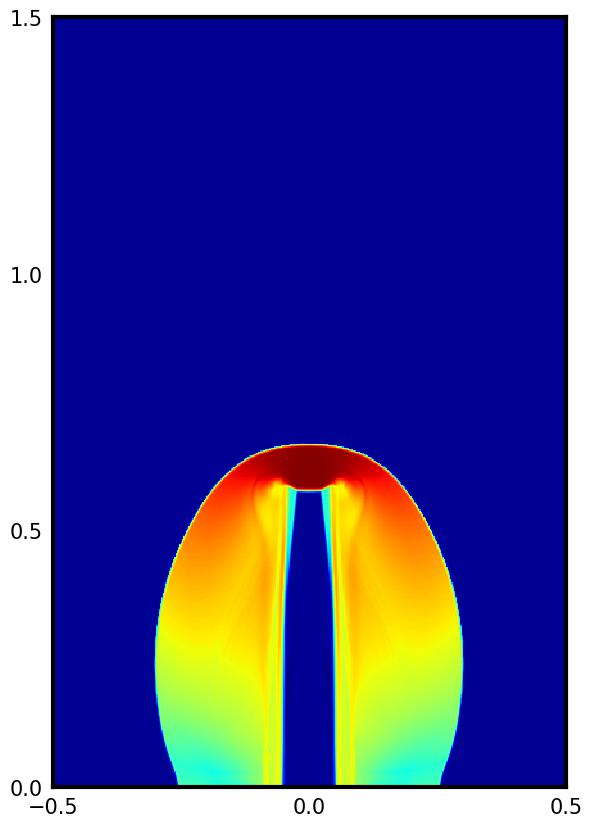}
		\end{subfigure}
		\hfill
		\begin{subfigure}{0.32\textwidth}
			\includegraphics[scale=0.32]{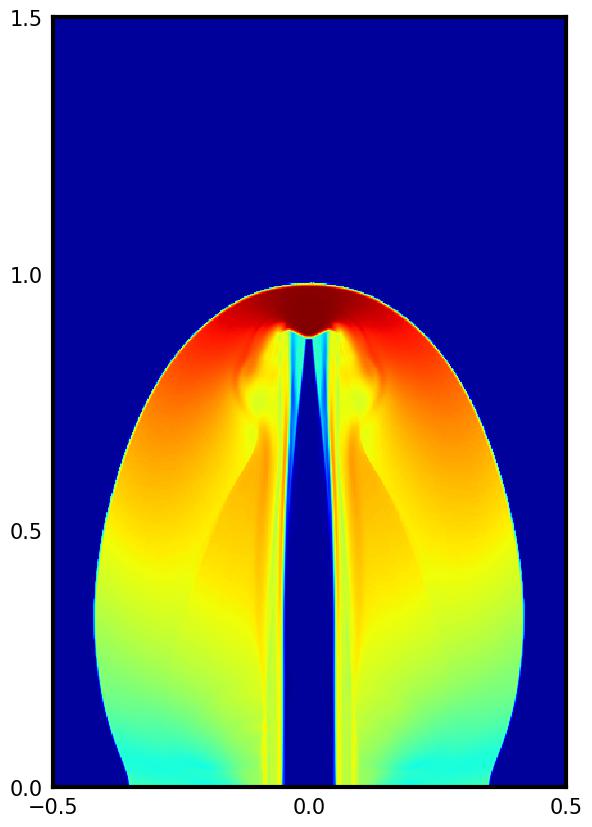}
		\end{subfigure}
		\hfill
		\begin{subfigure}{0.32\textwidth}
			\includegraphics[scale=0.32]{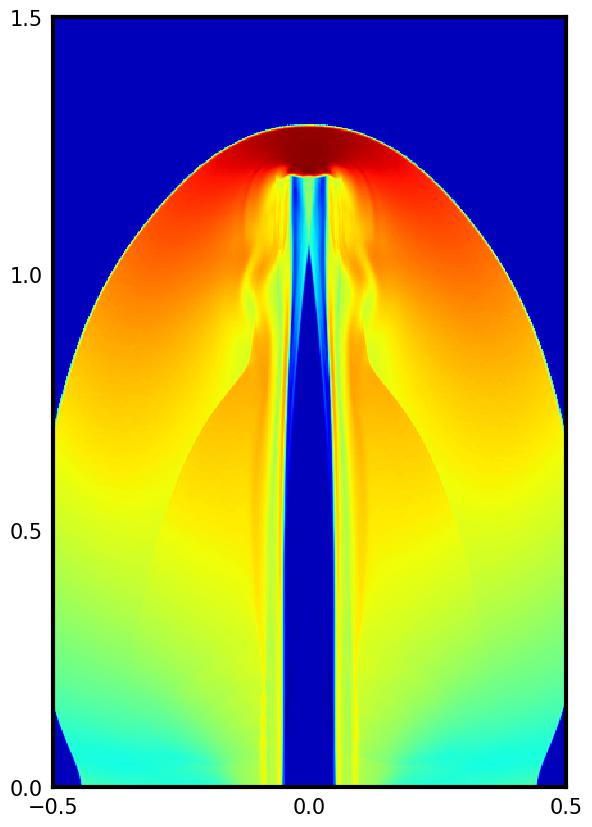}
		\end{subfigure}
		
		\caption{Mach 800 jet with $B_0 = \sqrt{2000}$: Density logarithm (top) and pressure logarithm (bottom) at $t = 0.001, 0.0015$, and $0.002$ (from left to right).
		}
		\label{fig:Ex-Jet_800_2000}
	\end{figure}

	\begin{expl}[Astrophysical Jets]\label{Ex:Jet} \rm
		Our final test case considers the high Mach number MHD jet problem, a widely used PP benchmark introduced in~\cite{WuShu2018,WuShu2019}. This problem presents significant challenges for numerical MHD schemes due to the combined effects of supersonic jet propagation, strong magnetic fields, and high kinetic energy. The interaction of strong shocks, shear layers, and interface instabilities often leads to numerical difficulties—most notably the emergence of nonphysical negative pressure values—making this a stringent test for both robustness and accuracy.

		We consider the Mach 800 jet problem on the computational domain $[-0.5, 0.5] \times [0, 1.5]$, discretized uniformly with $400 \times 600$ cells. The initial condition is specified as
		\[
		(\rho, {\bm v}, {\bm B}, p) = \left(0.1\gamma,\, 0,\, 0,\, 0,\, 0,\, B_0,\, 0,\, 1\right), \qquad \text{with} \quad \gamma = 1.4.
		\]
		At the bottom boundary, a narrow inflow region defined by $|x| < 0.05$ is prescribed, where the jet is injected with the following state:
		\[
		(\rho, {\bm v}, {\bm B}, p) = \left(\gamma,\, 0,\, 800,\, 0,\, 0,\, B_0,\, 0,\, 1\right).
		\]
		All remaining boundaries are treated as outflow.

		Following the setup in~\cite{WuShu2018,WuShu2019}, we examine the performance of the proposed numerical method under three levels of magnetization by varying the magnetic field strength $B_0$:
		\begin{itemize}
			\item[(i)] $B_0 = \sqrt{200}$, corresponding to a plasma-beta $\beta_0 = 10^{-2}$;
			\item[(ii)] $B_0 = \sqrt{2000}$, corresponding to $\beta_0 = 10^{-3}$;
			\item[(iii)] $B_0 = \sqrt{20000}$, corresponding to $\beta_0 = 10^{-4}$.
		\end{itemize}
		As $B_0$ increases (i.e., as the plasma-beta $\beta_0$ decreases), the dynamics become increasingly dominated by magnetic forces. This leads to significantly greater numerical stiffness and a higher risk of nonphysical solutions, posing a stringent test for robustness.
		
		Figures~\ref{fig:Ex-Jet_800_200},~\ref{fig:Ex-Jet_800_2000}, and~\ref{fig:Ex-Jet_800_20000} show the results for the three cases, including logarithmic density contours, pressure contours and the distribution of troubled cells. As seen in these figures, key features such as Mach shocks, shear layers, and contact discontinuities are sharply captured without visible nonphysical oscillations. No negative pressure or density values are observed in any of the cases, demonstrating the robustness and reliability of the proposed scheme under highly magnetized, extreme flow conditions.

		The bottom row of Figure \ref{fig:Ex-Jet_800_20000} displays the distribution of troubled cells detected by our shock indicator. As observed, troubled cells are predominantly concentrated near the jet head, bow shock, and internal shear layers---regions characterized by strong gradients and discontinuities. In contrast, the indicator remains inactive in smooth regions, thereby avoiding unnecessary application of the COE limiting procedure and preserving high-order accuracy where it is not needed. 
		Moreover, the detected troubled-cell patterns exhibit symmetry consistent with the physical structure of the jet, demonstrating that the indicator preserves the inherent symmetry of the solution. These results highlight the accuracy and reliability of our troubled-cell indicator in guiding the selective activation of the COE procedure.

		To further evaluate the robustness of our scheme under even more extreme conditions, we simulate the Mach 2000 and Mach 10,000 jet problems \cite{WuShu2019} with $B_0 = \sqrt{20000}$. The corresponding results are shown in Figures~\ref{fig:Ex-Jet_2000_20000} and~\ref{fig:Ex-Jet_10000_20000}. 
		As the Mach number increases, the jet becomes increasingly narrow and elongated, with distinct flow features emerging at different magnetization levels. Our method accurately captures the Mach stem and associated discontinuities without generating any negative density or pressure values, thereby confirming the scheme’s robustness and its provably PP property. 
		The distributions of troubled cells shown in the second row illustrate that the shock indicator continues to detect discontinuities precisely, while avoiding false positives in smooth regions. Moreover, the axial symmetry of the solution is well preserved in both flow variables and troubled-cell patterns, further demonstrating the physical consistency and reliability of the proposed method.

		\begin{figure}[!htb]
			\centering	
			\begin{subfigure}{0.32\textwidth}
				\includegraphics[scale=0.32]{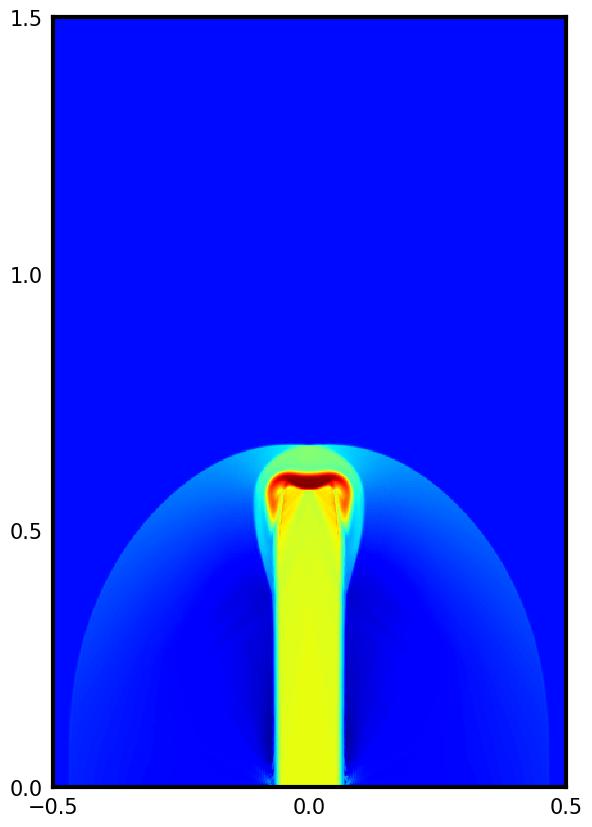}
			\end{subfigure}
			\hfill
			\begin{subfigure}{0.32\textwidth}
				\includegraphics[scale=0.32]{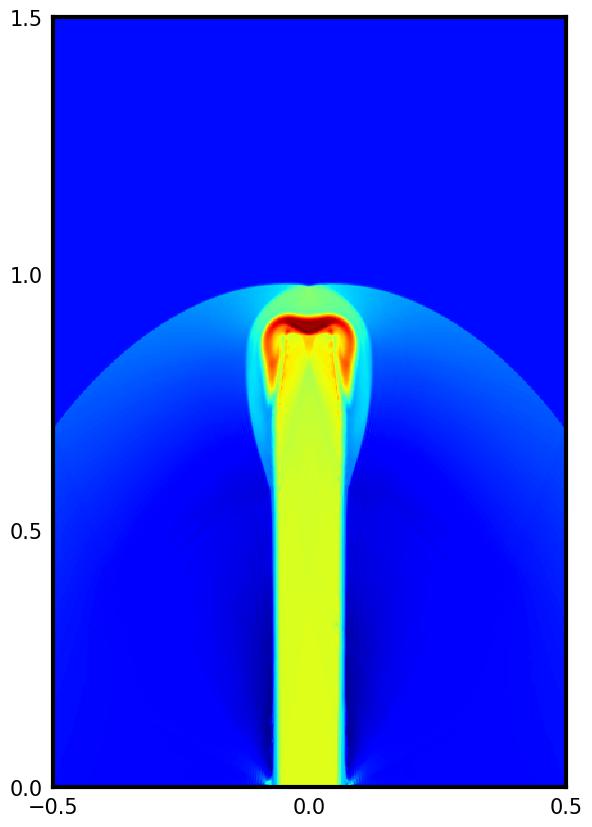}
			\end{subfigure}
			\hfill
			\begin{subfigure}{0.32\textwidth}
				\includegraphics[scale=0.32]{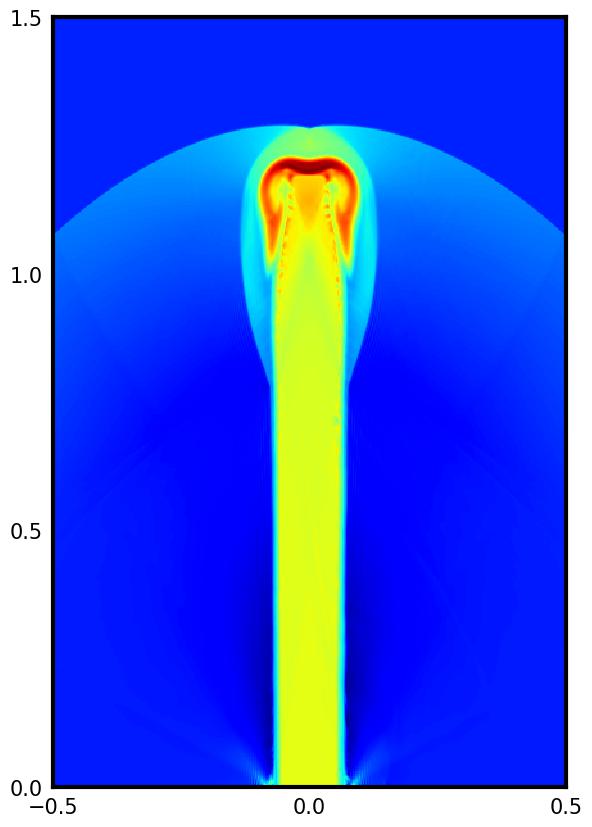}
			\end{subfigure}
			
			\begin{subfigure}{0.32\textwidth}
				\includegraphics[scale=0.32]{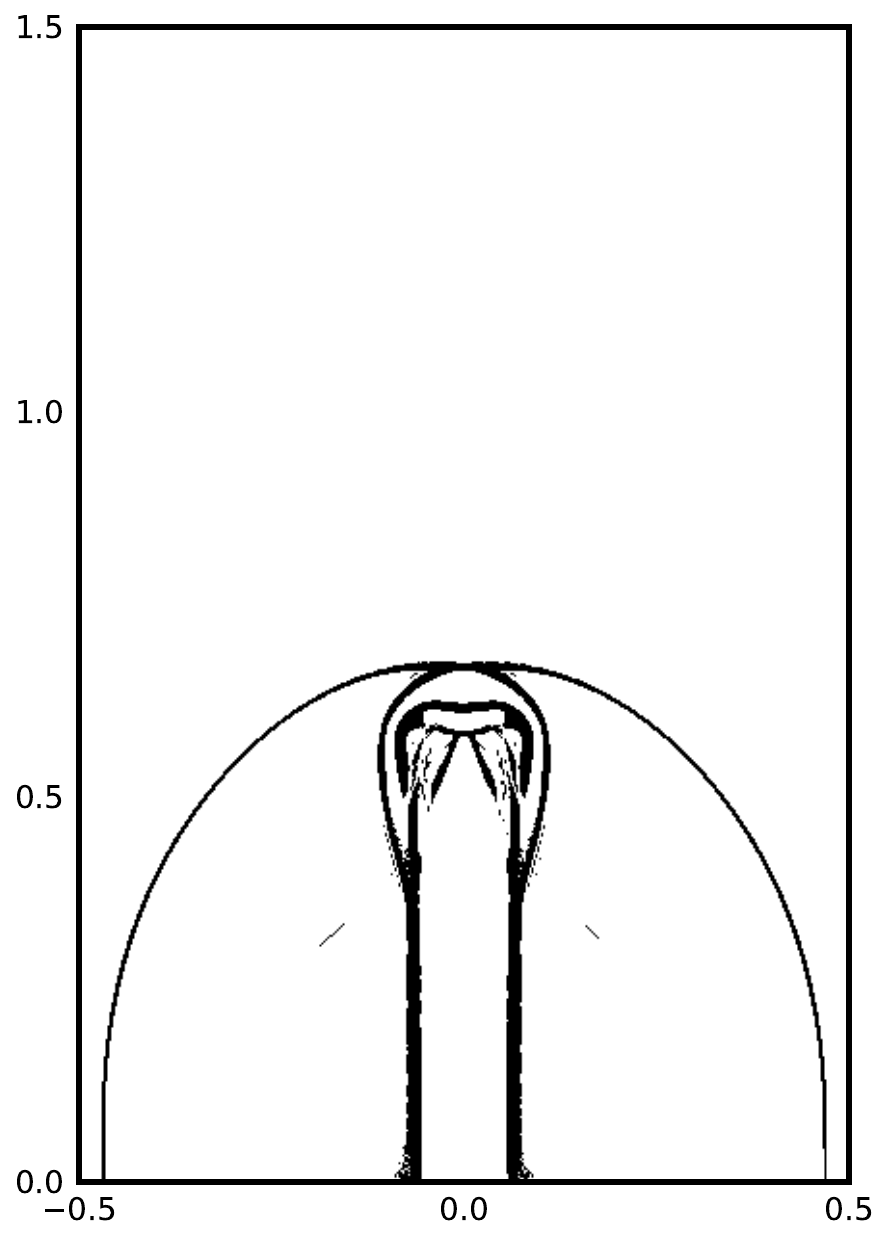}
			\end{subfigure}
			\hfill
			\begin{subfigure}{0.32\textwidth}
				\includegraphics[scale=0.32]{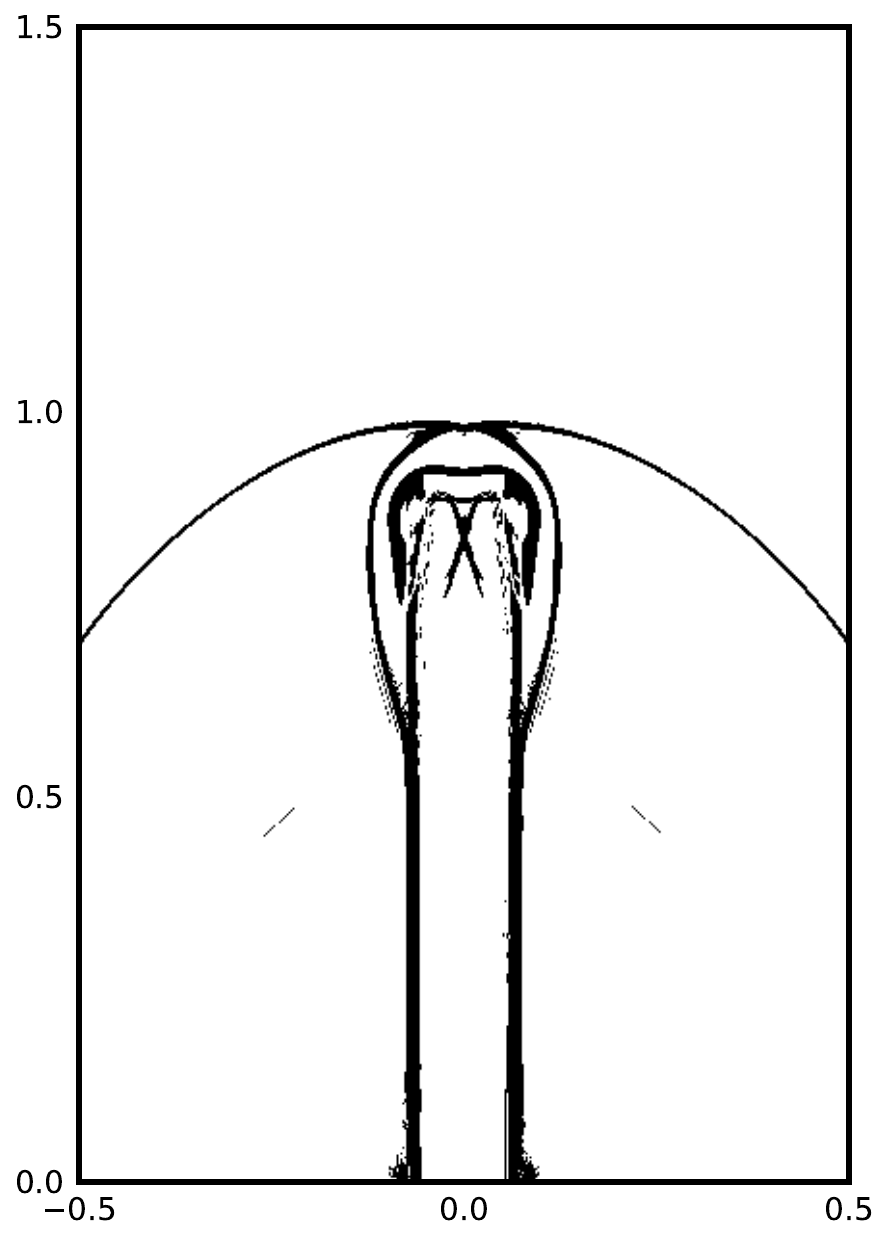}
			\end{subfigure}
			\hfill
			\begin{subfigure}{0.32\textwidth}
				\includegraphics[scale=0.32]{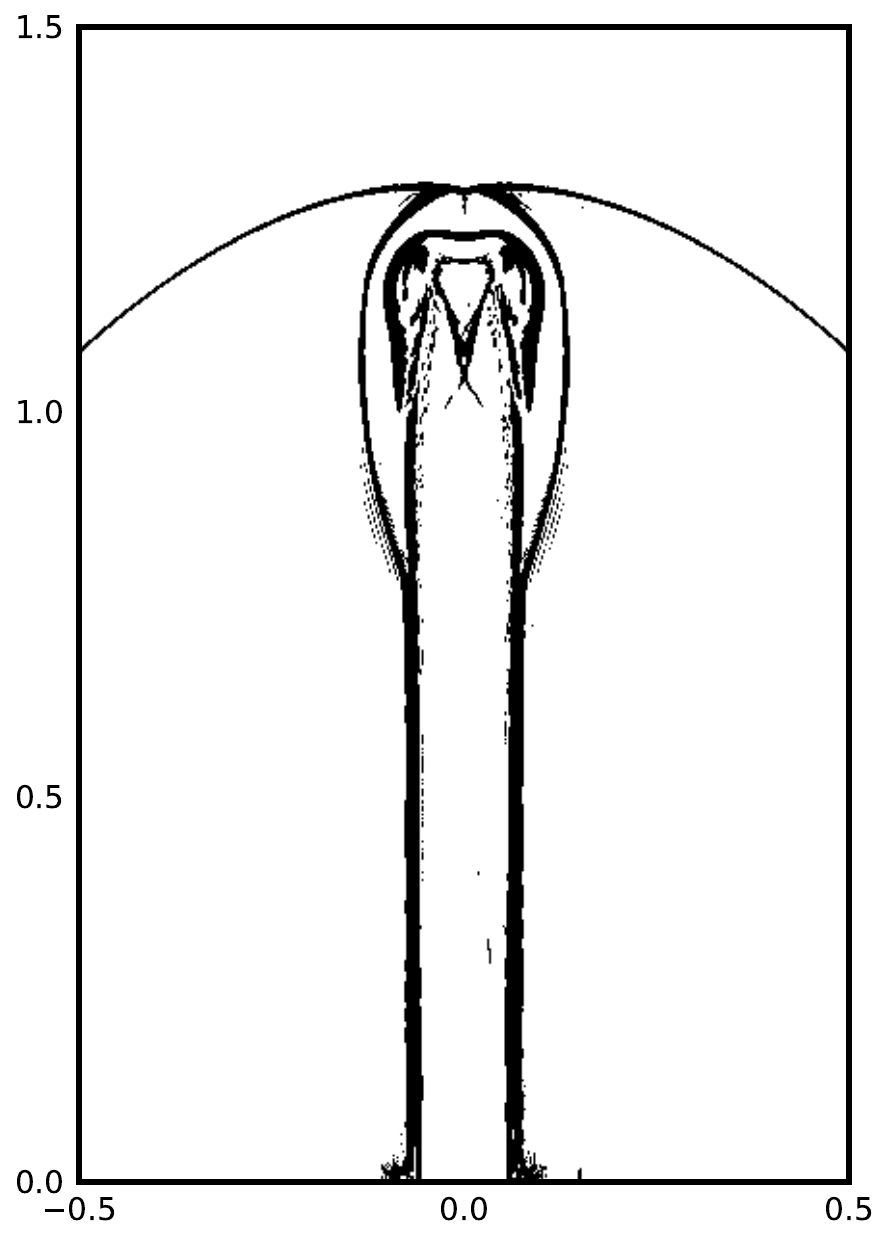}
			\end{subfigure}
			
			\caption{Mach 800 jet with $B_0 = \sqrt{20000}$: Density logarithm (top) and trouble cells(bottom) at $t = 0.001, 0.0015$, and $0.002$ (from left to right).
			}
			\label{fig:Ex-Jet_800_20000}
		\end{figure}

		\begin{figure}[!thb]
			\centering
			\begin{subfigure}{0.32\textwidth}
				\includegraphics[scale=0.32]{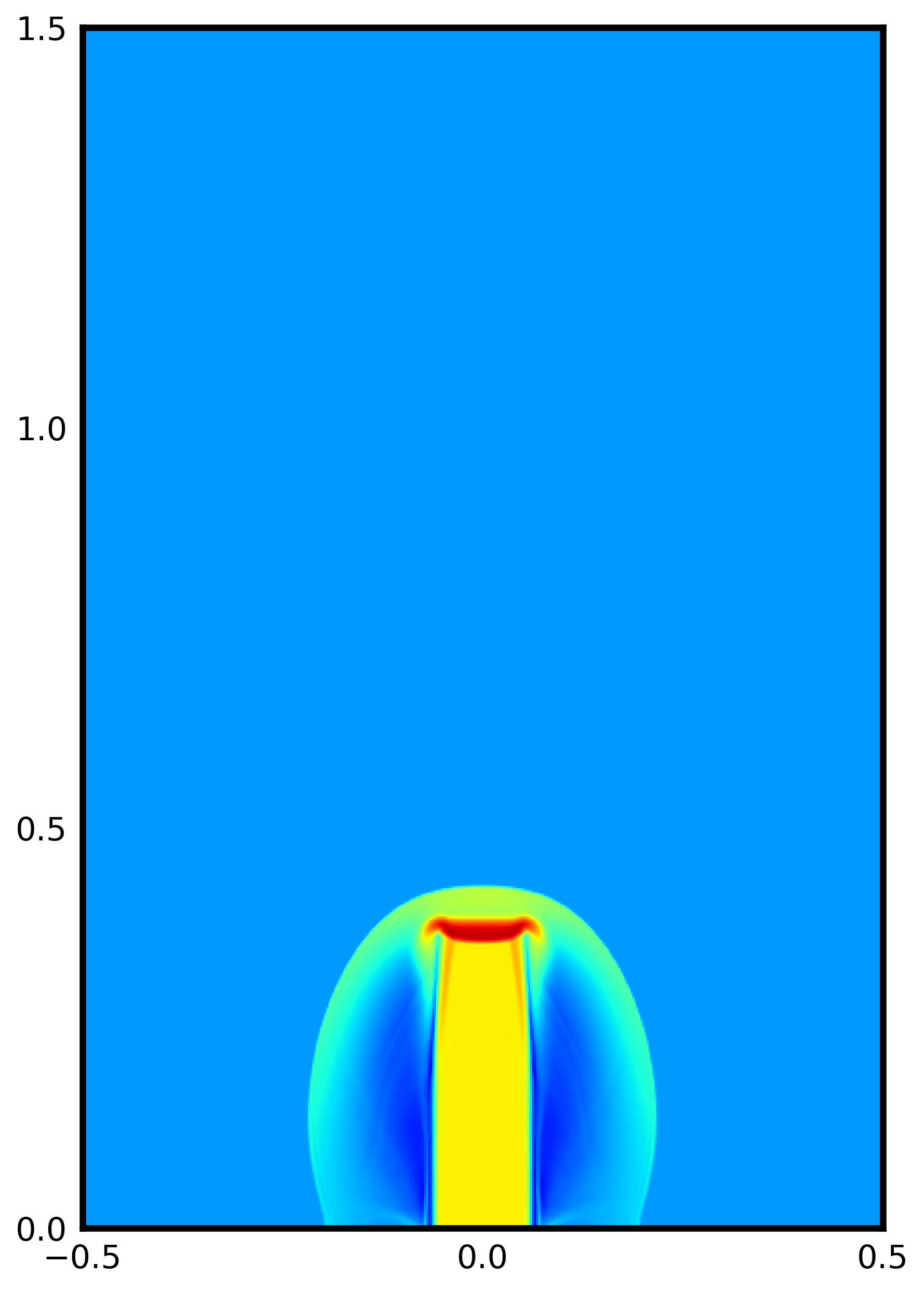}
			\end{subfigure}
			\hfill
			\begin{subfigure}{0.32\textwidth}
				\includegraphics[scale=0.32]{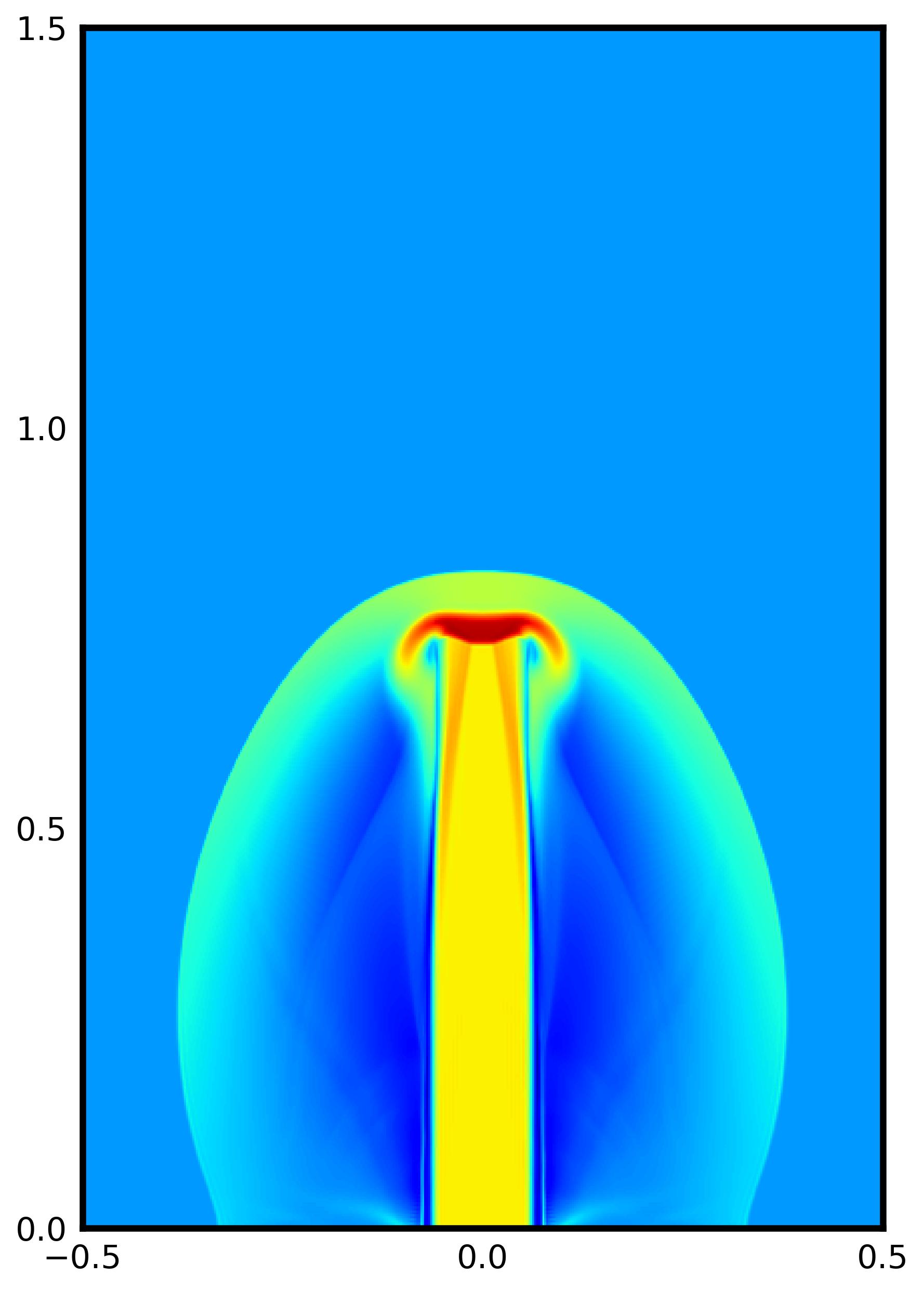}
			\end{subfigure}
			\hfill
			\begin{subfigure}{0.32\textwidth}
				\includegraphics[scale=0.32]{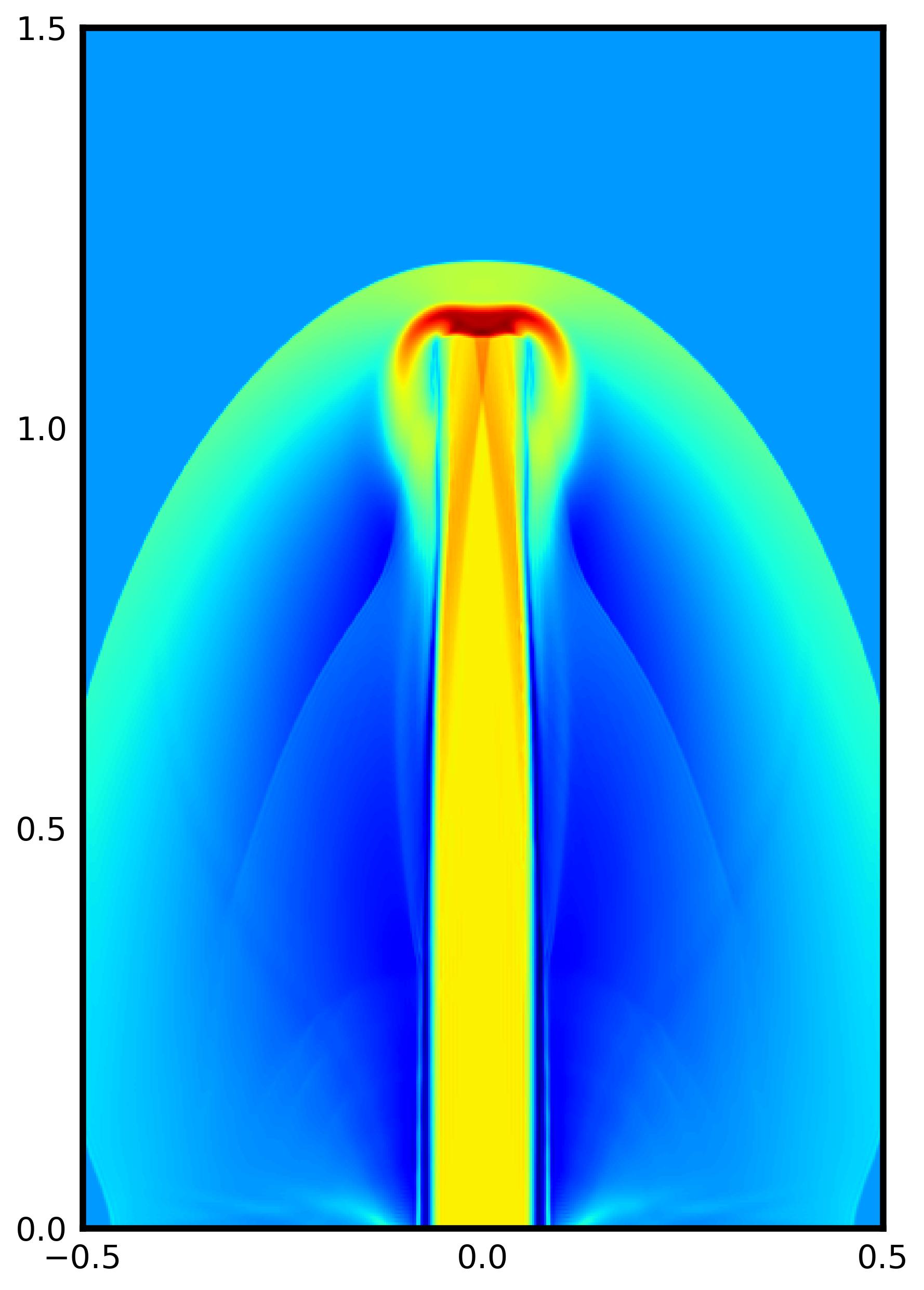}
			\end{subfigure}
			
			\begin{subfigure}{0.32\textwidth}
				\includegraphics[scale=0.32]{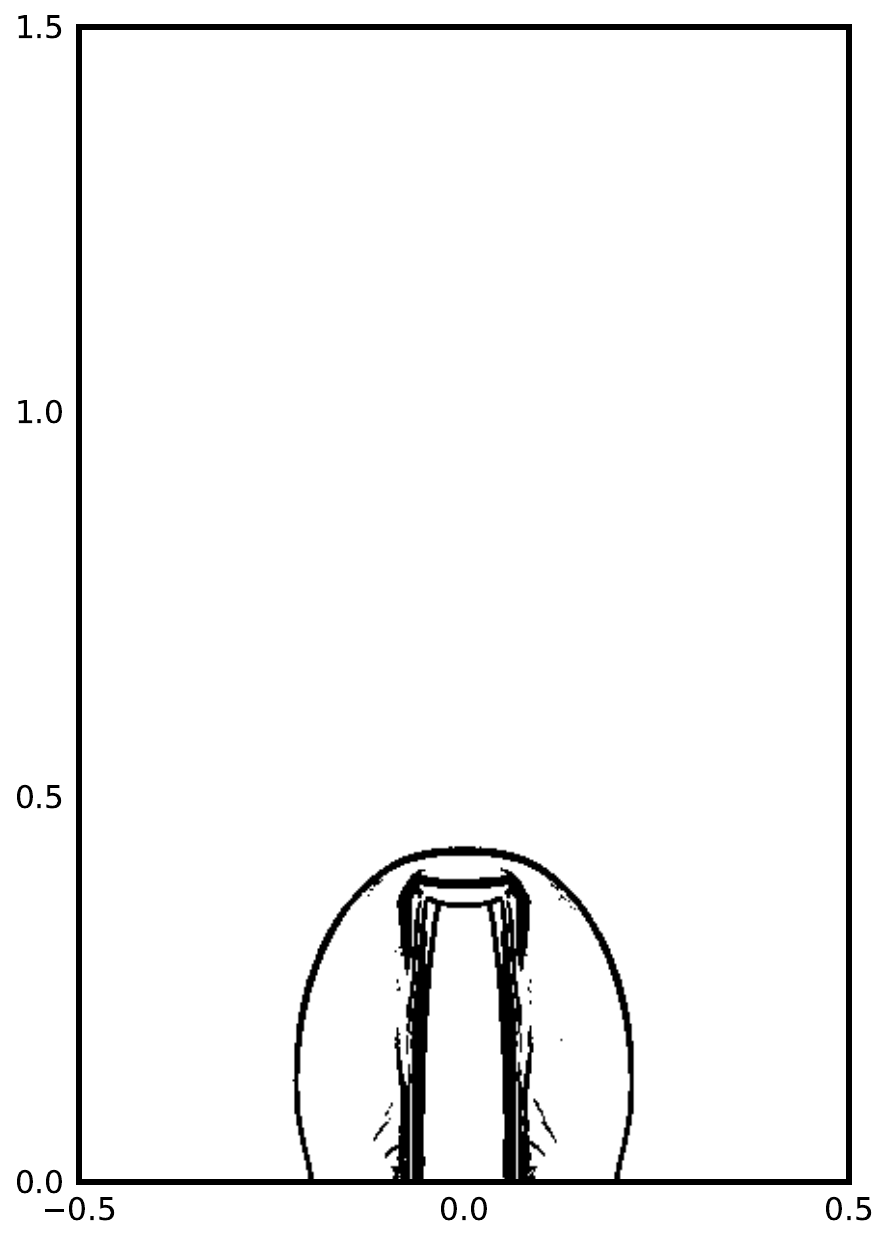} 
			\end{subfigure}
			\hfill
			\begin{subfigure}{0.32\textwidth}
				\includegraphics[scale=0.32]{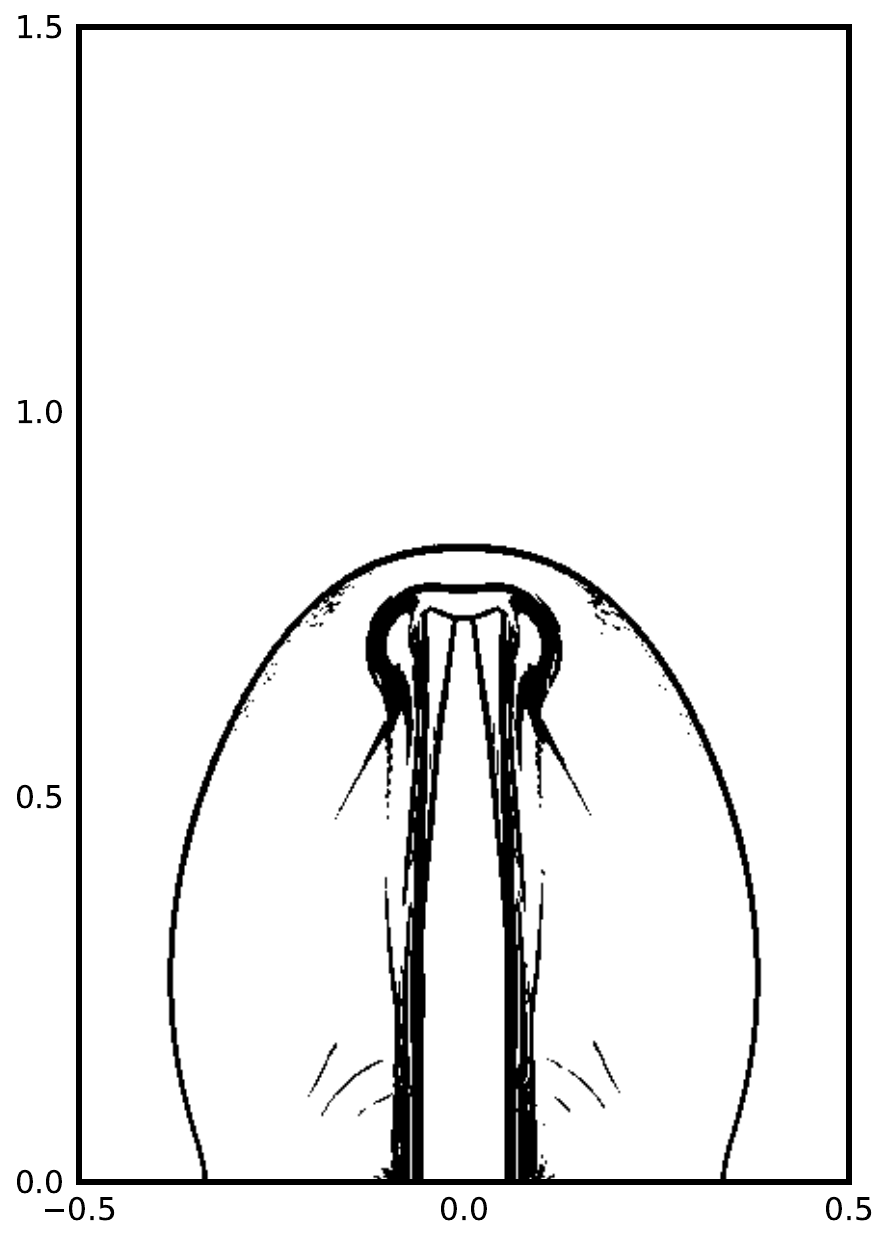}
			\end{subfigure}
			\hfill
			\begin{subfigure}{0.32\textwidth}
				\includegraphics[scale=0.32]{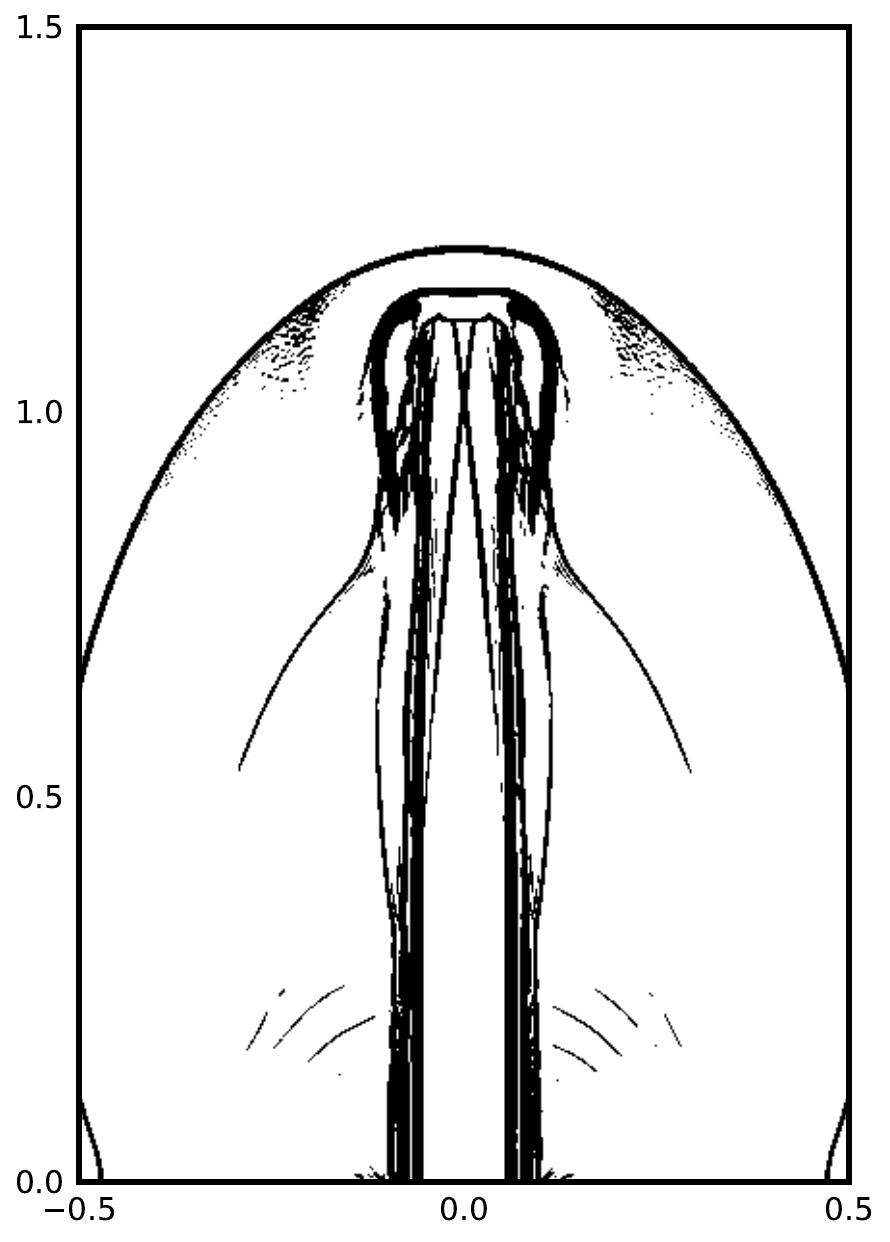}
			\end{subfigure}
			\caption{Mach 2000 jet problem with $B_0 = \sqrt{20000}$: Density logarithm(top) and trouble cells(bottom) at $t = 0.00025, 0.0005$, and $0.00075$ (from left to right).
			}
			\label{fig:Ex-Jet_2000_20000}
		\end{figure} 
		
		\begin{figure}[!thb]
			\centering
			\begin{subfigure}{0.32\textwidth}
				\includegraphics[scale=0.32]{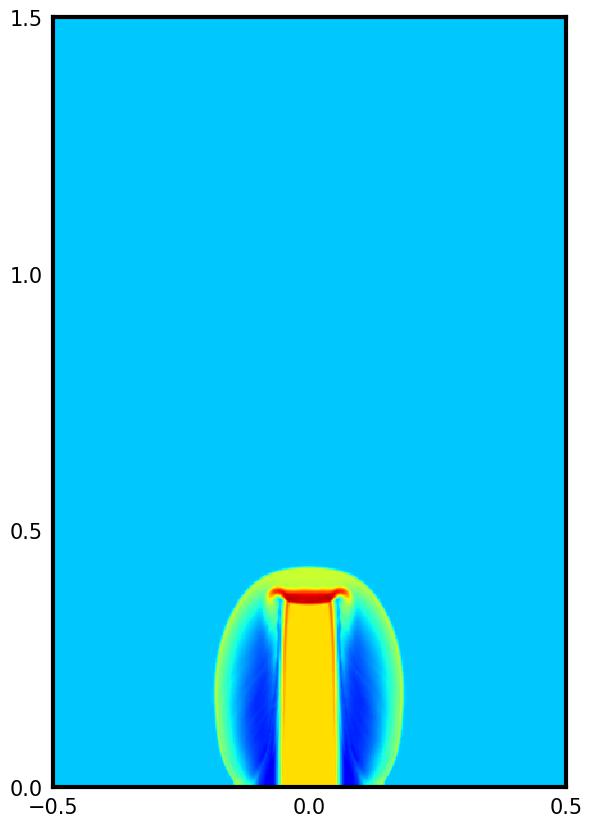}
			\end{subfigure}
			\hfill
			\begin{subfigure}{0.32\textwidth}
				\includegraphics[scale=0.32]{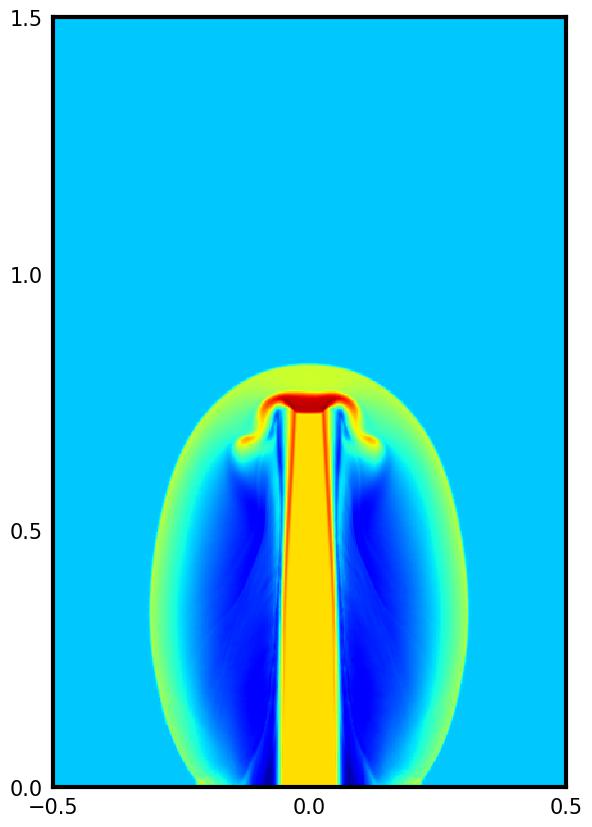}
			\end{subfigure}
			\hfill
			\begin{subfigure}{0.32\textwidth}
				\includegraphics[scale=0.32]{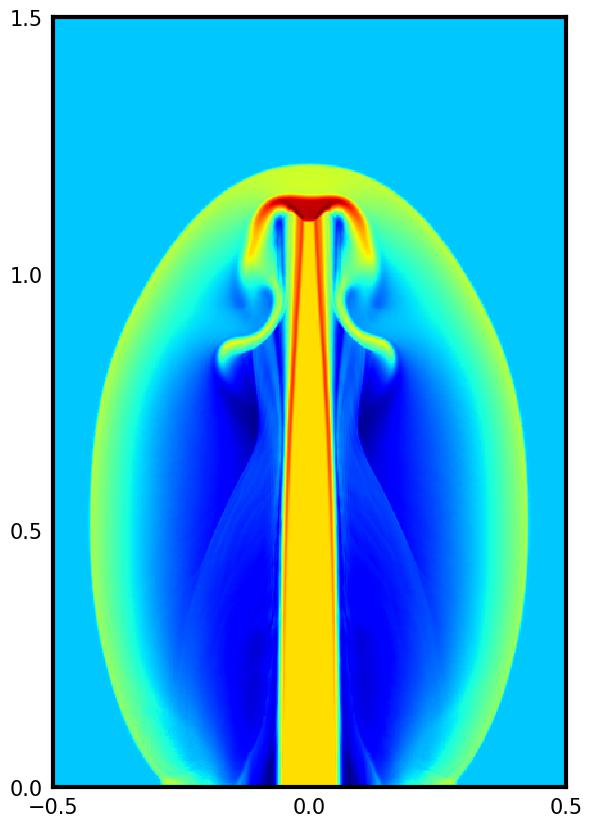}
			\end{subfigure}
			
			\begin{subfigure}{0.32\textwidth}
				\includegraphics[scale=0.32]{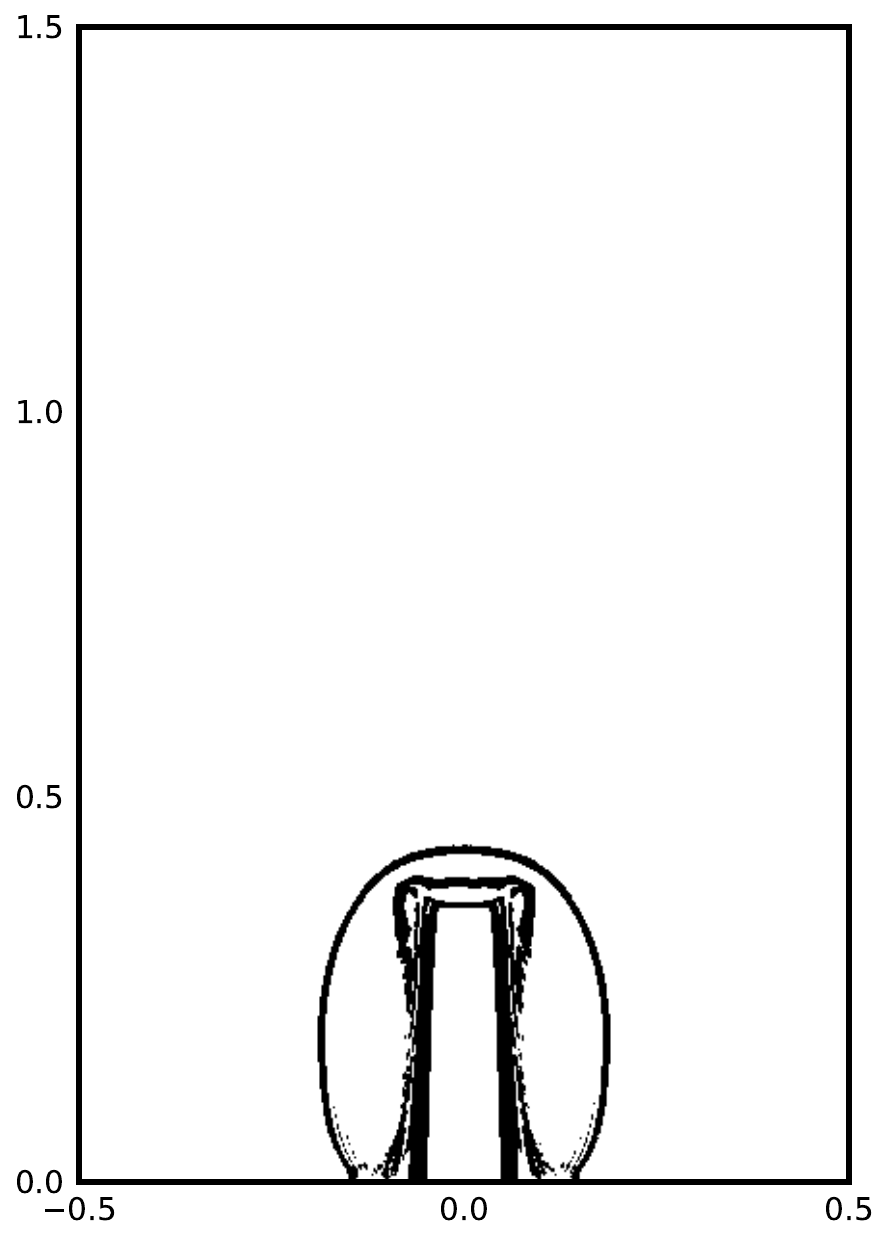}
			\end{subfigure}
			\hfill
			\begin{subfigure}{0.32\textwidth}
				\includegraphics[scale=0.32]{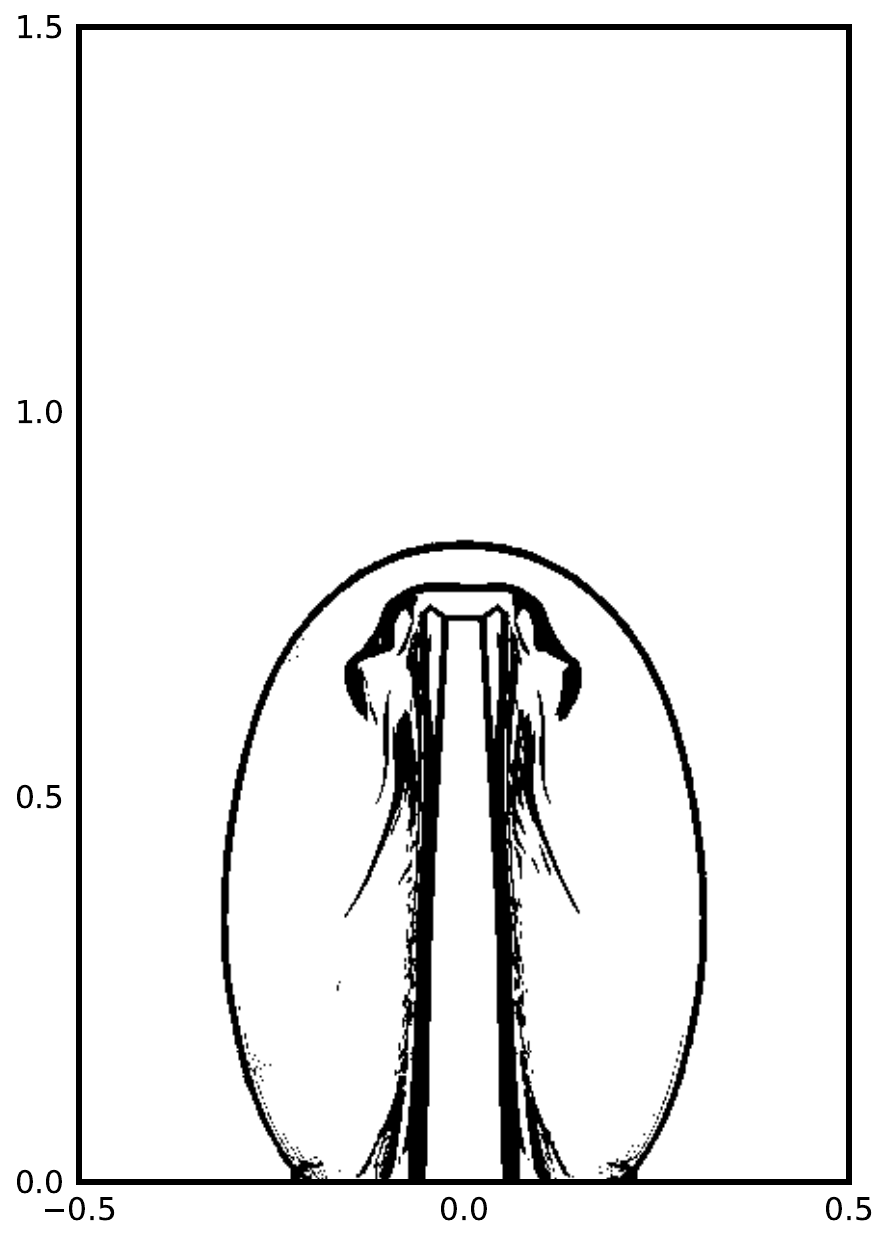}
			\end{subfigure}
			\hfill
			\begin{subfigure}{0.32\textwidth}
				\includegraphics[scale=0.32]{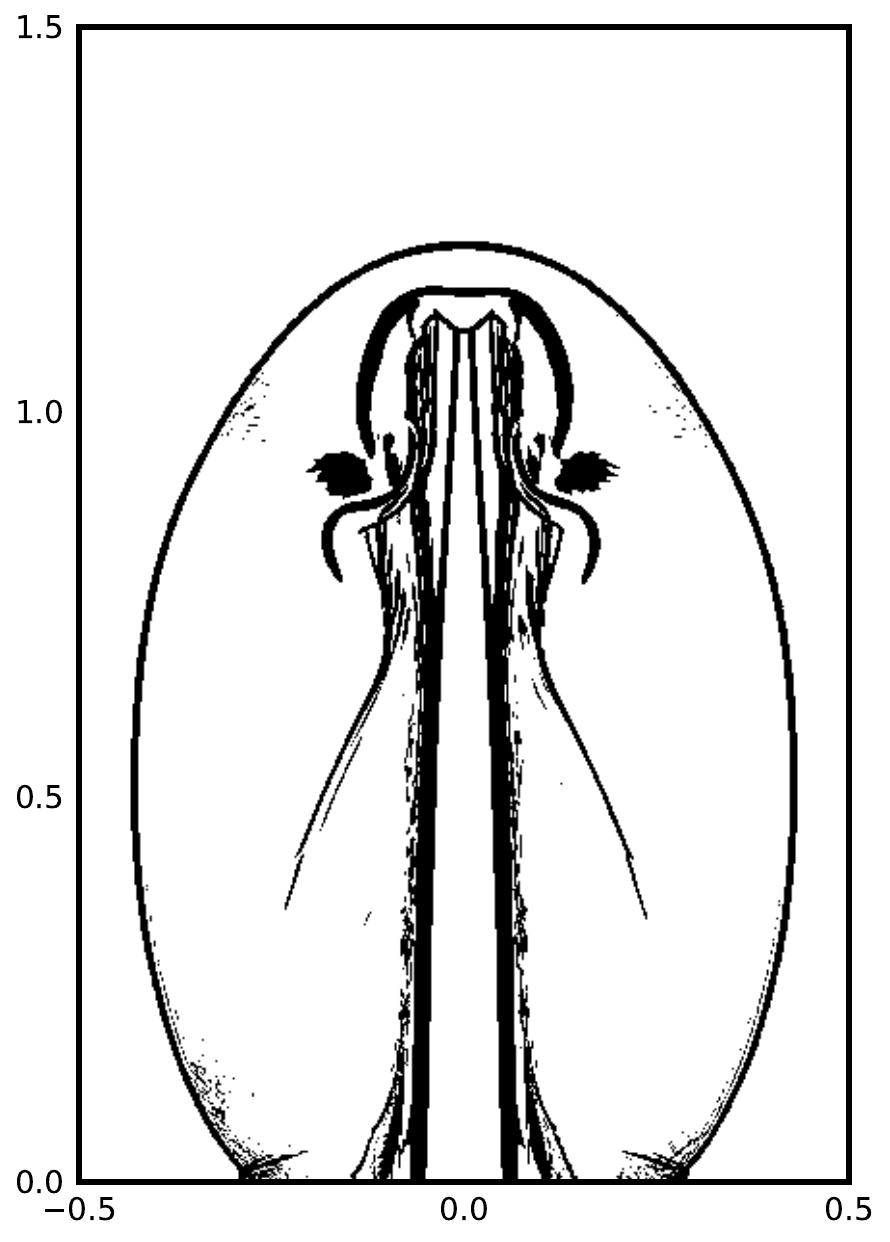}
			\end{subfigure}
			\caption{Mach $10000$ jet problem with $B_0 = \sqrt{20000}$: Density logarithm(top) and trouble cells(bottom) at $t = 0.00005, 0.0001$, and $0.00015$ (from left to right).
			}
			\label{fig:Ex-Jet_10000_20000}
		\end{figure} 		 
	\end{expl}

	\section{Conclusion}\label{conclusion}
	
	We have developed a compact, third-order structure-preserving PAMPA (Point-Average-Moment PolynomiAl-interpreted) scheme for the ideal magnetohydrodynamics (MHD) equations on Cartesian meshes that simultaneously enforces a discrete divergence-free (DDF) magnetic field and rigorously preserves positivity (PP) for both interface point values and cell averages. The method couples a limiter-free, automatically PP point-value update—enabled by a new nonconservative reformulation and a local DDF projection—with a provably PP finite-volume update for cell averages analyzed within the geometric quasi-linearization (GQL) framework. Together, these ingredients yield a coherent active-flux-type methodology with a small stencil and straightforward implementation. To the best of our knowledge, this is the first active-flux-type method for ideal MHD that is rigorously PP for both cell averages and interface point values while maintaining a DDF constraint throughout the evolution. 
	
	On the point-value side, our nonconservative reformulation guarantees that all updated point states remain in the admissible set without any PP limiter; the accompanying cell-local DDF projection maintains a discretely divergence-free magnetic field and is essential to the subsequent PP proof for cell averages. On the cell-average side, the update is proven PP under a mild a~priori positivity condition at a single cell-centered state and relies on four components: (i) a DDF constraint at interface point values, (ii) a PP limiter applied only to the cell-centered value, (iii) a PP numerical flux with properly estimated wave speeds, and (iv) a suitable discretization of the Godunov–Powell source term. In practice, we adopt a PP Lax–Friedrichs flux (while PP HLL fluxes are also admissible) with viscosity chosen according to theoretically established PP estimates. 
	
	To enhance robustness for strong shocks, we introduced a problem-independent, symmetry-preserving shocked-cell indicator based on a Lax-type entropy criterion that uses only two characteristic speeds computed from cell averages, and a convex oscillation elimination (COE) mechanism that blends limited point values and cell averages using a new inter-cell polynomial-difference norm. The indicator focuses limiting near genuine discontinuities and remains inactive in smooth regions, supporting both accuracy and efficiency.  
	
	Extensive numerical experiments—including very low plasma-beta blast waves and high–Mach-number jets—demonstrate that the scheme achieves high-order accuracy, sharply resolves MHD features, avoids nonphysical negative states, and robustly identifies troubled cells with symmetry-consistent patterns. These results validate the effectiveness of the proposed design under highly magnetized, extreme-flow conditions. 
	Finally, we note that although our presentation focuses on the two-dimensional case, the methodology and analysis naturally extend to three dimensions.

	\bibliographystyle{model1-num-names}%
	\bibliography{references_article}
	
\end{document}